\documentclass[11pt]{article}
\usepackage[utf8]{inputenc}
\usepackage[margin=1in]{geometry}
\usepackage[]{todonotes}
\usepackage{graphicx}
\usepackage{mathpazo}
\usepackage{hyperref}
\hypersetup{
  colorlinks = true,
  urlcolor = {blueGrotto},
  linkcolor = {royalBlue},
  citecolor = {navyBlueP}
}
\usepackage{xcolor}
\usepackage{bm}
\usepackage{amsmath,amssymb,amsthm,mathtools}
\usepackage{thm-restate,thmtools}
\usepackage{algorithm,algpseudocode}
\usepackage[backend=biber,style=alphabetic,natbib=true,maxbibnames=99]{biblatex}
\usepackage{comment}
\usepackage{url}

\usepackage[capitalise,nameinlink]{cleveref}
\crefname{lemma}{Lemma}{Lemmas}
\crefname{fact}{Fact}{Facts}
\crefname{theorem}{Theorem}{Theorems}
\crefname{corollary}{Corollary}{Corollaries}
\crefname{claim}{Claim}{Claims}
\crefname{example}{Example}{Examples}
\crefname{problem}{Problem}{Problems}
\crefname{setting}{Setting}{Settings}
\crefname{definition}{Definition}{Definitions}
\crefname{assumption}{Assumption}{Assumptions}
\crefname{subsection}{Subsection}{Subsections}
\crefname{section}{Section}{Sections}

\newtheorem{definition}{Definition}
\newtheorem{assumption}{Assumption}
\newtheorem{lemma}{Lemma}
\newtheorem{theorem}{Theorem}
\newtheorem{corollary}{Corollary}
\newtheorem{claim}{Claim}
\newtheorem{fact}{Fact}
\newtheorem{inftheorem}{Informal Theorem}
\newtheorem{infassumption}{Informal Assumption}

\newcommand{\lprp}[1]{\ensuremath{\left({#1}\right)}}
\newcommand{\lbrb}[1]{\ensuremath{\left\{{#1}\right\}}}
\newcommand{\lsrs}[1]{\ensuremath{\left[{#1}\right]}}

\newcommand{\slice}{C}

\renewcommand{\vec}[1]{\bm{#1}}
\newcommand{\lr}[1]{\ensuremath{\left({#1}\right)}}
\newcommand{\sqlr}[1]{\ensuremath{\left[ {#1} \right] }}
\newcommand{\blr}[1]{\ensuremath{\left\{ {#1} \right\} }}

\newcommand{\ddt}[2]{\ensuremath{\frac{\partial^2}{\partial{#1} \partial{#2}}}}
\newcommand{\E}{\mathbb{E}}
\renewcommand{\P}{\mathbb{P}}
\renewcommand{\vec}[1]{\boldsymbol{#1}}
\newcommand{\eps}{\varepsilon}
\newcommand{\R}{\mathbb{R}}
\newcommand{\poly}{\mathrm{poly}}
\DeclareMathOperator*{\argmax}{arg\,max}
\DeclareMathOperator*{\argmin}{arg\,min}
\DeclarePairedDelimiterX{\abs}[1]{\lvert}{\rvert}{#1}

\newcommand{\mc}[1]{\mathcal{#1}}
\newcommand{\mb}[1]{\mathbb{#1}}
\newcommand{\proj}[1]{\mathcal{P}_{#1}}
\newcommand{\pproj}[1]{\mathcal{P}^\perp_{#1}}
\newcommand{\wt}[1]{\widetilde{#1}}

\newcommand{\norm}[1]{\left\|#1\right\|}

\definecolor{niceRed}{RGB}{190,38,38}
\definecolor{blueGrotto}{HTML}{059DC0}
\definecolor{royalBlue}{HTML}{057DCD}
\definecolor{navyBlueP}{HTML}{0B579C}
\definecolor{limeGreen}{HTML}{81B622}

\DeclareMathOperator*{\Exp}{{\mathbb{E}}}

\title{What Makes A Good Fisherman? \\ Linear Regression under Self-Selection Bias}

\author{
\begin{tabular}{cc}
& \\
\textbf{Yeshwanth Cherapanamjeri} & \textbf{Constantinos Daskalakis}\\
\small{University of California Berkeley} & \small{Massachusetts Institute of Technology}\\
\small{\texttt{yeshwanth@berkeley.edu }} & \small{\texttt{costis@csail.mit.edu}}\\
& \\
\textbf{Andrew Ilyas} & \textbf{Manolis Zampetakis}\\
\small{Massachusetts Institute of Technology} & \small{University of California Berkeley}\\
\small{\texttt{ailyas@mit.edu}} & \small{\texttt{mzampet@berkeley.edu}}\\
& \\
\end{tabular}
}

\date{}
\begin{document}
    \maketitle

    \begin{abstract}
    In the classical setting of self-selection, the goal is to learn $k$ models,
    simultaneously from observations $(\vec{x}^{(i)}, y^{(i)})$ where $y^{(i)}$ is the
    output of one of $k$ underlying models on input $\vec{x}^{(i)}$. In contrast to
    mixture models, where we observe the output of a randomly selected model, and therefore 
    the selection of which model is observed is {\em exogenous}, in self-selection models
    which model is observed depends on the realized outputs of the underlying models themselves, as determined by some
    known selection criterion (e.g.~we might observe the highest output, the
    smallest output, or the median output of the $k$ models), and is thus {\em endogenous}. In known-index
    self-selection, the identity of the observed model output is observable; in
    unknown-index self-selection, it is not. Self-selection has a long history in
    Econometrics (going back to the works of \citet{roy1951some},
    \citet{gronau1974wage}, \citet{lewis1974comments},
    \citet{heckman1974shadow} and others) and many applications in various theoretical and applied fields,
    including treatment effect estimation, imitation learning, learning from
    strategically reported data, and learning from markets at disequilibrium.  
            
    In this work, we present the first computationally and statistically efficient
    estimation algorithms for the most standard setting of this problem where the
    models are linear. In the known-index case, we require $\poly(1/\eps, k,
    d)$ sample and time complexity to estimate all model parameters to accuracy
    $\eps$ in $d$ dimensions, and can accommodate quite general selection
    criteria. In the more challenging unknown-index case, even the identifiability
    of the linear models (from infinitely many samples) was not known. We show three
    results in this case for the commonly studied $\max$ self-selection criterion:
    (1) we show that the linear models are indeed identifiable, (2) for general $k$
    we provide an algorithm with $\poly(d)\cdot \exp(\poly(k))$ sample and time
    complexity to estimate the regression parameters up to error $1/\poly(k)$,
    and (3) for $k = 2$ we provide an algorithm for any error $\eps$ and
    $\poly(d, 1/\eps)$ sample and time complexity.

        \end{abstract}
    \thispagestyle{empty}

\clearpage
\setcounter{page}{1}
    \section{Introduction}
    \label{sec:intro}
    To introduce our problem we present the following story adapted from the seminal work of~\citet{roy1951some}. In a small village, two mutually exclusive occupations are available:
hunting and fishing. An analyst visits the village with a simple question:
\begin{center}
    {\em What makes a good fisher and what makes a good hunter?}
\end{center}
More precisely, the analyst wishes to construct a
statistical model mapping villagers' features (e.g., their height and weight) to
their proficiency at hunting and fishing (as measured, e.g., by their income). 
To accomplish this, the analyst might collect a random sample of hunters and 
fishers from the village, record their relevant features as well as their
income, and then use this data to estimate the parameters of two linear
models---one for each occupation.\footnote{For the purposes of this discussion, we assume an abundance of game and fish, and that everything is exported at fixed prices so that the income from each occupation is not affected by how many villagers exercise each occupation.} For this purpose, it is natural for the analyst to use the OLS estimator on all the hunter data to estimate the hunter model, and the OLS estimator on all the fisher data to estimate the fisher model.

It turns out, however, that even 
with a perfectly
representative sample of villagers, 
the resulting linear fits will likely be {\em biased}. 
Indeed, if 
the villagers are 
rational agents, 
they will 
choose their occupations based on which one  generates more income for them:
those who are better at hunting than fishing (in terms of earnings) will opt
to hunt, and vice-versa. As a result, the analyst will
never observe, e.g.,~the hunting earnings of an individual who is better at fishing than hunting, since that individual will choose to fish.
This induces bias in the observed hunting and fishing datasets which makes 
the outputs of naive estimators on these datasets biased as well: Figure~\ref{fig:headline_fig} illustrates this
effect in one dimension.
In fact, this bias arises even in the simpler case where earnings
from both occupations
are normally distributed and fully independent of an individual's
features and of one another, as discussed in the  work of~\citet{roy1951some}.\footnote{
In particular, suppose that 
hunting earnings were low-mean and low-variance, 
while fishing earnings were higher-mean but higher-variance. 
In this case, if we use only fishing data to estimate villagers' expected fishing
earnings, we would get an over-estimate, as bad fishers with low earnings
would almost certainly turn to hunting.}

\begin{figure}[hb]
    \centering
    \includegraphics[width=.75\textwidth]{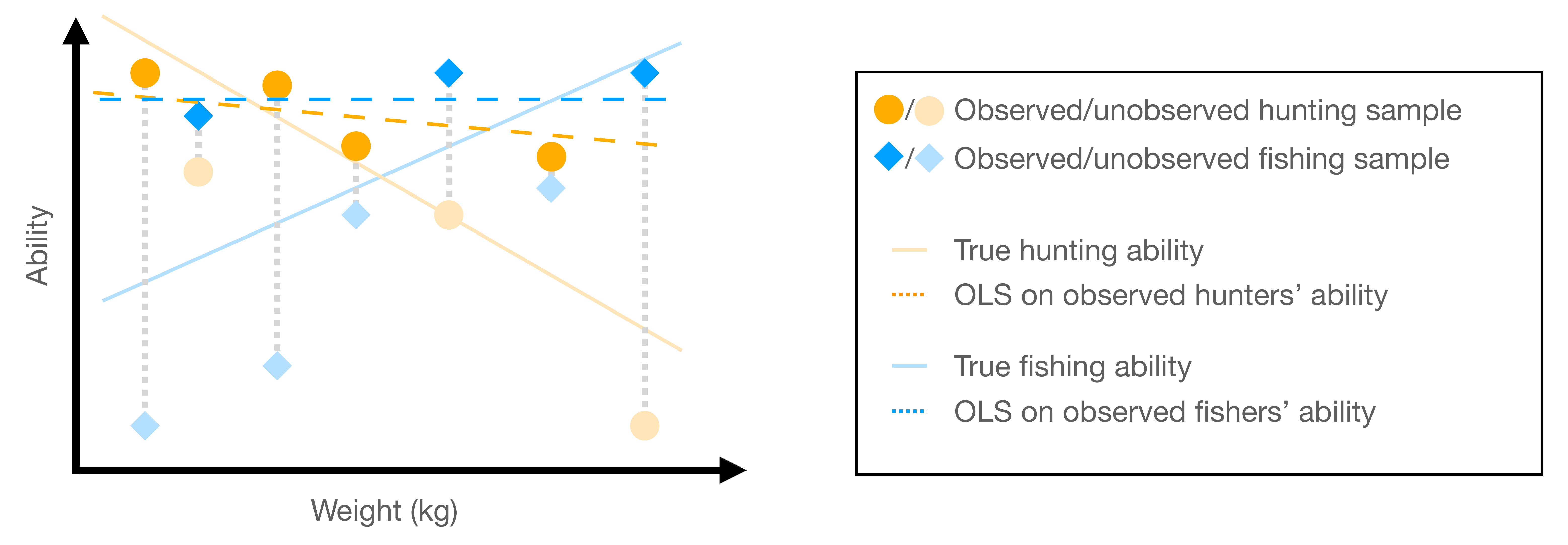}
    \caption{An illustration of self-selection bias in one-dimensional linear regression.
    In a fictional population, the orange points represent individuals' hunting 
    ability, while the corresponding blue points represent the same individuals'
    fishing ability. Each individual in the population {\em self-selects} and 
    only performs the activity at which they are best. As a result, data analysis
    performed based on only observed data (i.e., the bright points in the above 
    graph) leads to biased inference (compare the dotted estimated lines to the 
    solid ground-truth lines).}
    \label{fig:headline_fig}
\end{figure}

\paragraph{Outcome self-selection.} The above example (due to \citet{roy1951some}) is just one illustration of bias due to
{\em self-selection}, wherein the outcome variable that we
observe is selected, often due to strategic considerations from a set of potential
outcomes. In this setting, 
we
observe $n$ feature vectors $\{\vec{x}^{(i)}\}_{i=1}^n$, each
accompanied by a label $y^{(i)}$ that is the output of {\em one} out of $k$
underlying models: 
\[
    y^{(i)} \in \left\{f_{\vec{w}_1}\left(\vec{x}^{(i)},\eps_{1}^{(i)}\right),\ldots,f_{\vec{w}_k}\left(\vec{x}^{(i)},\eps_{k}^{(i)}\right)\right\},
\]
where $\vec{\eps}^{(i)}=(\eps_{1}^{(i)},\ldots,\eps_{k}^{(i)})$ is a noise vector sampled 
independently across different observations, and $f_{\vec{w}_1},\ldots,f_{\vec{w}_k}$ are $k$ unknown models from some class. The model whose output $y^{(i)}$ we observe
is determined by some known function $S: \mathbb{R}^k \rightarrow \{1,\ldots,k\}$ called the {\em self-selection criterion}---in
the village example, the self-selection criterion was the
maximum, i.e., $\smash{y^{(i)} = \max_{j \in [k]}
\{f_{\vec{w}_j}(\vec{x}^{(i)},\eps^{(i)}_{j})\}}$.

We will consider two instantiations of the self-selection problem that differ
in the amount of information available to the statistician.
In the easier version of the problem---the ``known-index self-selection 
model''---the statistician observes the identity $j_*^{(i)}$ of the
model that produced each output $y^{(i)}$, in addition to observing $y^{(i)}$ itself. 
This setting captures the hunting/fishing example of
\citet{roy1951some}, where we observe {\em both} the earnings and the occupation of each villager.
In the harder version of the problem---the
``unknown-index self-selection model''---the statistician does not observe the
identity of the  model that produced each output $y^{(i)}$.

\paragraph{Applications and prior work.} It turns out that the above formulation 
captures a wide variety of settings wherein observed data is the output of a strategic or systematic selection process operating on some underlying data, the entirety of which is never observed. We discuss some examples:
\begin{enumerate}
    \item {\bf Imitation learning:} Consider the problem of learning an optimal
    policy in some contextual bandit setting wherein we observe the arms (e.g.~treatments) pulled by an expert (e.g.~doctor) in different contexts (e.g.~patients). Modeling the reward (e.g.~efficacy) from each arm $j$ as an unknown function $f_{\vec{w}_j}(\vec{x},\eps_{j})$ of the context $\vec{x}$ and additional randomness $\eps_{j}$ that the expert might observe but we do not, we assume that the expert  selects the arm $j$ with the highest reward $\max_j \{ f_{\vec{w}_j}(\vec{x},\eps_{j}) \}$. Our goal is to learn the underlying models $\vec{w}_1,\ldots,\vec{w}_k$ of the arms by observing the expert make decisions in different contexts. This scenario is an instantiation of the known-index self-selection model with the maximum selection criterion.
    
    \item {\bf Learning from strategically reported data:} A widely studied setting featuring self-selected data is one wherein agents are incentivized to  strategically choose which data to report. This is a standard challenge in Econometrics, which has recently received increased attention in 
    machine learning literature due to the impact of learning-mediated decisions  in various  contexts; see e.g.~\citep{strategicclassification,strategicwithhelddata,testoptimalpolicies} and their references. A common example is the reporting of standardized test scores in college admissions, where applicants have a variety of standardized tests available
    to them, and are 
    only required to report a chosen subset of them. 
    In a concrete setting of two tests, $A$ and $B$, let $S_A(\vec{x},
    \eps_A)$  and $S_B(\vec{x}, \eps_B)$ 
    denote the scores of an applicant with features $\vec{x}$ on each test.
    Upon completing both tests and receiving their scores $s_A$ and $s_B$, however, an applicant can compute their
    conditional probabilities of getting accepted based on each score,
    $\mathbb{P}(\text{accept}|s_A)$ and $\mathbb{P}(\text{accept}|s_B)$, and only report the score
    that is more likely to result in acceptance.
    Estimating parameters of the models determining the scores as a function of a student's features 
    now corresponds to a known-index self-selection problem with a complicated
    selection rule (albeit one that is handled by our results in~Section~\ref{sec:known}).
    \item {\bf Learning from market data: } Markets are often in disequilibrium. Following \citet{fair1972methods}, consider a model of
    the housing market, wherein there is a supply  
    $S(\vec{x},\eps_S) = \vec{w}_S^{\rm T}\vec{x} + \eps_{S}$ and a demand 
    $D(\vec{x},\eps_D) = \vec{w}_D^{\rm T}\vec{x} + \eps_{D}$ of houses with features $\vec{x}$, but the market is in 
    disequilibrium and supply does not equal demand. So the quantity transacted is 
    $Q(\vec{x}) = \min\{S(\vec{x},\eps_S),D(\vec{x},\eps_D)\}$, where 
    $(\eps_S,\eps_D)$ are random shocks. This example can be captured by the general 
    model discussed above by setting $k=2$, considering linear models, and taking the selection 
    criterion to be the minimum selection criterion. Moreover, this is an instance of
    the unknown-index model as we do not observe whether the disequilibrium is caused
    by lack of supply or lack of demand.
    \item {\bf Learning from auction data:} 
    \citet{athey2002identification} and a large body of literature in Econometrics
    consider the problem of learning bid (and valuation) distributions from auction data with partial observability, wherein only the winner of each auction and the price they paid are observed. Consider such observations in repeated first-price auctions.
    We can cast this problem as an instance of known-index self-selection where $k$ is the number of bidders,
    the models are parametric/non-parametric bid distributions 
    mapping randomness $\eps_j$ to a bid, and the 
    selection rule is the maximum function. A body of work in the literature has provided estimation and identification results in this setting~\citep{athey2007nonparametric}, including recent work of \citet{cherapanamjeriestimation} which demonstrates polynomial-time algorithms
    for estimating the bid distributions non-parametrically to within 
    Kolmogorov distance $\eps$.
\end{enumerate}

 As suggested by the diversity of examples above, models with self-selection bias %
have received extensive study  due to their
numerous applications.
These include
studies of participation in the labor
force~\cite{heckman1974shadow,heckman1979sample,nelson1977censored,cogan2014labor,hanoch2014hours,hanoch2014multivariate},
retirement decisions~\cite{gordon1980market},  returns to
education~\cite{griliches1978missing,kenny1979returns,willis1979education},
effects of unions on wages~\cite{lee1978unionism,abowd1982job}, migration and
income~\cite{nakosteen1980migration,borjas1987self}, physician and lawyer
behavior~\cite{poirier198111,weisbrod1983nonprofit}, tenure choice and the
demand for
housing~\cite{lee1978estimation,rosen1979housing,king1980econometric},
identification of auction models under partial
observability~\cite{guerre2000optimal,athey2002identification,athey2007nonparametric},
and more;
see~\cite{maddala1986limited,cameron2005microeconometrics,brooks2019introductory}
for textbook introductions to this field and further applications. At a high
level, the reason why self-selection bias is so prevalent  is
that in many practical scenarios the observed labels are the outcomes of some
selection procedure that looks at an underlying ``complete sample'' to select
which part of it will be revealed. 

This {\em endogeneity} in selecting which of the $k$ models is observed in each sample, i.e.~the
dependence of this selection on the realized output of each model,
makes the
estimation of the underlying models challenging {\em even in the known-index case}. This stands in sharp contrast to mixture models where the selection of which model is observed is random, and thus {\em exogenous}, and as a result the estimation is straight-forward in the known-index case.

Since the work of Roy~\cite{roy1951some}, and despite the long history and use of self-selection
models, there only exist results proving that unbiased estimates can be recovered in the asymptotic sample regime in certain simple variants of the self-selection problem~\cite{heckman1974shadow,berndt1974estimation,goldfelfd1975estimation,lee1978estimation},
but computationally and statistically efficient algorithms are lacking, even
in very simple cases of the problem such as the housing market disequilibrium
model described above, even in the known-index setting. 
More broadly, self-selection models fall under the literature of regression with missingness in the outcomes, where the missingness is not at random; see e.g.~\cite{rotnitzky1995semiparametric,rotnitzky1998semiparametric,tchetgen2018discrete} and their references. However, the identification structures considered in this literature do not apply to our setting and/or the results are asymptotic.
We discuss existing approaches for self-selection and related models more extensively in Section~\ref{sec:intro:related}.

\subsection{Our results} 
\label{sec:intro:setup}
In this work, we focus on the simple yet prevalent case where the potential
outcomes are linear in the collected features (as in the housing market
disequlibrium example), with a residual (random) error term that is uncorrelated
across potential outcomes.
More precisely, 
we consider a setting involving $n$ individuals
(observations) and $k$ potential outcomes (models). Each individual $i \in [n]$ has a feature
vector $\vec{x}^{(i)} \in \R^d$ and each model $j \in [k]$ has a vector of regression
parameters $\vec{w}^*_j \in \R^d$. Then, each individual with feature vector
$\vec{x}$ recieves a label for each outcome $j$ equalling 
$\smash{y_{j} = \vec{w}_j^{* \top} \vec{x} + \eps_{j}}$,
where $\vec{\eps}$ is assumed to be an standard multivariate Gaussian random
variable. Our setting will involve self-selection bias arising from the fact
that every individual will only choose to reveal one of their labels, 
$j_* \in \{1, \ldots, k\}$, as determined by some function of
that individual's full set of labels $y_{1}, \ldots, y_{k}$.
Our goal in this paper is to estimate the $k$ parameter vectors 
$\vec{w}^*_1$, $\dots$, $\vec{w}^*_k$ under both the known-index and
unknown-index observational models.

We begin with our results in the known-index setting, which we formally define
in Definition~\ref{defn:index_observed}: 
\begin{description}
  \item[Known-index Setting:] We observe $n$ samples 
    $(\vec{x}^{(i)}, y^{(i)}, j_*^{(i)})$, where 
    $j_*^{(i)} = S(y_{1}^{(i)}, \dots, y_{k}^{(i)})$, and 
    $y^{(i)} = y_{j_*^{(i)}}^{(i)}$ for
    some known self-selection rule
    $S : \mathbb{R}^k \rightarrow \{1, \ldots, k\}$.
\end{description}
Here, we can estimate the unknown parameter vectors $\vec{w}^*_1$, $\dots$,
$\vec{w}^*_k$ to arbitrary accuracy and we allow for quite general
self-selection rules.
In particular, we allow the self-selection rule $S$ to be any 
\textit{convex-inducing rule} (Definition~\ref{defn:convex-inducing}): 
letting $\vec{y} \in \mathbb{R}^k$ be the vector of potential
outcomes, if we fix the $j$-th 
coordinate $y_j$ for any $j \in [k]$, then deciding whether $j$ will be the
winner (i.e., whether $j_* = j$) is the same as
deciding whether $\vec{y}_{-j}$ belongs to some convex set (this convex set
can depend on $y_j$).
We formally define the self-selection rule in Definition 
\ref{def:selfSelectionRule} and the convex-inducing self-selection rule in 
Definition \ref{defn:convex-inducing}. Below we present an informal version of 
our estimation theorem for the known-index case. The corresponding formal 
version of the theorem can be found in Theorem \ref{thm:knownIndex}.

\begin{inftheorem}[Known-Index Estimation -- Theorem \ref{thm:knownIndex}] \label{infthm:knownIndex}
  Let $(\vec{x}^{(i)}, y^{(i)}, j_*^{(i)})_{i = 1}^n$ be $n$ observations from 
  the known-index self-selection model with $k$ linear models $\vec{w}^*_1$,
  $\dots$,  $\vec{w}^*_k$ as described in Section 
  \ref{sec:intro:setup}. If the self-selection rule is convex-inducing and that the probability of observing each model is lower 
  bounded by $(\alpha / k)$ for some $\alpha > 0$, then there is an estimation algorithm that outputs 
  $\hat{\vec{w}}_1$, $\dots$, $\hat{\vec{w}}_k$ with 
  $\norm{\hat{\vec{w}}_i - \vec{w}_i^*} \le \eps$, when 
  $n \ge \poly(d, k, 1/\alpha, 1/\eps)$. Furthermore, the running time of the algorithm is 
  $\poly(d, k, 1/\alpha, 1/\eps)$.
\end{inftheorem}

The case of unknown-index is significantly more challenging and as we already 
mentioned even the with infinite number of samples it is unclear if we have 
enough information to estimate $\vec{w}_j$'s. We define the setting formally in
Definition \ref{defn:index_unobserved} and informally below:
\begin{description}
    \item[Unknown-index Setting:] We observe $n$ samples of the
    form $(\vec{x}^{(i)}, y^{(i)})$, where $\vec{x}^{(i)} \sim \mathcal{N}(0,
    \bm{I}_d)$, and
    $y^{(i)} = \max_{j \in [k]} y_{j}^{(i)}$.
\end{description}

Observe that in the unknown-index setting problem we assume a Gaussian prior distribution for the covariates $\vec{x}^{(i)}$. This is a classical assumption in other linear regression settings, e.g., mixtures of linear regressions, where even the identifiability of the parameters is unclear without prior distribution assumption on $\vec{x}^{(i)}$. We face a similar situation here, even our identifiability result for the unknown-index setting rely on the prior distribution of $\vec{x}^{(i)}$. This is not the case for the known-index setting where we can assume that $\vec{x}_i$ can be picked arbitrarily.

Our first result below shows this
identifiability and its formal version is Theorem \ref{thm:identifiability_no_index}.

\begin{inftheorem}[Unknown-index Identifiability -- Theorem \ref{thm:identifiability_no_index}] \label{infthm:unknownIndex:identifiability}
    If we have infinitely many samples from the unknown-index
    setting, then we can identify
    all $k$ linear models $\vec{w}^*_1$, $\dots$, $\vec{w}^*_k$.
\end{inftheorem}

  Next we continue with finite-sample and finite-time algorithms for the 
unknown-index case. To achieve this problem we need some separability assumption
between the $\vec{w}_i^*$.

\begin{infassumption}[Separability Assumption -- See Assumption \ref{as:no_index}] \label{infasp:separability}
    The projection of any vector
  $\vec{w}^*_i$ to the direction of any other vector $\vec{w}^*_j$ cannot be
    larger that the norm of $\vec{w}_j^*$ (and in fact, must be at least
    $\Delta$ smaller). Also, each $\vec{w}_j^*$ is bounded in norm.
\end{infassumption}

Our first result for arbitrary $k$ guarantees the estimation of the parameter
vectors $\vec{w}^*_i$ within accuracy $1/\poly(k)$. For a formal statement of 
the theorem below we refer to Theorem \ref{thm:no_index}.

\begin{inftheorem}[Unknown-index Estimation for General $k$ -- Theorem \ref{thm:no_index}] \label{infthm:unknownIndex:general}
    Let $(\vec{x}^{(i)}, y^{(i)})_{i = 1}^n$ be $n$ observations from a
  self-selection setting with $k$ linear models $\vec{w}^*_1$, $\dots$, 
  $\vec{w}^*_k$ as described in the unknown-index setting of Section  
  \ref{sec:intro:setup}. If we also assume Informal Assumption 
  \ref{infasp:separability}, then there exists an estimation algorithm that outputs  
  $\hat{\vec{w}}_1$, $\dots$, $\hat{\vec{w}}_k$ with 
  $\norm{\hat{\vec{w}}_i - \vec{w}_i^*} \le 1/\poly(k)$, assuming that 
  $n \ge \exp(\poly(k)) \cdot \poly(d)$. Furthermore, the running time of our
  algorithm is also $\exp(\poly(k)) \cdot \poly(d)$.
\end{inftheorem}

We pause briefly to make a quick remark on the results of
Theorem~\ref{infthm:unknownIndex:general}. Note that
Theorem~\ref{infthm:unknownIndex:general} (unlike Theorem \ref{infthm:knownIndex}) 
does not require a lower bound on the
observation probability $\alpha / k$.
This is due to the fact that
Informal Assumption~\ref{infasp:separability} implies a weaker exponential lower bound on
$\alpha$ which in addition, does not hold \emph{uniformly} over the regressors
$\vec{x}$. 
This actually suffices for our known-index algorithm, but incurs a worse sample
complexity than the one from Theorem~\ref{infthm:knownIndex}. Finally, we study the
case where $k = 2$ and show how to estimate the 
parameters to arbitrary accuracy:

\begin{inftheorem}[Unknown-index Estimation for $k = 2$ -- Theorem \ref{thm:no_index:2}] \label{infthm:unknownIndex:2}
  Consider any $\eps > 0$ and let $(\vec{x}^{(i)}, y^{(i)})_{i = 1}^n$ be $n$ observations from a
  self-selection setting with $2$ linear models $\vec{w}^*_1$, $\vec{w}^*_2$ 
  as described in the unknown-index setting of Section \ref{sec:intro:setup}. Under Informal Assumption \ref{infasp:separability}, there exists an estimation 
  algorithm that outputs $\hat{\vec{w}}_1$, $\hat{\vec{w}}_2$ with 
  $\norm{\hat{\vec{w}}_i - \vec{w}_i^*} \le \eps$, assuming that 
  $n \ge \poly(d, 1/\eps)$. Furthermore, the running time of the
  algorithm is also $\poly(d, 1/\eps)$.
\end{inftheorem}

\subsection{Our Techniques} \label{sec:intro:contributions}
We initiate a line of work on attaining statistical and
computational efficiency guarantees in the face of structured self-selection
bias. Below we briefly explain the main ideas for the proofs of our main results.

\paragraph{Known-index case.} 
In the known-index case, we require $\poly(1/\eps, k, d)$ sample and time
complexity to estimate all $k$ model parameters to accuracy $\eps$ in $d$
dimensions, and can accommodate quite general selection criteria. To prove this
known-index result, we construct a log-likelihood-inspired objective function that has the true set of parameters as optimum. The key difficulty associated with this formulation is that unlike ``nice'' settings (for example, the data generating model belongs to an exponential family), the log-likelihood involves an integration over possible outputs of the \emph{unobserved} models. This scenario is reminiscent of latent-variable models where strong structural properties on the objective functions are uncommon. Nevertheless, we show that this objective function is \emph{strongly convex} where we crucially rely on the variance reduction properties of log-concave densities conditioned on convex sets. Our next goal is to run projected stochastic gradient descent (PSGD) on this objective function. Unfortunately, in contrast to standard settings in stochastic optimization, we do not have simple
access to unbiased stochastic estimates of the gradient due to the integrating out of the unobserved models in the objective function. Consequently, the gradient in this case involves sampling from the conditional distribution over outputs from the unobserved models given the observed sample at the candidate parameter set currently being considered. To sample from this conditional distribution, we show that a projected version of the Langevin Monte Carlo sampling algorithm due to Bubeck, Eldan and Lehec \citep{bubecksampling} mixes fast and produces an approximate stochastic gradient. Finally, we show that this
approximate stochastic gradient suffices for the PSGD algorithm to converge.
We provide the details of this algorithm and its analysis in Section
\ref{sec:known}.

\paragraph{Identification with unknown index.} 
In the more challenging unknown-index case, it is not even clear whether the 
parameters $\vec{w}_j$ are \emph{identifiable} from the sample that we have. First, we show that this is in fact the case. Our proof uses an novel identification argument which we believe can be applied to other self-selection settings beyond the Max-selection criterion considered in this work. Formally speaking, we would like to exhibit the existence of a
mapping $f$ from $\Phi$ to the set of parameters given access to the distribution function, $\Phi$, of the pairs $(\vec{x}, y)$ generated according to the self-selection model with unknown indices. Our construction of $f$ is based on a conditional moment calculation where we analyze the moments of $y$ conditioned on $\vec{x}$ lying in various one-dimensional subspaces. The main observation is that while closed form solutions are not known for the conditional moments, the higher order moments of $y$ still determine the length of the projection of the parameter vector with the largest projection along $\vec{x}$. Concretely, we show that the higher-order moments of $y$ conditioned on $\vec{x}$ being parallel to a unit vector $v$ are upper and lower bounded (up to constants) by the moments of normal distribution with $0$ and variance $\max_i (\vec{v}^\top \vec{w}_i)^2 + 1$. While a single direction does not uniquely determine any of the underlying parameter vectors, the direction maximizing this quantity over all one dimensional subspaces corresponds to the unit vector along the longest parameter vectors allowing recovery of \emph{one} of $k$ vectors. In the next step, we show that we may effectively ``peel off'' the single identified model from the distribution function, $\Phi$, reducing the problem of recovering the remaining parameter vectors to a self-selection problem with $k - 1$ parameter vectors. A recursive application of this argument allows identification of the remaining parameter vectors one by one.

\paragraph{Estimation with unknown index.}
  We then move our attention to estimation with finite time and samples. We target
the common $\max$-selection criterion and, under some separability assumption
among the $\vec{w}_j$'s, we provide an algorithm with 
$\poly(d) \cdot \exp(\poly(k))$ sample and time complexity to estimate the
regression parameters up to error $1/\poly(k)$. 
Our technique to prove these finite-time and finite-sample results is to try to
develop a finite accuracy version of our identifiability argument. Unfortunately, 
this requires an exponential (in the ambient dimension $d$) sized grid search over the unit sphere. The traditional approach to address such difficulties is to restrict our search to a
suitably chosen $k$ dimensional subspace, $U$, identifiable from the data and
crucially \emph{contains} the parameter vectors, $\vec{w}_i$. We identify this
subspace through the spectrum of a suitably chosen matrix but we encounter an additional 
\emph{key} difficulty, compared to the previous applications of this method. While
our choice of matrix, $M \triangleq \E [y^2 \vec{x} \vec{x}^\top]$, is natural, $M$ does not decompose into a independently weighted sum of matrices each corresponding to a single parameter vector, in stark contrast to the scenario encountered in simpler problems such as mixtures of linear regressions. Hence, showing that the top singular subspace of $M$
contains the $\vec{w}_i$ is significantly more involved and requires novel 
approximation ideas. Despite these difficulties, we derive a closed form \emph{lower bound} for $M$ as a (positively) weighted sum of the identity matrix, $I$ and outer products $\vec{w}_i (\vec{w}_i)^\top$ which is \emph{tight} on the nullspace of the span of $\vec{w}_i$. Notably, the coefficients of a single parameter vector in the bound depend on the other vectors -- a scenario markedly different from other applications of this method such as mixtures of linear regressions. To obtain this closed form lower bound, we replace the max function in the self-selection criterion with a \textit{smooth maximum} function resulting
in a matrix $M'$ approximating $M$ and analyze its spectrum through several careful applications of Stein's Lemma facilitated by the differentiability of the smooth maximization function. Applying a limiting argument to $M'$, we obtain our lower bound and consequently, show that the span of the parameter vectors $\vec{w}_i$ is contained in the top-$k$ singular subspace of $M$ and the singular values associated with these directions are bounded away from those for the orthogonal complement establishing a strict spectral gap. Having identified the 
low-dimensional subspace containing the $\vec{w}_i$, we may now restrict our search to 
this subspace. We conclude our estimation argument with a careful finite-sample
adaptation of our identifiability argument highlighted above. 

\paragraph{Efficient unknown-index estimation for $k=2$.} For the specific yet well-studied case of $k = 2$, e.g. \cite{fair1972methods}, we develop a estimation algorithm based on the method of moments that achieves estimation error $\eps$ with  $\poly(d, 1/\eps)$ time and sample complexity. It is an interesting open problem whether a similar procedure may be derived for the general case.

\subsection{Related Work} \label{sec:intro:related}
As previously discussed, bias due to outcome self-selection is
a well-documented phenomenon across statistics, econometrics, and the social
sciences (see Section \ref{sec:intro} for a list of references). In this
section, we discuss existing approaches to solve such problems as well as a few similar problems to the statistical and computational
ones addressed in this work.

\paragraph{Classical approaches to self-selection.} There are many parametric and semi-parametric methods 
from Econometrics for parameter estimation in the presence of 
self-selection (see \citep{lee2001self} for an overview).
To describe results in the literature let us consider a
generalization of the problem we consider here, where 
the potential outcomes $y_i$ are generated
as 
$$y_i = x^\top w_i + u_i \qquad \text{ and } \qquad
i^* = \max\left(\{x^\top w_i + \eps_i: i \in [k]\}\right),$$
where the $\epsilon_i$ and $u_i$ are all jointly 
normally distributed.
Analytical algorithms for this setting 
(i.e., where the noise that determines selection 
is non-identical to the noise in the observed outcome 
\citep{manning1987monte, maddala1985survey, hay1984let, leung1996choice})
typically focus on the case when $k=2$ for tractability reasons.
The prevailing algorithm here is the two-stage estimator
\citep{lee1978estimation, heckman1979sample}, where 
one first estimates the model $x \to i^*$
(whose true parameter is $w_1 - w_2$) then uses a Heckman correction
to estimate the outcome parameters $w_1$ and $w_2$.
Alternatively, one can use likelihood-based approaches \citep{heckman1974shadow, nelson1977censored, griliches1978missing}, 
which---as shown by \citet{olsen1982distributional}---are identified 
for a known correlation $\rho$ between $\epsilon_i$ and $u_i$ above.
For $k > 2$, the likelihood function involves several integrals,
and so the prevailing approach is to use 
Markov-Chain Monte Carlo (MCMC) and simulation-based approaches, for which there are no 
established convergence rates and the underlying algorithms are not efficient (see \citep{geweke2001simulation} for an overview).

Thus, although the problems are similarly motivated, 
our approach to the known-index case differs from 
standard Econometrics approaches in a few ways.
First, in our model the selection noise and observation noise are 
identical, which introduces bias in typical two-stage estimators.
Second, our algorithm for the known-index case applies directly to $k > 2$
and has a (known) convergence rate that is nearly linear in $n$ and 
polynomial in $k$.
Finally, we are unaware of any algorithms in Econometrics
(whether analytical or simulation-based)
that tackle the unknown-index case, where even identifiability
(Section \ref{ssec:identifiability}) is non-trivial.

\paragraph{Truncated linear regression.} Both the known-index and unknown-index
cases addressed in this work bear some similarity to the truncated and censored
linear regression problem, wherein there is a single outcome that is only seen
if it falls within a fixed observation window. 
Truncated regression problems date back to at least the works of
\citet{tobin1958estimation,amemiya1973regression,hausman1977social},
who all note the effect of omitting
outcomes that fall below or above a certain threshold on resulting regression
models. 
Recently, \citet{daskalakis2019computationally} propose a gradient-based
algorithm (and derive statistical and computational efficiency guarantees) for
estimating parameters of truncated linear regression;
\citet{ilyas2020theoretical} extend these results to truncated probit and
logistic regression.

What makes self-selection bias more challenging than these settings is the
inherent {\em endogeneity} in the process of selecting samples to be observed.
That is, in the case of truncated regression, the learner is always aware of
which samples {\em would be} truncated had they been generated by a given set of
parameter estimates. In contrast, under the self-selection model, an incorrect
set of parameter estimates can lead to an incorrect estimate of the truncation
mechanism (since one depends on the other).

\paragraph{Mixtures of linear regressions.}
The problem that we consider---and in particular the unknown-index case---bears
some similarity to the well-studied problem of learning (noisy) mixtures of linear
regressions. In the mixture-of-regressions model, the observations are also of
the form $(\bm{x}^{(i)}, y^{(i)})$ where $y^{(i)}$ is the output of one of $k$
linear models: unlike our setting, however, in mixtures of linear regressions
the model from which $y^{(i)}$ is observed is selected at random, and crucially,
{\em independently} of the model outputs themselves.

To illustrate the significance of this distinction, observe that for 
{\em mixtures} of linear regressions, the corresponding ``known-index'' case
(i.e., where we observe the regression from which each datapoint was generated)
is trivial. In particular, it corresponds to estimating $k$ independent 
ordinary least squares models. In contrast, the known-index case here 
still requires a more elaborate algorithmic approach and analysis, as the 
datasets remain correlated with one another even conditioned on the selection 
indices.

In terms of finite-sample estimation algorithms, 
for the mixtures of linear regressions problem both noiseless
\cite{li2018learning} and noisy cases 
\cite{kwon2020converges, chen2020learning, diakonikolas2020small} have been
considered in the literature. The guarantee that we get for our unknown-index case 
is similar to the guarantee of \cite{chen2020learning} for mixtures of linear 
regression, i.e., to achieve error of order $1/\poly(k)$ when the magnitude of the
noise is also $1/\poly(k)$, we require time and sample complexity $\poly(d) \cdot \exp(\poly(k))$. For the
mixtures of linear regression case this $1/\poly(k)$ accuracy is enough to get 
arbitrary small accuracy $\eps$ by applying the EM algorithm that admits local 
convergence in this case as shown by \cite{kwon2020converges}. This local 
convergence result is missing in our unknown-index setting and it is a very 
interesting open problem that will lead to estimation with arbitrary small 
accuracy when combined with the results that we show in this paper. 

\paragraph{Max-affine regression.}
Another related line of literature to the unknown-index case is the problem of 
\textit{max-affine regression} \citep{ghosh2019max}. The main difference between
the settings is that in max-affine regression, the noise is added to the 
model after the maximum operator. Hence, the self-selection bias due to the noise 
does not appear in this setting. Although this seems like a small change it makes the problems completely different both technically and conceptually:
\begin{enumerate}
  \item[-] Conceptually, the main application of max-affine regression is in settings where we want to regress with respect to the set of convex functions. The set of convex functions is non-parametric and in high-dimensions this will result in an exponential sample and computational complexity. On the other hand if we parametrize the convex function as the maximum of a set of affine functions then we get a finite number of parameters and we can achieve much better efficiency. The resulting statistical problem is max-affine regression. But max-affine regression cannot capture the biases that we mention above due to strategic agents of imitation learning.
  \item[-] Technically, the max-affine regression could be solved via analogs of OLS. The main problem there is that the optimization landscape is non necessarily convex but from a statistical point of view minimizing the average square loss is a meaningful thing to do. In our case, the main problem is that the solutions of naive optimization problem are biased and hence we need to find the correct optimization problem that effectively debiases the data. In the unknown-index setting this is particularly challenging and in fact even arguing about the identifiability is non-trivial.
\end{enumerate}
Our unknown-index setting seems more difficult
from this point of view, but the fact that there is noise added after the maximum
makes the two problems not reducible from one to another. 

In fact, one can view max-affine regression as a case of self-selection where
the error terms $\eps_j$ are perfectly correlated across potential outcomes $j
\in [k]$, as opposed to perfectly independent (as in our setting). 
This raises a question of whether we can design algorithms to handle more
complex correlations between the model-specific error terms $\eps_j$.

    \section{Model and Main Results} \label{sec:model}
    
\noindent \textbf{Notation.} We use $\mathcal{N}(\vec{\mu}, \vec{\Sigma})$ to 
denote the normal distribution with mean $\vec{\mu}$ and covariance matrix 
$\vec{\Sigma}$. For any measurable set $\mathcal{K} \subseteq \R^d$ 
we denote 
with 
$\mathcal{N}(\vec{\mu}, \vec{\Sigma}; \mathcal{K})$ the normal distribution 
$\mathcal{N}(\vec{\mu}, \vec{\Sigma})$ conditioned on that the output belongs to 
$\mathcal{K}$. We will use $f_{\sigma}$ and $F_{\sigma}$ to denote the PDF and
CDF, respectively, of the single-dimensional normal distribution 
$\mathcal{N}(0, \sigma^2)$. When $\sigma$ is clear from the context we way just 
use $f$ and $F$. Let $\vec{A}$ be an $n \times m$ matrix, we define 
$\vec{A}^{\flat}$ to be a vector in $\R^{n \cdot m}$ that is the flattening of 
$\vec{A}$, where for the flattening we use the lexicographic order of the 
coordinates of $\vec{A}$. 
Let $\mathcal{K} \subseteq \R^d$ be a
convex set let $\vec{x} \in \R^d$, we define $\Pi_{\mathcal{K}}(\vec{x})$ to be 
the projection of $\vec{x}$ to $\mathcal{K}$. We use $\mathcal{B}(\vec{x}, r)$ to
denote the Euclidean ball with center $\vec{x}$ and radius $r$ and when 
$\vec{x} = \vec{0}$ we also use simply $\mathcal{B}(r)$. For a subspace $\mathcal{V}$, we use $\proj{\mathcal{V}}$ and $\pproj{\mathcal{V}}$ to denote the projection operators onto $\mathcal{V}$ and the orthogonal complement of $\mathcal{V}$ respectively. For a vector $\vec{v}$ and matrix $V$, $ \proj{\vec{v}}$ and $\proj{V}$ denote the projection operators onto the one-dimensional subspace along $\vec{v}$ and the column space of $V$ respectively and analogously for $\pproj{\vec{v}}$ and $\pproj{V}$.
\medskip

In this section we define the two models that we are solving: the known-index 
setting and the unknown-index setting, we describe the assumptions that we use 
for each of the settings and we formally state our main results.

\subsection{Known-Index Setting} \label{sec:model:knownIndex}

  We start with the definition of a self-selection rule that is fundamental in 
the modeling of the linear regression problem with self-selection bias in the 
known-index setting.

\begin{definition}[Self-Selection Rule] \label{def:selfSelectionRule}
  A \textit{self-selection rule} is a function 
  $S: \mathbb{R}^k \to [k]$. We assume throughout this work that we have query
  access to $S$, i.e.  for every $\bm{y} \in \mathbb{R}^k$ there is an oracle
  that outputs $S(\bm{y})$.
  We also define a \underline{slice} $\slice_j(a)$ of $S$ as follows:
  \[
    \slice_j: \mathbb{R} \rightrightarrows \mathbb{R}^{k-1} \qquad 
    \slice_j(a) \triangleq \{\bm{y}_{-j}: \bm{y} \in \mathbb{R}^{k}, S(\bm{y}) = j, \bm{y}_j = a\}.
  \]
  where $\rightrightarrows$ refers to a point-to-set map.
\end{definition}

  The first setting that we consider is the \textit{known-index setting}, where
the observed data for each covariate include the response variable, as well as
the index of the corresponding regressor.

\begin{definition}[Self-Selection with Observed Index]
    \label{defn:index_observed}
      Self-selection with observed index is parameterized by an unknown set of
    weight vectors $\vec{w}^*_1,\ldots \vec{w}_k^* \in \mathbb{R}^d$, a known 
    variance $\sigma > 0$,
    and a self-selection rule
    $S: \mathbb{R}^k \rightarrow [k]$. For $i \in [n]$, an observation
    $(\vec{x}^{(i)}, y^{(i)}, j_*^{(i)})$ in this model is a triplet, comprising
    a feature vector $\vec{x}^{(i)}$, and a pair $(y^{(i)},j_*^{(i)})$ sampled as
    follows conditioning on $\vec{x}^{(i)}$: 
    \begin{enumerate}
        \item[(1)] Sample the latent variables $y_{j}^{(i)} \sim  \mathcal{N}(\vec{w}_j^{* \top} \vec{x}^{(i)},
        \sigma^2)$ for each $j \in [k]$.
        \item[(2)] Reveal the observation index $j_*^{(i)} =
        S(y_{1}^{(i)},\ldots,y_{k}^{(i)})$ 
        and the response variable $y^{(i)} = y_{j_*^{(i)}}^{(i)}$
    \end{enumerate}
    For a fixed $\vec{x}$ and $\vec{W}^*=(\vec{w}^*_1,\ldots \vec{w}_k^*)$,  we use $\mathcal{D}(\vec{x}; \vec{W}^*)$ to denote the
    probability distribution of the pair $(i, y)$ sampled according  to Steps (1) and (2) above. For example, if 
    $S(y_{1}^{(i)},\ldots,y_{k}^{(i)}) = \arg\max_{j \in [k]} y_{j}^{(i)}$, then under
    this model we observe only the largest $y_{j}^{(i)}$ and its index. Our goal is
    to obtain an accurate estimate $\hat{\vec{w}}_j$ for each $\vec{w}^*_j$ given
    only samples from the above model.
\end{definition}

\noindent In most of this paper we are mainly concerned with estimating 
the weights $\vec{w_1^*}$, $\ldots$, $\vec{w_k^*}$ from observation of the
covariates and the {\em maximal} response variable 
\[y = \max_{j \in [k]}\ \{\vec{w_j}^\top \vec{x} + \bm{\eps}_j\}, \qquad \text{
where } \bm{\eps} \sim \mathcal{N}(0, \sigma^2 \cdot \bm{I}_{k}). \] 
In the observed-index setting, however, it turns out that our efficient estimation
can be applied to a much larger set of selection functions $S(\vec{y})$, which we
call {\em convex-inducing}.

\begin{definition}[Convex-inducing Self-Selection Function]
    \label{defn:convex-inducing}
    We call a self-selection function $S: \mathbb{R}^k \to [k]$ {\em
    convex-inducing} if, for each $j \in [k]$ and $a \in \mathbb{R}$, 
    the slice $\slice_j(a)$ (Definition \ref{def:selfSelectionRule})
    is a convex set.
\end{definition}

Notably, setting $S(\bm{y}) = \arg\max_{j \in [k]} \bm{y}_j$ recovers the maximum-response
observation model, and this choice of $S(\cdot)$ also satisfies Definition
\ref{defn:convex-inducing} with $C_j(a) = (-\infty, a]^{k-1}$. 
The definition also allows us to capture cases beyond the maximum; 
for example, convex-inducing functions also include $S(\bm{y}) = \arg\max_{j \in [k]} f_j(\bm{y}_j)$ for any set of monotonic functions $\{f_j\}_{j = 1}^k$.

Our estimation procedure relies on the following assumptions
on the feature and parameter vectors, which are classical and present even in many standard 
linear regression instances.

\begin{assumption}[Feature and Parameter Vectors] \label{asp:known:1}
    For the set of feature vectors $\{\vec{x}^{(i)}\}_{i = 1}^n$ that we have
  observed we assume that:
  \begin{align} \label{eq:asp:known}
    \norm{\vec{x}^{(i)}}_2 \le C \quad \text{for all $i \in [n] \quad$ and } \quad \quad
    \frac{1}{n} \sum_{i = 1}^n \vec{x}^{(i)} \vec{x}^{(i) \top} \succeq \mathbf{I}.
  \end{align}
  For the true parameter vectors $\vec{w}_1^*$, $\ldots$, $\vec{w}_k^*$ we
  assume that $\norm{\vec{w}_j^*}_2 \le B$.
\end{assumption}

Although our theorem allows for a wide range of self-selection rules, we need some
additional assumptions that allow the recovery of every parameter vector 
$\vec{w}_j$. To see why this is needed imagine the setting where the 
self-selection rule $S$ is the $\argmax$ and for some $j$ the coordinates of
$\vec{w}_j$ are extremely small. In this case it is impossible to hope to estimate
$\vec{w}_j$ since there is a huge probability that we do not even observe one 
sample of the form $y = \vec{x}^T \vec{w}_j + \eps_j$. For this reason, we need the
following assumption.

\begin{assumption}[Survival Probability] \label{asp:known:2}
  There exists a constant $\alpha > 0$ such that, 
  for every sample $(\vec{x}, y, j_*)$ that we have observed, 
  the following properties hold: 
  \begin{itemize}
    \item[(i)] for every $j \in [k]$, the
    probability that $j_* = j$ is at least $\alpha / k$,
    \item[(ii)] the mass of the set $C_{j_*}(y)$ with respect to 
    $\mathcal{N}((\vec{W}^{* \top} \vec{x})_{-j_*}, \sigma^2 \cdot \vec{I}_{k-1})$
    is at least $\alpha$.
  \end{itemize}
\end{assumption}

Finally, we need to assume oracle access to a self-selection rule $S(\cdot)$
that is convex-inducing, in the same sense introduced in Definition
\ref{defn:convex-inducing}: 
\begin{assumption}[Self-Selection Rule] \label{asp:known:3}
  We assume that the self-selection rule $S: \mathbb{R}^k \to [k]$ is
  convex-inducing, and that for every $j \in [k]$ and $y \in \R$ we have access
  to the following:  
  \begin{itemize}
    \item[(i)] a membership oracle for $S(\vec{y})$, i.e., for every 
    $\vec{y} \in \R^{k}$ and $j \in [k]$, we can know whether $S(\vec{y}) = j$,
    and
    \item[(ii)] a projection oracle to the convex sets  that correspond
    to slices $\slice_j(\cdot)$ of the self-selection rule, i.e., for every
    $\vec{y} \in \mathbb{R}^{k-1}$ 
    and every $a \in \mathbb{R}$,
    we can efficiently
    compute 
    $\vec{y}' =
    \argmin_{\vec{s} \in \slice_j(a)} \norm{\vec{y} - \vec{s}}$.
  \end{itemize}
\end{assumption}

Having established the above definitions and assumptions, 
we are now ready to state our main 
theorem for the known-index setting.

\begin{theorem}[Known-Index Estimation] \label{thm:knownIndex}
    Let $(\vec{x}^{(i)}, y^{(i)}, i^{(i)})_{i = 1}^n$ be $n$ observations from 
  a self-selection setting with $k$ linear models $\vec{w}^*_1$, $\dots$, 
  $\vec{w}^*_k$ as in Definition \ref{defn:index_observed}. Under Assumptions \ref{asp:known:1}, \ref{asp:known:2}(i), and \ref{asp:known:3}(i),
  there exists an algorithm that outputs $\hat{\vec{w}}_1$, $\dots$,
  $\hat{\vec{w}}_k$ such that with probability at least $0.99$, for every 
  $j \in [k]$
  \[ \norm{\hat{\vec{w}}_j - \vec{w}_j^*}_2^2 \le \poly(\sigma, k, 1/\alpha, B, C) \cdot \frac{\log(n)}{n}.\] 
  If we additionally assume the Assumptions \ref{asp:known:2}(ii), and 
  \ref{asp:known:3}(ii), then the running time of the algorithm is 
  $\poly(n, d, k, 1/\alpha, B, C, \sigma, 1/\sigma)$.
\end{theorem}

  In Section \ref{sec:known}, we explain our algorithm for proving Theorem 
\ref{thm:knownIndex} and we describe the main ideas and techniques for the proof.
Specifically,
to construct an efficient algorithm we use an interesting combination of 
projected gradient descent and the Langevin algorithm, after we develop an 
appropriate objective function.

\subsection{Unknown-Index Setting} \label{sec:model:unknownIndex}

We next consider the more challenging {\em unknown-index} setting, where we have
sample access to the response variable, but do not observe the index of the
corresponding weight vector. In this more challenging setting, we make a few 
additional assumptions on the structure of the problem, namely that the
covariates $\vec{x}^{(i)}$ are drawn from an mean-zero identity-covariance
Gaussian distribution (rather than being arbitrary); that the noise terms
$\vec{\eps}^{(i)}$ are also identity-covariance (rather than $\sigma\cdot
\bm{I}_k$ for $\sigma > 0$); and that the self-selection rule is the maximum
response rule $S(\bm{y}) = \arg\max_{j \in [k]} \bm{y}_j$ (rather than an arbitrary
convex-inducing rule).

\begin{definition}[Self-Selection with Unknown Index]
    \label{defn:index_unobserved}
    Just as for Definition \ref{defn:index_observed}, we have a set of weight
  vectors $\vec{w}^*_1,\ldots \vec{w}^*_k \in \mathbb{R}^d$. 
  We assume that vectors $\vec{x}^{(i)}$ are drawn i.i.d. from the standard
  multivariate normal distribution $\mathcal{N}(\vec{0}, \bm{I}_d)$. For
  each sampled covariate $\vec{x}^{(i)}$, we:
  \begin{enumerate}
    \item[(1)] Sample $y_{j}^{(i)} \sim \mathcal{N}(\vec{w_j}^{* \top}
    \vec{x}^{(i)}, 1)$ for each $j \in [k]$.
    \item[(2)] Compute the maximum response $y^{(i)} = \max_{j \in [k]} y_{j}^{(i)}$.
    \item[(3)] Observe only the covariate-response pair $(\vec{x}^{(i)},
    y^{(i)})$: in particular, we do not observe the index of the maximum
    response. 
  \end{enumerate}
  Our goal is to obtain an accurate estimate $\hat{\vec{w}}_j$ for each
  $\vec{w}^*_j$ given only samples from the above model.
\end{definition}

This model resembles mixtures of linear regressions, with the key
difference being that in the latter, the index of the weight vector used for
each covariate is sampled i.i.d. from a categorical distribution with fixed
mixture probability. In contrast, here the probability of observing a response
from a given weight vector depends on both the covariate and the sampled noise
$\eta_i$. 
\smallskip

  In this model it is not even clear that estimation of $\vec{w}^*_j$'s is
  possible: 
indeed, our first result (Theorem \ref{thm:identifiability_no_index}) is an
information-theoretic one that shows (infinite-sample) identifiability for
unknown-index model. 
To show this information-theoretic result, we use a novel argument that might be
of independent interest---we believe it may be used to show 
identifiability of other statistical problems with self-selection bias as well.
The proof of the theorem below is presented in Section 
\ref{ssec:identifiability}.

\begin{restatable}{theorem}{unknownidentify} \label{thm:identifiability_no_index}
    Let $\vec{W}^* = [\vec{w}^*_j]_{j = 1}^k \in \R^{d \times k}$ and
  $\Phi_{\vec{W}^*}$ be the distribution function of the pairs 
  $(\vec{x}^{(i)}, y^{(i)})$ associated with the self-selection model with
  unknown indices as described in Definition~\ref{defn:index_unobserved}. 
  Then, there exists a mapping $f$ satisfying:
  \begin{equation*}
    \forall~\vec{W}^* = [\vec{w}^*_j]_{j = 1}^k \subset \R^d ~ \text{it holds that} ~ f(\Phi_{\vec{W}^*}) = \vec{W}^*.
  \end{equation*}
\end{restatable}

\noindent In order to transform the above information-theoretic result to a 
finite-sample and finite-time bound in the unknown-index case, we need the 
following separability assumption for the $\vec{w}^*_j$'s.

\begin{assumption}[Separability Assumption] %
\label{as:no_index}
  For some known real values $B$, $\Delta$ it holds that:
\begin{equation*}
  \forall i \neq j: \frac{\abs*{{\vec{w}_i^{* \top} \vec{w}^*_j}}}{\norm{\vec{w}^*_j}} + \Delta \leq \norm{\vec{w}^*_j} \text{ and } \max_{j \in [k]} \norm{\vec{w}^*_j} \leq B.
  \tag{A} \label{eq:asump_no_index}
 \end{equation*}
\end{assumption}

\noindent Assumption \ref{as:no_index} allows us to design an algorithm 
that yields $\eps$-accurate estimates of the $\bm{w}_j$ vectors:
\begin{restatable}{theorem}{unknownind}
    \label{thm:no_index}
    Let $(\vec{x}^{(i)}, y^{(i)})_{i = 1}^n$ be $n$ observations from a
  self-selection setting with $k$ linear models $\vec{w}^*_1$, $\dots$,
  $\vec{w}^*_k$ as described in the unknown-index setting in Definition 
  \ref{defn:index_unobserved}. If we assume Assumption \ref{as:no_index}, then there exists an estimation algorithm that outputs a set of weights $\{\hat{\vec{w}}_i\}_{i = 1}^k$ and an ordering of these weights $\hat{\vec{w}}_1$, $\dots$, $\hat{\vec{w}}_k$ such that with
  probability at least $0.99$, for every $j \in [k]$
  \[ \norm{\hat{\vec{w}}_j - \vec{w}_j^*}_2 \le \eps,\] 
  as long as 
  $n \geq \poly (d) \cdot \exp \blr{\poly (B / \eps) \cdot \widetilde{O}(k)}$ and
  $\eps \le \Delta/16$. 
  Furthermore, the running time of the algorithm is at most 
  $n \cdot \exp \blr{\poly (B / \eps) \cdot \widetilde{O} (k)}$. 
\end{restatable}

\noindent Our estimation algorithm above, 
which we provide in Sections \ref{ssec:identifying_subspace}
and \ref{ssec:est_low_dim}, only makes sense for $\eps = 1/\poly(k)$ and 
resembles the corresponding results of mixtures of linear regressions 
\cite{li2018learning, chen2020learning}. What we are missing for this model is a 
local analysis corresponding to \cite{kwon2020converges} that will enable us to 
get error $\eps$ with running time and number of samples that are polynomial in
$1/\eps$. 

Finally, in the special yet very relevant case of $k = 2$, we are able to
improve our aforementioned result and show (in Section \ref{ssec:two_models}) that a moment-based algorithm
efficiently recovers $\vec{w}_1$ and $\vec{w}_2$ up to $\eps$ with running time
and number of samples that are polynomial in $1/\eps$.

\begin{restatable}{theorem}{unknownindtwo}
    \label{thm:no_index:2}
    Let $(\vec{x}^{(i)}, y^{(i)})_{i = 1}^n$ be $n$ observations from a
  self-selection setting with $2$ linear models $\vec{w}^*_1$, $\vec{w}^*_2$ as
  described in the unknown-index setting of Section \ref{sec:model}. If
  Assumption \ref{as:no_index} holds, then there
  exists an estimation algorithm that outputs $\hat{\vec{w}}_1$,
  $\hat{\vec{w}}_2$ such that with probability at least $0.99$, for every 
  $j \in [2]$,
  \[ \norm{\hat{\vec{w}}_j - \vec{w}_j^*}_2 \le \eps,\] 
  as long as $n \geq \poly (d, 1/\eps, B, 1/\Delta)$ and 
  $\eps \le \Delta/4$. Furthermore, the running time of the algorithm is at most 
  $n \cdot \poly (d, 1/\eps, B,  1/ \Delta)$. 
\end{restatable}

\noindent In Section \ref{sec:unknown} we describe the algorithms and proofs for Theorem 
\ref{thm:no_index} and Theorem \ref{thm:no_index:2}.
\medskip

\noindent \textbf{Remark.} (High-Probability Results)
  All of the above results are expressed in term of constant probability of 
  error. We can boost this probability to $\delta$ by paying an additional 
  $\log(1/\delta)$ factor in the sample and time complexities. This boosting can 
  be done because we are solving a parametric problem and it is a folklore idea that
  any probability of error less that $1/2$ can be boosted to $\delta$. Roughly 
  the way that this boosting works is that we run the algorithm independently 
  $\log(1/\delta)$ times and from the $\log(1/\delta)$ different estimates we 
  keep one that contains at least half of all the others within a ball of radius 
  $2 \eps$.

    \section{Parameter Estimation for the Known-Index Setting} \label{sec:known}
      In this section we present and analyze our algorithm for estimating the
parameters $\{\vec{w}_k\}$ from samples distributed according to Definition
\ref{defn:index_observed}. At a high level, our approach is to run projected
stochastic gradient descent (SGD) on the appropriate objective function whose 
optima coincide with the true set of parameters. We start with the definition of 
this objective function, then we show the design of the appropriate projection set
and then we proceed with proving its main properties. We conclude with the proof
of our estimation theorem.

\subsection{Objective Function for Linear Regression with Self-Selection Bias} 
\label{sec:log_lik}
  
The objective function that we use is inspired by the log-likelihood function.
We show that our objective function is convex (even though linear regression with 
self-selection bias does not belong to any exponential family). Suppose we have a
given parameter estimate for the $[\vec{w}_j^*]_{j=1}^k$ given by 
$\bm{W} = [\vec{w}_j]_{j=1}^k$ then we define its objective value
$\overline{\ell}(\vec{W})$ as follows.
\begin{align}
    \overline{\ell}(\vec{W}) 
    &\triangleq 
    \frac{1}{n} \sum_{i = 1}^n 
        \mathbb{E}_{(y, j_*) \sim \mathcal{D}(\vec{x}^{(i)}; \vec{W}^*)} 
        \left[ \ell(\vec{W}; \vec{x}^{(i)}, y, j_*) \right]  \notag \\
    &\triangleq \frac{1}{n} \sum_{i = 1}^n \mathbb{E}_{(y, j_*)} \left[
        \log\lr{f_{\sigma}(y - \bm{w}_{j_*}^\top \vec{x}^{(i)})} + 
        \log\lr{\int_{C_{j_*}(y)} \prod_{j \neq j_*} 
        f_{\sigma}(\bm{z}_j - \vec{w}_j^\top \vec{x}^{(i)})\, d\bm{z}_{-j_*} }
    \right] \label{eq:objectiveDefinition},
\end{align}
where we recall that $f_\sigma$ is the density function of the standard normal
distribution. 
The above expression is based on the population
likelihood under the current estimate $\vec{W}$ of the pair $(y, j_*)$
conditioned on the value of $\vec{x}$: see Appendix
\ref{app:computationsLogLikelihood} for the exact derivation. 
The gradient of $\overline{\ell}$ can then be expressed in the following form: 
\begin{align}
    \label{eq:ll_grad_i:main}
    \nabla_{\vec{w}_j} \overline{\ell}(\vec{W}) &= 
    \frac{1}{n\sigma^2} \sum_{i = 1}^n \Exp_{(y, j_*)} \left[ 
    \bm{1}_{j = j_*}
    \cdot y +
    \bm{1}_{j \neq j_*} \cdot 
    \mathbb{E}_{\vec{z}_{-j_*} \sim \mathcal{N}((\vec{W}^\top x^{(i)})_{-j_*}, \sigma^2 \bm{I}_{k-1})}\left[z_j|\bm{z}_{-j_*} \in C_{j_*}(y)\right]
    - \vec{w}_{j}^\top \vec{x}^{(i)}
    \right] \vec{x}^{(i)}.
\end{align}
The first thing to verify is that the set of true parameters $\vec{W}^*$ are 
a stationary point of the objective function that we proposed above. The 
proof of the following lemma can be found in Appendix \ref{app:known}. 

\begin{lemma} \label{lem:known:stationary}
  It holds that $\nabla \overline{\ell}(\vec{W}^*) = 0$, where 
  $\bm{W}^* = [\bm{w}_j^*]_{j=1}^k$ is the set of true parameters 
  of the known-index self-selection model described in Definition
  \ref{defn:index_observed}. 
\end{lemma}
\begin{proof}
  See Appendix \ref{app:known:stationary}.
\end{proof}

Our goal is to apply projected stochastic gradient descent (PSGD) on
$\overline{\ell}$. To this end, we need to prove that our objective function is
actually strongly concave and hence the optimum of $\overline{\ell}$ is unique and
equal to $\vec{W}^*$. We show this strong convexity in Section
\ref{sec:convexity}. Next, we need to show that we actually apply PSGD and hence
need to find a procedure to sample unbiased estimates of the gradient of
$\bar{\ell}$. Unfortunately the form of the objective function does not allow us
to find such an efficient procedure. For this reason we relax our requirement to
finding approximately unbiased estimates of $\nabla \bar{\ell}$. To achieve this
we use a projected version of Langevin dynamics as we show in Section 
\ref{sec:Langevin}. Additionally, we need to show that the second moment of our 
gradient estimates cannot be very large which we also show in Section 
\ref{sec:Langevin}. Finally, we need to adapt the proof of convergence of PSGD 
to show that the small bias that Langevin dynamics introduces can be controlled 
in a way that does not severely affect the quality of the output estimation 
which we show in Section \ref{sec:biased_sgd}. In Section 
\ref{sec:known:main:proof} we combine everything together to prove our 
estimation result.

\subsection{Strong Concavity} \label{sec:convexity}

The Hessian of $\overline{\ell}$ is difficult to analyze directly. We thus start
with the Hessian of the log-likelihood for a single sample $(\vec{x}^{(i)},
y^{(i)}, j_*^{(i)})$. In particular, in Appendix
\ref{app:computationsLogLikelihood} we derive the Hessian of this function
$\ell(\vec{W}; \vec{x}^{(i)}, y^{(i)}, j_*^{(i)})$, 
which comprises blocks $\bm{H}_{jl}$ such that 
\[ (\bm{H}_{j,l})_{ab} = \ddt{(\bm{w}_l)_a}{(\bm{w}_j)_b} \ell(\bm{W};
\vec{x}, y, j_*). \]
Following the computation in Appendix \ref{app:computationsLogLikelihood}, it
follows that for a single sample $\vec{x}, y, j_*$, the matrix block 
$\bm{H}_{j,j_*} = 0$ for all $j \neq j_*$. Thus, it remains to consider only the
blocks $\bm{H}_{j_*,j_*}$ and $\bm{H}_{j,l}$ for $j,l \neq j_*$. In Appendix 
\ref{app:computationsLogLikelihood} we show that
\begin{align*}
  \bm{H}_{j_*,j_*} = -\frac{1}{\sigma^2}\vec{x}\vec{x}^\top
\end{align*}
and that
\begin{equation*}
  \bm{H}_{\vec{W}_{-j_*}} = \frac{1}{\sigma^4} \lr{
    \text{Cov}_{\bm{z}_{-j_*} \sim \mathcal{N}((\bm{W}^\top \vec{x})_{-j_*},\, \sigma^2 \bm{I}_{k-1})}
        \left[z_j, z_l\,|\, \bm{z}_{-j_*} \in C_{j_*}(y)\right]
    - \sigma^2 \bm{I}
  } \otimes \vec{x}\vec{x}^\top,
\end{equation*}
where $\otimes$ represents the Kronecker product. Now, the key property of
convex-inducing selection functions in our proof is that, for Gaussian random
variables over $\mathbb{R}^{k-1}$, the variance is non-increasing when the
variable is restricted to a convex set.

\begin{lemma}[Corollary 2.1 of \citep{kanter1977reduction}]
    Let $\vec{X} \in \mathbb{R}^n$ be a random vector with Gaussian density 
  $f_{\vec{X}}$. For a convex set $A \subseteq \mathbb{R}^n$ with positive mass
  under the distribution of $\vec{X}$, define $\vec{X}_A$ to be $X$ restricted to
  $A$, i.e., a random variable with density 
  $f_{\vec{X}_A}(\vec{x}) = f_{\vec{X}}(\vec{x}) \cdot (\int_{A} f_X(\vec{z})\, d\vec{z})^{-1}$.
  Then, for all $\vec{v} \in \mathbb{R}^n$, 
  \[ \text{Var}[\vec{v}^\top \vec{X}_A] \leq \text{Var}[\vec{v}^\top \vec{X}]. \]
\end{lemma}

In particular, together with our thickness assumption and properties of the
Kronecker product, this implies that 
$\bm{H}_{\vec{W}_{-j_*}} \preceq 0$. Thus, the complete Hessian of the 
function $\ell$ can be expressed as a block matrix of the form:
\[ \bm{H} 
    = \left[\begin{matrix}
        -\frac{1}{\sigma^2} \vec{x}\vec{x}^\top & \bm{0} \in \mathbb{R}^{d \times d(k-1)} \\
        \bm{0} \in \mathbb{R}^{d(k-1) \times d} & \bm{H}_{\vec{W}_{-j_*}}
    \end{matrix}\right]
    \preceq \left[\begin{matrix}
        -\frac{1}{\sigma^2} \vec{x}\vec{x}^\top & \bm{0} \\
        \bm{0} & \bm{0}
    \end{matrix}\right].
\]
We are now ready to upper bound the Hessian $\bm{H}_{pop}$ of our objective
function $\bar{\ell}$. 
In particular, 
at this point we can use our minimum-probability assumption 
(Assumption \ref{asp:known:1}) and our thickness of covariates assumption
(Assumption \ref{asp:known:2}) from which we get that for the Hessian
$\bm{H}_{pop}$ it holds that 
\[ \bm{H}_{pop} \preceq -\frac{\alpha}{\sigma^2 \cdot k} \bm{I}. \]
From the above we conclude that the following lemma

\begin{lemma} \label{lem:strongConcavity}
    The objective function $\bar{\ell}$ is 
  $\left(\frac{\alpha}{\sigma^2 \cdot k}\right)$-strongly-concave.
\end{lemma}

\subsection{Approximate Stochastic Gradient Estimation} \label{sec:Langevin}
In this section we describe an algorithm for sampling approximate stochastic 
estimates of our objective function $\bar{\ell}$. Our algorithm is based on 
projected Langevin dynamics. We start with the expression of the gradient of
$\bar{\ell}$ based on \eqref{eq:ll_grad_i:main}. 
\begin{align}
    \label{eq:ll_grad_i:main:2}
    \nabla_{\vec{w}_j} \overline{\ell}(\vec{W}) &= 
       \frac{1}{n\sigma^2} \sum_{i = 1}^n \Exp_{(y, j_*)} \left[ 
       \bm{1}_{j = j_*}
       \cdot y +
       \bm{1}_{j \neq j_*} \cdot 
       \mathbb{E}_{\vec{z}_{-j_*} \sim \mathcal{N}((\vec{W}^\top x^{(i)})_{-j_*}, \sigma^2 \bm{I}_{k-1})}\left[z_j|\bm{z}_{-j_*} \in C_{j_*}(y)\right]
       - \vec{w}_{j}^\top \vec{x}^{(i)}
       \right] \vec{x}^{(i)}.
\end{align}
where we remind that $\mathbb{E}_{(y,j^*)}$ denotes the expectation of the pair
$(y^{(i)}, j_*^{(i)})$ conditioned on $\vec{x}^{(i)}$ and $\vec{W}^*$.
To obtain stochastic gradient estimates, we will replace these expectations 
with their 
corresponding observed values as we will see below. The more difficult step of
the gradient estimation process is 
sampling the last term of \eqref{eq:ll_grad_i:main:2}, for which 
it suffices to be able to sample the truncated normal distribution 
\begin{align}
  \mathcal{N}((\vec{W}^\top x^{(i)})_{-j_*}, \sigma^2 \bm{I}_{k-1};\ C_{j_*}(y)) 
  \label{eq:difficultTermToSample}
\end{align}
given some set of parameter $\vec{W}$, a vector of covariates 
$\vec{x}^{(i)}$ and a pair $(y, j_*)$ drawn from 
$\mathcal{D}(\vec{x}^{(i)}; \vec{W}^*)$. The simplest way to get a sample from 
\eqref{eq:difficultTermToSample} is to first sample from 
$\mathcal{N}((\vec{W}^\top x^{(i)})_{-j_*}, \sigma^2 \bm{I}_{k-1}),$ 
and then
apply rejection sampling until we get a sample inside $C_{j_*}(y)$.
This is feasible
information-theoretically but it might require a lot of computational steps if 
the survival probability of $C_{j_*}(y)$ is small. In particular, the rejection 
sampling might require time that is exponential in the norm of the $\vec{W}_j$'s.
For this reason, if we require statistical efficiency we need to apply a more 
elaborate technique. In particular, we use projected Langevin 
dynamics. Let $K = \slice_{j_*}(y) \cap \mathcal{B}(R)$ for some sufficiently large 
constant $R$ and let $\vec{\mu}_{-j_*} = (\vec{W}^\top \vec{x}^{(i)})_{-j_*}$ 
for the rest of this section. The iteration of projected Langevin algorithm for
sampling 
is the following \cite{bubecksampling}:
\begin{align}
  \vec{z}^{(t + 1)} = \Pi_K\left(\vec{z}^{(t)} - \frac{\gamma}{2 \cdot \sigma^2} 
  (\vec{z}^{(t)} - \vec{\mu}_{-j_*}) + \sqrt{\gamma} \cdot \vec{\xi}^{(t)} \right) \label{eq:projectedLangevin}
\end{align}
where $\vec{\xi}^{(1)}, \vec{\xi}^{(2)}, \dots$ are i.i.d. samples from the 
standard normal distribution in $(k - 1)$-dimensions. The next lemma describes 
the sampling guarantees of the Langevin iteration \eqref{eq:projectedLangevin}.

\begin{lemma} \label{lem:samp_trunc}
    Let $L \subset \R^{(k - 1)}$ be a convex set with 
  $\mathbb{P}_{\mathcal{N}(\vec{0}, \sigma^2 \vec{I})}(L) \geq \alpha$ for $\alpha > 0$.
  Then, for any $\vec{\mu}_{-j_*} \in \R^{k - 1}$ and $\epsilon \in (0, 1/2]$, the 
  projected Langevin sampling algorithm \eqref{eq:projectedLangevin} with 
    $K = L \cap \mathcal{B}(R)$ for some appropriate value $R$, and initialized
    with $\bm{z}^{(0)} = \Pi_K(\bm{0}_{k-1})$, generates a random
  variable $\hat{\vec{X}} = \vec{z}^{(m)}$ satisfying 
    \begin{equation*}
        \mathrm{TV} \left(\widehat{\vec{X}}, \mathcal{N}(\vec{\mu}_{-j_*}, \sigma^2 \cdot \vec{I}_{k-1}, K)\right) \leq
        \epsilon
    \end{equation*}
  assuming that the number of steps $m$ is larger than 
  $\mathrm{poly} (k, \norm{\vec{w}}, 1 / \epsilon, 1 / \alpha, \sigma^2, 1/\sigma^2)$.
\end{lemma}
\begin{proof}
The proof of this lemma can be found in Appendix \ref{app:lem:samp_trunc}.
\end{proof}
\smallskip

Now that we can sample from the distribution 
\eqref{eq:difficultTermToSample}, we can move to approximately estimating a 
stochastic gradient of $\overline{\ell}$. First, we sample uniformly $i \in [n]$a
uniformly at random, and we fix the corresponding $\vec{x}^{(i)}$. Then, we use 
the $i$-th sample from the true model to substitute in the pair $(y, j_*)$. 
Finally, we use the Langevin algorithm that we described above to sample 
\eqref{eq:difficultTermToSample}. Before moving to bounding the bias of our 
estimator there is one more thing that we need to take care of, and this is that
for every $\vec{x}^{(i)}$ we only have one sample of the pair $(y, j_*)$. Hence,
we need to make sure that during the execution of the algorithm while we pick 
the indices $i \in [n]$ uniformly at random we will never pick the same 
index $i$ twice. To ensure that we are going to require more samples than the
ones we need. 

Let $n$ be the total number of samples that we have and $T$ be total number of
samples that we need for our PSGD algorithm. A straightforward birthday paradox
calculation 
yields that the probability 
of sampling the same $i$ twice is at most $2 T^2/n$. Thus,
if we pick $n \ge 2 T^2/\zeta$, then the collision probability 
during the execution of the PSGD algorithm is at most $\zeta$.
\smallskip

We are now ready to put everything together in algorithm that 
describes our combined estimation procedure. The following lemma whose proof be
found in Appendix \ref{app:lem:samplingGradient} describes the performance
guarantees of the above estimation algorithm.

\begin{algorithm}[t]
    \caption{Approximate Stochastic Gradient Estimation Algorithm}\label{alg:langevin}
    \begin{algorithmic}[1]
    \Procedure{EstimateGradient}{$\vec{W}$}
        \State sample $i$ uniformly from $[n]$
        \State $K \gets C_{j_*^{(i)}}(y^{(i)}) \cap \mathcal{B}(R)$
        \State $\vec{\mu} \gets (\vec{W}^\top \vec{x}^{(i)})_{-j_*^{(i)}}$
        \State $\vec{z}^{(0)} \gets \Pi_K(\vec{0})$
        \For{$t = 1,\ldots, m$}
            \State sample $\vec{\xi}^{(t-1)}$ from $\mathcal{N}(\vec{0}, \vec{I})$
            \State  $\vec{z}^{(t)} \gets \Pi_{K}\left(\vec{z}^{(t-1)} - 
            \frac{\gamma}{2 \cdot \sigma^2} (\vec{z}^{(t-1)} - \vec{\mu}) + 
            \sqrt{\gamma} \cdot \vec{\xi}^{(t-1)} \right)$
        \EndFor
        \For{$j = 1,\ldots, k$}
        \State $\vec{g}_j \gets \frac{1}{\sigma^2} \left(
        \bm{1}_{j_*^{(i)} = j} \vec{y}^{(i)} + \bm{1}_{j_*^{(i)} \neq j} \cdot
        z^{(m)}_j
          - \vec{w}_j^\top \vec{x}^{(i)}
          \right) \cdot \vec{x}^{(i)}$
        \EndFor        
        \State \Return $\vec{g} = (\vec{g}_1, \dots, \vec{g}_k)$
    \EndProcedure
    \end{algorithmic}
\end{algorithm}

\begin{algorithm}[t]
    \caption{Projected Stochastic Gradient Descent}\label{alg:psgd}
    \begin{algorithmic}[1]
    \Procedure{PSGD}{}
       \State $\vec{W}^{(0)} \gets 0$
       \For{$t = 1, \ldots, T$}
          \State $\eta_t \gets 1/\lambda \cdot t$
          \State $\vec{g} \gets \textsc{EstimateGradient}(\vec{W}^{(t)})$
           \State $\vec{w}^{(t)}_j \gets \Pi_{\mathcal{K}}\lr{\vec{w}^{(t)}_j 
           - \eta_t \cdot \vec{g}^{(t)}_j}$ for all $j \in [k]$
       \EndFor
       \State \Return $\bar{\vec{W}} \triangleq \frac{1}{T} \sum_{t = 1}^{T} \left(\vec{w}^{(t)}_j\right)_{j = 1}^k$
    \EndProcedure
    \end{algorithmic}
\end{algorithm}

\begin{lemma} \label{lem:samplingGradient}
    Let $\vec{g}^{(1)}, \dots, \vec{g}^{(T)}$ be a sequence of outputs of 
  Algorithm \ref{alg:langevin} when used with input 
  $\vec{W}^{(1)}, \dots, \vec{W}^{(T)}$, where
  $\norm{\vec{W}^{(p)}}_2 \le k \cdot B$ and $\vec{W}^{(p)}$ can
  depend on $\vec{W}^{(p - 1)}, \vec{g}^{(p - 1)}$. If $n \ge 2 T^2/\zeta$,
  for the hyperparameters $\eta$, $R$, it holds that 
  $R, 1/\eta \le \mathrm{poly} (k, B, 1 / \beta, 1 / \alpha, \sigma^2, 1/\sigma^2)$, and $m \ge \mathrm{poly} (k, B, 1 / \beta, 1 / \alpha, \sigma^2, 1/\sigma^2)$
  then with probability at least $1 - \zeta$ it holds that for every $p \in [T]$
  \begin{align}
    \norm{\Exp \left[\vec{g}^{p} \mid \vec{W}^{(p - 1)}, \vec{g}^{(p - 1)} \right]  -  \nabla \bar{\ell}(\vec{W}^{(p)})}_2 \le \beta. \label{eq:boundOnBias}
  \end{align}
\end{lemma}

\subsection{Stochastic Gradient Descent with Biased Gradients}
\label{sec:biased_sgd}

  In the previous section we showed that we can compute approximate stochastic 
gradients of the strongly-concave function $\bar{\ell}$. In this section we show 
that this is enough to approximately optimize $\bar{\ell}$ using projected 
gradient descent. We start with a description of the PSGD algorithm.

\begin{lemma}%
    \label{lemma:sgd}
    Let $f: \mathbb{R}^k \to \mathbb{R}$ be a convex function, $\mathcal{K} \subset
    \mathbb{R}^k$ a convex set, and fix an initial estimate $w^{(0)} \in \mathcal{K}$.
    Now, let $\vec{x}^{(1)}, \ldots, \vec{x}^{(T)}$ be the iterates generated by 
    running $T$
    steps of projected SGD using gradient estimates 
    $\vec{g}^{(1)}, \ldots, \vec{g}^{(M)}$
    satisfying $\mathbb{E}[\vec{g}^{(i)} | \vec{x}^{(i-1)}] = 
    \nabla f(\vec{x}^{(i-1)}) + \vec{b}^{(i)}
    $. Let $\vec{x}_* = \arg\min_{\vec{x} \in
    \mathcal{K}} f(w)$ be a minimizer of $f$. Then, if we assume:
    \begin{enumerate}
        \item[(i)] \textbf{Bounded step variance:} 
        $\mathbb{E}\left[\|\vec{g}^{(i)}\|_2^2 \right] \leq \rho^2$,
        \item[(ii)] \textbf{Strong convexity:} 
        $f$ is $\lambda$-strongly convex, and 
        \item[(iii)] \textbf{Bounded gradient bias:}
        $\|\vec{b}^{(i)}\| \leq 
        \frac{\rho^2}{2\cdot \lambda\cdot \mathrm{diam}(\mathcal{K}) \cdot i},$
    \end{enumerate}
    then the average iterate $\hat{\vec{x}} = \frac{1}{T} \sum_{t=1}^T \vec{x}^{(t)}$
    satisfies 
    $\mathbb{E}[f(\hat{\vec{x}}) - f(\vec{x}_*)] 
    \leq \frac{\rho^2}{\lambda T}(1
    + \log(T))$. 
\end{lemma}
\begin{proof}
    See Appendix \ref{app:bias_sgd_proof}.
\end{proof}

\subsection{Proof of Theorem \ref{thm:knownIndex}}
\label{sec:known:main:proof}
We are now ready to combine the results of the previous sections into a recovery
guarantee for $\bm{W}^* = \{\vec{\vec{w}}_i\}_{i = 1}^k$. In particular, we will apply Lemma \ref{lemma:sgd}
to show that Algorithm \ref{alg:psgd} converges to an average iterate $\hat{\vec{W}}$ 
that is close to $\bm{W}^*$.
First, observe that the norm of the gradient estimates outputted by Lemma \ref{app:lem:samp_trunc}
are bounded in norm by
\begin{align*}
    \mathbb{E}\left[\|\vec{v}\|_2^2\right] \leq \mathbb{E}\left[
        \sum_{j=1}^k \left\| \nabla_{\vec{w}_j} \ell(\bm{W}; (\vec{x}, i, y)) \right\|^2
    \right] + \beta,
\end{align*}
where $\beta$ is as in Lemma \ref{lem:samplingGradient}.
Our bounds on the norm of the weights and covariates directly implies  
\[
    \mathbb{E}\left[\|\vec{v}\|_2^2\right] \in O\lr{k\cdot \text{poly}(B, C)} + \beta.
\] 
Next, Lemma \ref{lem:strongConcavity} guarantees that $f$ in Lemma \ref{lemma:sgd} is 
strongly convex with $\lambda = \alpha/(\sigma k)$.
Finally, Lemma \ref{lem:samp_trunc} ensures access to gradients with appropriately
bounded bias (i.e., satisfying assumptions (i) and (iii) in Lemma \ref{lemma:sgd}) 
in $\text{poly}(k, B, C, T, 1/\alpha, \sigma^2, 1/\sigma^2)$-time.
We are thus free to apply Lemma \ref{lemma:sgd} to our
problem---after averaging $T$ steps of projected stochastic gradient
descent, we will find $\bm{W} = \{\vec{w}_j\}_{j=1}^k$ such that
\begin{align}
    \label{eq:sgd_bound}
    \mathbb{E}[\overline{\ell}(\bm{W})] - \overline{\ell}(\bm{W}^*) 
        \leq \frac{\sigma^2 \cdot k^2\cdot \poly(B, C)}{2\cdot \alpha\cdot T} (1 + \log(T)).
\end{align}
An application of Markov's inequality shows that, with probability at least $1 -
\delta$, 
\[
    \overline{\ell}(\bm{W}) - \overline{\ell}(\bm{W}^*)
        \leq \frac{\sigma^2 \cdot k^2\cdot \poly(B, C, 1/\delta)}{2\cdot \alpha\cdot T} (1 + \log(T)).
\]
Thus, we can condition on the event in \eqref{eq:sgd_bound} while only losing a
factor of $1 - \delta$ in success probability. Finally, a parameter-space
recovery bound follows from another application of convexity:
\[
    \left\|\bm{W} - \bm{W}^* \right\|_F
        \leq \frac{\sigma^4 \cdot k^3\cdot \poly(B, C, 1/\delta)}{\alpha^2\cdot T} (1 + \log(T)).
\]

    \section{Parameter Estimation for the Unknown-index Setting} \label{sec:unknown}
    In this section, we establish Theorems \ref{thm:no_index} and \ref{thm:no_index:2}, our main results for parameter recovery from a self-selection model with unknown indices (Definition~\ref{defn:index_unobserved}). Note that in this setting, even information theoretic identifiability of the parameters (i.e. parameter identification given access to infinite samples from the model) is not known. Hence, we start our discussion with a simple identifiabilty proof in the information theoretic setting in the limit of infinite samples, Theorem~\ref{thm:identifiability_no_index}, in Subsection~\ref{ssec:identifiability}. We then expand on these ideas in Subsections \ref{ssec:identifying_subspace} and \ref{ssec:est_low_dim} to prove Theorem~\ref{thm:no_index}. A naive adaptation of our identifiability proof results in a runtime and sample complexity scaling exponentially in the dimension of the input points. Therefore, in Subsection~\ref{ssec:identifying_subspace}, we show how one can efficiently identify a $k$-dimensional containing the span of the weight vectors, $\{\vec{w}^*_i\}_{i = 1}^k$. While similar approaches based on effective subspace identification have also been employed for other statistical learning tasks such as that of learning mixtures of well-separated gaussians \citep{vempalawang}, our analysis is significantly more intricate as the moments of the distributions under consideration do not have an obvious closed-form expression. Having identified a suitable low-dimensional subspace, we then carry out a finite sample analysis of our information theoretic idenitifiability proof from Subsection~\ref{thm:identifiability_no_index} in Subsection~\ref{ssec:est_low_dim} to prove Theorem~\ref{thm:no_index}. Finally, in the special case where $k = 2$, we describe a procedure which enables parameter recovery with sample complexity and runtime scaling as $\mathrm{poly} (1 / \eps)$ improving on the $\exp (1 / \eps)$ sample complexity and runtime from Theorem~\ref{thm:no_index}, proving Theorem~\ref{thm:no_index:2} and concluding the section.

\subsection{Identifiability with Unknown Indices}
\label{ssec:identifiability}
Here, we establish the information theoretic identifiability of the self-selection model with unknown indices. Recall, that we receive samples generated according to $y^{(i)} = \max_{j \in [k]} \vec{w}_j^\top \vec{x}^{(i)} + \eta_j^{(i)}$ where $\vec{x}^{(i)} \thicksim \mathcal{N} (0, \vec{I})$ and $\eta^{(i)}_j \overset{iid}{\thicksim} \mathcal{N} (0, 1)$. We now establish the following theorem:

\unknownidentify*

\begin{proof}
    Our proof will be based on an inductive argument on the number of components, $k$. We will use a peeling argument to reduce the parameter recovery problem with $k$ components to one with $k - 1$ components. The base case when $k = 1$, reduces to standard linear regression where, for example, $\E [\vec{x}^{(i)} \cdot y^{(i)}] = \vec{w}_1$ suffices. For the inductive argument, suppose $k > 1$ and consider the following function:
    \begin{equation*}
        \forall \vec{v} \in \R^d, \norm{\vec{v}} = 1: F(\vec{v}) = \lim_{\text{Even } p \to \infty} \lim_{\gamma \to 0} \lr{\frac{\E \sqlr{y^p \mid \norm{\pproj{\vec{v}} \vec{x}} \leq \gamma}}{(p-1)!!}}^{1 / p},
    \end{equation*}
    where $\pproj{\vec{v}}$ is the projection matrix orthogonal to the direction of $\vec{v}$. We will now show that the above function is well defined for all $\norm{\vec{v}} = 1$. The conditional moments may be evaluated with access to the distribution function $\Phi_{\vec{W}}$. Defining $j^* = \argmax_{j \in [k]} \abs{\vec{v}^\top \vec{w}_j}$ and $\sigma_j = \abs{\vec{v}^\top \vec{w}_{j}}$, we now lower bound the conditional moment:
    \begin{align*}
        \E [y^p \mid \norm{\pproj{\vec{v}} \vec{x}} \leq \gamma] &\geq \E [y^p \cdot \bm{1} \blr{\vec{w}_{j^*}^\top \vec{x} + \eta_{j^*}\geq 0} \mid \norm{\pproj{\vec{v}} \vec{x}} \leq \gamma] \geq \frac{1}{2} \cdot \E [y_{j^*}^p\mid \norm{\pproj{\vec{v}} \vec{x}} \leq \gamma] \\
        &\geq \frac{1}{2} \cdot \E [(\vec{w}_{j^*}^\top \proj{\vec{v}} (\vec{x}) + \eta_{j^*} + \vec{w}_{j^*}^\top \pproj{\vec{v}} (\vec{x}))^p\mid \norm{\pproj{\vec{v}} \vec{x}} \leq \gamma] \\
        &\geq \frac{1}{2} \cdot \E \sqlr{\sum_{l = 0}^{p / 2} \binom{p}{2l} (\vec{w}_{j^*}^\top \proj{\vec{v}} (\vec{x}) + \eta_{j^*})^{p - 2l} (\vec{w}_{j^*}^\top \pproj{\vec{v}} (\vec{x}))^{2l} \mid \norm{\pproj{\vec{v}} \vec{x}} \leq \gamma} \\
        &\geq \frac{1}{2} \cdot (p - 1)!! \cdot (\sigma_{j^*}^2 + 1)^{p / 2}.
    \end{align*}
    Through a similar computation, we obtain an upper bound on the conditional moment:
    \begin{align*}
        &\E [y^p \mid \norm{\pproj{\vec{v}} \vec{x}} \leq \gamma] \\
        &\leq \sum_{j = 1}^k \E [(\vec{w}_j^\top \vec{x} + \eta_j)^p \mid \norm{\pproj{\vec{v}} \vec{x}} \leq \gamma] \\
        &= \sum_{j = 1}^k \E [(\vec{w}_j^\top \proj{\vec{v}} \vec{x} + \eta_j + \vec{w}_j^\top \pproj{\vec{v}} \vec{x})^p\mid \norm{\pproj{\vec{v}} \vec{x}} \leq \gamma] \\
        &= \sum_{j = 1}^k \E \sqlr{\sum_{l = 0}^{p / 2} \binom{p}{2l} (\vec{w}_{j}^\top \proj{\vec{v}} (\vec{x}) + \eta_{j})^{p - 2l}(\vec{w}_{j}^\top \pproj{\vec{v}} (\vec{x}))^{2l} \mid \norm{\pproj{\vec{v}} \vec{x}} \leq \gamma} \\
        &\leq \sum_{j = 1}^k (p - 1)!! \cdot (\sigma_j^2 + 1)^{p / 2} + \sum_{l = 1}^{p / 2} \binom{p}{2l} \E \sqlr{(\vec{w}_{j}^\top \proj{\vec{v}} (\vec{x}) + \eta_{j})^{p - 2l} (\vec{w}_{j}^\top \pproj{\vec{v}} (\vec{x}))^{2l} \mid \norm{\pproj{\vec{v}} \vec{x}} \leq \gamma} \\
        &\leq k \cdot (p - 1)!! \cdot (\sigma_{j^*}^2 + 1)^{p / 2} + \sum_{j = 1}^k \sum_{l = 1}^{p / 2} \binom{p}{2l} \E \sqlr{(\vec{w}_{j}^\top \proj{\vec{v}} (\vec{x}) + \eta_{j})^{p - 2l} (\gamma \cdot \norm{\vec{w}_j})^{2l} \mid \norm{\pproj{\vec{v}} \vec{x}} \leq \gamma}.
    \end{align*}
    From the previous two displays, we get by taking $p^{th}$ roots and taking the limit as $\gamma \to 0$:
    \begin{equation*}
        \lr{\frac{1}{2}}^{1 / p}\cdot \sqrt{(\sigma_{j^*}^2 + 1)} \leq \lim_{\gamma \to 0} \lr{\frac{\E [y^p \mid \norm{\pproj{\vec{v}} \vec{x}} \leq \gamma]}{(p - 1)!!}}^{1 / p} \leq k^{1 / p}\cdot \sqrt{(\sigma_{j^*}^2 + 1)}.
    \end{equation*}
    Taking $p \to \infty$, we obtain:
    \begin{equation}
        \label{eq:f_def}
        F(\vec{v}) = \sqrt{\max_{j} \abs{\vec{w}_j^\top \vec{v}}^2 + 1}.
    \end{equation}
    Now, let $\vec{v}^*$ be such that $\vec{v}^* = \argmax_{\norm{\vec{v}} = 1} F(\vec{v})$. From \eqref{eq:f_def}, we get that $\vec{v}^* = \pm \vec{w}_{j^*}$ for some $j^* \in [k]$ satisfying $j^* = \argmax_{j} \norm{\vec{w}_j}$. Furthermore, $\sigma^* \coloneqq \norm{\vec{w}_{j^*}} = \sqrt{F(\vec{v}^*)^2 - 1}$. To identify the correct sign, consider the random variable $(\vec{x}, y - \vec{x}^\top (\sigma^* \vec{v}^*))$. Note that this is a max-selection model with parameter set $\{\vec{w}_j - \sigma^* \vec{v}^*\}_{j \in [k]}$ with associated function $\wt{F}$ defined analogously to $F$. Now, we have the following two cases:
    \begin{itemize}
        \item[] \textbf{Case 1: } $\wt{F} (\vec{v}^*) = \sqrt{4 \cdot (\sigma^*)^2 + 1}$. In this case, there exists $\vec{w} \in \vec{W}$ with $\vec{w} = - \sigma^* \vec{v}^*$
        \item[] \textbf{Case 2: } $\wt{F} (\vec{v}^*) < \sqrt{4 \cdot (\sigma^*)^2 + 1}$. In this case, we must have $\vec{w}_{j^*} = \sigma^* \vec{v}^*$.
    \end{itemize}
    In either case, we identify a single $\vec{w} \in \vec{W}$. To complete the reduction, note that $(\vec{x}, y - \vec{w}^\top \vec{x})$ is a self-selection model with parameter set $\vec{W}^\prime = \{\vec{w}_j - \vec{w}\}_{j \in [k]}$. Defining the $k - 1$ sized point set $\vec{W}^\dagger = \{\vec{w}_j - \vec{w}\}_{\substack{j \in [k], j \neq j^*}}$, we note the relationship between the distribution functions $\Phi_{\vec{W}^\prime}$ and $\Phi_{\vec{W}^\dagger}$ for all $S \subset \R^d$ with $\P \blr{x \in S} \neq 0$:
    \begin{equation*}
        \forall t \in \R: \Phi_{W^\dagger} (S \times (-\infty, t]) = \frac{\Phi_{W^\prime} (S \times (-\infty, t])}{\Phi (t)}.
    \end{equation*}
    Hence, the distribution function $\Phi_{\vec{W}^\dagger}$ is a function of the distribution function of $\Phi_{\vec{W}^\prime}$ which is in turn a function of the distribution function of $\Phi_{\vec{W}}$. From our induction hypothesis, we have that $\vec{W}^\dagger$ is identifiable from $\Phi_{\vec{W}^\dagger}$ and consequently, from $\Phi_{\vec{W}}$. The proof of the inductive step now follows from the observation that $\vec{W} = \{\vec{z} + \vec{w}\}_{\vec{z} \in \vec{W}^\dagger} \cup \{\vec{w}\}$.
\end{proof}
\subsection{Finding an $O(k)$-subspace}
\label{ssec:identifying_subspace}
\newcommand{\svdmat}{\bm{M}}
We now move towards a finite-sample estimation algorithm for the unknown-index
case. The first step in our approach is an algorithm for approximately
identifying a size-$k$ subspace that has high overlap with $\text{span}(\vec{w}_1,
\ldots, \vec{w}_k)$. 
In order to estimate the subspace, we will consider the matrix 
$\svdmat = \mathbb{E}\lr{\max(0, y)^2\cdot \vec{x}\vec{x}^\top}$. The following Lemma
shows that the top $k$ eigenvectors of $\svdmat$ capture the span of the weight
vectors $\vec{w}_k$:

\begin{lemma}[Weighted covariance]
    \label{lemma:svd_cov}
    Consider the matrix $\svdmat = \mathbb{E}\lr{\max(0, y)^2 \cdot
    \vec{x}\vec{x}^\top}$, and let $$p_i = \mathbb{P}\lr{\left\{
    i = \argmax_{j \in [k]} \vec{w}_j^\top \vec{x} + \eta_j\right\} 
    \text{ and } \left\{ \vec{w}_i^\top \vec{x} + \eta_i > 0\right\}}.$$ Then,
    if $v$ is a unit vector,
    \begin{align*}
        \vec{v} \in \mathrm{span}(\vec{w}_1,\ldots, \vec{w}_k) &\implies v^\top \svdmat v \geq  \mathbb{E}\lr{\max(0, y)^2} 
        + 2\sum_{i=1}^{k} p_i \cdot (v^\top \vec{w}_i)^2 \\
        \vec{v} \in \mathrm{null}(\vec{w}_1, \dots, \vec{w}_k) &\implies \vec{v}^\top \svdmat \vec{v} = \mathbb{E}\lr{\max(0, y)^2}.
    \end{align*}
\end{lemma}
\begin{proof}
    First, for all unit vectors $\vec{v} \in \text{null}({\vec{w}_1,\ldots, \vec{w}_k})$,
independence of Gaussian directions implies:
\begin{align*}
    \vec{v}^\top \svdmat \vec{v} &= \mathbb{E}\lr{\max(0, y)^2 \cdot (\vec{v}^\top \vec{x})^{2}} 
              = \mathbb{E}\lr{\max(0, y)^2} \cdot \mathbb{E}\lr{(\vec{v}^\top \vec{x})^{2}} 
              = \mathbb{E}\lr{\max(0, y)^2}.
\end{align*}

\noindent We now consider a unit vector $\vec{u}$ in the span of $\vec{w}_i$ given by $\vec{u} =
\sum_{i=1}^k a_i \vec{w}_i$.
In order to compute $\vec{u}^\top \svdmat  \vec{u}$, we will need some
preliminary results. First, we recall the smooth maximum function $F_\beta$: 
\[
    F_\beta: \mathbb{R}^k \to \mathbb{R} \qquad \text{where} \qquad F(\vec{W}) \coloneqq \beta^{-1} \log\lr{\sum_{j=1}^k \exp(\beta \cdot W_j)}
\]

\noindent The key property of the smooth maximum function is that, for any $\vec{W}
\in \mathbb{R}^k$,
\begin{equation}
    \label{eq:smoothmax_bounds}
    0 \leq F_\beta (\vec{W}) - \max_{j \in [k]} W_j \leq \beta^{-1} \log(k),
\end{equation}
and in particular $\lim_{\beta \to \infty} F_\beta = \max$. 
Partial derivatives of $F_\beta$ are given by the following Lemma:
\begin{lemma}[Derivatives of the smooth maximum function]
    \label{lemma:derivs}
   The partial derivatives of the smooth maximum functions with smoothing
   parameter $\beta$ are given by:
   \begin{align*}
       \partial_a F_\beta(z) = \pi_a(z) \qquad \text{and} \qquad 
            \partial_a \partial_b F_\beta(z) = \beta\cdot \lr{\bm{1}_{a = b} \cdot \pi_a(z) - \pi_a(z)\pi_b(z)}, 
   \end{align*}
   where $\pi_a(z) = \exp\{\beta z_a\}/\sum_{\ell=1}^k \exp\{\beta z_\ell\}$.
   Observe that $\pi_a(z) > 0$ for all $a \in [k]$, and $\sum_{a=1}^{k} \pi_a(z) = 1$.
\end{lemma}
\noindent We next state the main identity behind our proof, due to Stein (see, e.g., A.6 of \citep{Talagrand2003SpinG} for a proof):
\begin{lemma}[Stein's Identity]
    \label{lemma:stein}
    Let $\vec{W} = (W_1, \ldots W_p)$ be a centered Gaussian random vector in
    $\mathbb{R}^p$. Let $f : \mathbb{R}^p \to \mathbb{R}$ be a $C_1$-function
    such that $\mathbb{E}[|\partial_j f(\vec{W})|] < \infty$ for all $j \in [p]$.
    Then, for every $j \in [p]$,
    \[
        \mathbb{E}\lr{W_j \cdot f(\vec{W})} = \sum_{a=1}^p \mathbb{E}[W_a\cdot W_j] \mathbb{E}[\partial_a f(\vec{W})].
    \]
\end{lemma}
\noindent Now, by convention let $\vec{w}_{k+1} = \bm{0}$ and $\eta_{k + 1} = 0$. We
construct $V \in \mathbb{R}^{2k}$ as 
\[
   V_i = \begin{cases}
        (\vec{w}_i^\top \vec{x}) &\text{for } 1 \leq i \leq k+1 \\
        \eta_{i - k} &\text{for } k+2 \leq i \leq 2(k+1). \\
   \end{cases} 
\]
For convenience, we next define $\vec{Z} \in \mathbb{R}^k$ such that $Z_i = W_i +
W_{k+i} = y_i$, and $f_i: \mathbb{R}^{2k} \to \mathbb{R}$ as
\[
    f_i(V) = V_i \cdot F_\beta(\vec{Z})^2.
\]
\noindent 
By construction, $\max(0, y) = \max(\vec{Z}) \geq 0$.
Applying Equation \eqref{eq:smoothmax_bounds},
\begin{align*}
    \mathbb{E}&\lr{(\vec{u}^\top \vec{x})^2 y^2} = 
    \mathbb{E}\lr{(\vec{u}^\top \vec{x})^2 \max\left\{\vec{w}_1^\top \vec{x} + \eta_1, \ldots, \vec{w}_{k+1}^\top \vec{x} + \eta_{k+1}\right\}^2} \\
    &= \lim_{\beta \to \infty}
    \mathbb{E}\lr{(\vec{u}^\top \vec{x})^2 \lr{F_{\beta}(\vec{w}_1^\top \vec{x} + \eta_1, \ldots, \vec{w}_k^\top \vec{x} + \eta_k) - \beta^{-1} \log(k)}^2} \\
    &= \lim_{\beta \to \infty} \overbrace{
        \mathbb{E}\lr{\lr{\sum_{i=1}^k a_i \vec{w}_i^\top \vec{x}}^2 F_{\beta}\lr{\vec{Z}}^2}
    }^{S_1} 
    - \overbrace{
        \frac{2}{\beta} \cdot \mathbb{E}\lr{\lr{\sum_{i=1}^k a_i \vec{w}_i^\top \vec{x}}^2 F_{\beta}\lr{\vec{Z}}} \log(k) 
    }^{S_2} 
    + \overbrace{\frac{1}{\beta^{2}} \log^2(k)}^{S_3},
\end{align*}
where we have used in $S_3$ that $\vec{u}$ is a unit vector. Now, as
$\beta \to \infty$, boundedness of the weight vectors and of
$\mathbb{E}\lr{\max(0, y)^2}$ imply $S_2, S_3 \to 0$. Thus, we can focus on the
first term (i.e., $S_1$):
\begin{align}
    \label{eq:chungus}
    S_1 &= \mathbb{E}\lr{\sum_{i=1}^k \sum_{j=1}^k a_i \cdot a_j \cdot V_j \cdot f_i(\vec{V})} 
    = \sum_{i=1}^k \sum_{j=1}^k a_ia_j \lr{\sum_{a=1}^{2k+2} \mathbb{E}[V_j \cdot V_a] \cdot \mathbb{E}[\partial_a f_i(\vec{V})]}.
\end{align}
Now, the $\eta_j$ terms are all i.i.d. $\mathcal{N}(0, 1)$, and
so $\mathbb{E}[V_a \cdot V_j] = 0$ for all $a > k + 1$. Furthermore, when $a = k
+ 1$, we have $V_a = \vec{w}_{k+1}^\top \vec{x} = 0$ by construction. For $a \in [k]$, we
have $\mathbb{E}[V_a \cdot V_j] = \mathbb{E}[\vec{w}_a^\top (\vec{x}\vec{x}^\top) \vec{w}_i] = \vec{w}_a^\top
\vec{w}_j$. Also, our definition of $f_i(\vec{V})$ implies that
\begin{align*}
    \partial_a f_i(\vec{V}) = 2 V_i \cdot F_\beta(\vec{Z}) \cdot \partial_a F_\beta(\vec{Z}) + \bm{1}_{a = i} \cdot F_\beta(\vec{Z})^2.
\end{align*}
Combining the aforementioned properties with Lemma \ref{lemma:stein}, we have that 
\begin{align}
    \label{eq:chungus2}
    \sum_{a=1}^{2k+2} \mathbb{E}[V_j \cdot V_a] \cdot \mathbb{E}[\partial_a f_i(\vec{V})] 
    &= \vec{w}_j^\top \vec{w}_i\cdot \mathbb{E}\left[F_\beta(\vec{Z})^2\right] + 
    2 \sum_{a=1}^{k} \vec{w}_j^\top \vec{w}_a \cdot \mathbb{E}\left[(\vec{w}_i^\top \vec{x}) F_\beta(\vec{Z}) \cdot \partial_a F_\beta(\vec{Z}) \right].
\end{align}
We now apply Stein's Identity (Lemma \ref{lemma:stein}) once more, this time
using $g(\vec{V}) \coloneqq F_\beta(\vec{Z})\cdot \partial_a F_\beta(\vec{Z})$, i.e.,
\begin{align*}
   \mathbb{E}\left[(\vec{w}_i^\top \vec{x}) F_\beta(\vec{Z}) \cdot \partial_a F_\beta(\vec{Z}) \right] 
   &= \sum_{b=1}^{2k+2} \mathbb{E}[V_i \cdot V_b] \cdot \mathbb{E}[F_\beta(\vec{Z}) \cdot \partial_a\partial_b F_\beta(\vec{Z}) + \partial_a F_\beta(\vec{Z}) \partial_b F_\beta(\vec{Z})].
\end{align*}
Again, $\mathbb{E}[V_b \cdot V_i] = 0$ for $b > k$ and $\mathbb{E}[V_b \cdot V_i] = \vec{w}_i^\top \vec{w}_b$ for $b \in [k]$, so:
\begin{align*}
   \mathbb{E}\big[(\vec{w}_i^\top \vec{x}) \cdot & \partial_a F_\beta(\vec{Z}) \big] 
    = \sum_{b=1}^{k} (\vec{w}_i^\top \vec{w}_b) \cdot \mathbb{E}[F_\beta(\vec{Z}) \cdot \partial_a\partial_b F_\beta(\vec{Z}) + \partial_a F_\beta(\vec{Z}) \partial_b F_\beta(\vec{Z})] \\
   & = \sum_{b=1}^{k} (\vec{w}_i^\top \vec{w}_b) \cdot \mathbb{E}[\beta \cdot F_\beta(\vec{Z})\lr{\bm{1}_{a = b} \cdot \pi_a(\vec{Z}) - \pi_a(\vec{Z})\pi_b(\vec{Z})} 
   + \pi_a(\vec{Z}) \pi_b(\vec{Z})]
\end{align*}
Thus, we can rewrite \eqref{eq:chungus2} as
\begin{align*}
    \sum_{a=1}^{2k+2} \mathbb{E}[V_j \cdot V_a] \cdot & \mathbb{E}[\partial_a f_i(\vec{V})] 
    = (\vec{w}_i^\top \vec{w}_j) \mathbb{E}\left[F_\beta(\vec{Z})^2 \right] 
    + 2 \sum_{a=1}^k \sum_{b=1}^k \lr{\vec{w}_i^\top \vec{w}_b} \lr{\vec{w}_j^\top \vec{w}_a} \mathbb{E}[\pi_a(\vec{Z}) \pi_b(\vec{Z})] \\
    &+ 2 \beta \sum_{a=1}^{k} \sum_{b=1}^k \lr{\vec{w}_i^\top \vec{w}_b} \lr{\vec{w}_j^\top \vec{w}_a} \cdot \mathbb{E}[F_\beta(\vec{Z}) \cdot \lr{\bm{1}_{a = b} \cdot \pi_a(\vec{Z}) - \pi_a(\vec{Z})\pi_b(\vec{Z})}],
\end{align*}
and summing over $i$ and $j$ as in \eqref{eq:chungus} yields:
\begin{align*}
    S_1 &= \sum_{i,j=1}^k a_i (\vec{w}_i^\top \vec{w}_j) a_j \mathbb{E}[F_\beta(\vec{Z})^2] 
    + 2 \sum_{a,b=1}^{k} (\vec{u}^\top \vec{w}_b)(\vec{u}^\top \vec{w}_a) \mathbb{E}[\pi_a(\vec{Z})\pi_b(\vec{Z})] \\
    &\qquad + 2\beta \sum_{a,b=1}^{k} (\vec{u}^\top \vec{w}_b)(\vec{u}^\top \vec{w}_a) \mathbb{E}[F_\beta(\vec{Z})\cdot (\bm{1}_{a = b} \pi_a(\vec{Z}) - \pi_a(\vec{Z})\pi_b(\vec{Z}))] \\
    &= \mathbb{E}\left[F_\beta(\vec{Z})^2
    + 2 \beta \cdot F_\beta(\vec{Z}) \cdot \text{Var}_{i \sim \text{Cat}(\pi(\vec{Z}))}\left[\vec{u}^\top \vec{w}_i\right] 
    + 2 \cdot \mathbb{E}_{i \sim \text{Cat}(\pi(\vec{Z}))} \left[\vec{u}^\top \vec{w}_i\right]^2
    \right],
\end{align*}
where the outer expectation is over the data (i.e., $\vec{V}$ and correspondingly $\vec{Z}$), and $i \sim \text{Cat}(\pi(\vec{Z}))$ is a multinomial distribution over $[k]$ where index $i$ is drawn with probability $\pi_i(\vec{Z})$. Now, noting $F_\beta (\vec{Z}) \geq 0$, we get:
\begin{align*}
    S_1 &= \mathbb{E}\left[F_\beta(\vec{Z})^2 + 2 \beta \cdot F_\beta(\vec{Z}) \cdot \text{Var}_{i \sim \text{Cat}(\pi(\vec{Z}))}\left[\vec{u}^\top \vec{w}_i\right] + 2 \cdot \mathbb{E}_{i \sim \text{Cat}(\pi(\vec{Z}))} \left[\vec{u}^\top \vec{w}_i\right]^2\right] \\
    &\geq \mathbb{E}\left[F_\beta(\vec{Z})^2 + 2 \cdot \mathbb{E}_{i \sim \text{Cat}(\pi(\vec{Z}))} \left[\vec{u}^\top \vec{w}_i\right]^2\right]
\end{align*}
and hence for $i^* = \argmax_i Z_i$ and recall, $p_i = \mathbb{P}\lr{\left\{
    i = \argmax_{j \in [k]} \vec{w}_j^\top \vec{x} + \eta_j\right\} 
    \text{ and } \left\{ \vec{w}_i^\top \vec{x} + \eta_i > 0\right\}}$,
\begin{equation*}
    \lim_{\beta \to \infty} S_1 \geq \E [\max(0, y)^2 + 2 \cdot (\vec{w}_{i^*}^\top \vec{u})^2] = \mathbb{E}[\max(0, y)^2] + 2 \vec{u}^\top \lr{\sum_{i=1}^k p_i\cdot \vec{w}_i \vec{w}_i^\top} \vec{u}
\end{equation*}

which concludes the proof.
\end{proof}

We will now establish lower bounds on the quantities, $p_i$, from the previous lemma which will aid in establishing a quantitative spectral gap. 

\begin{lemma}
    \label{lem:pi_bnd}
    Under Assumption~\ref{as:no_index}, we have for some absolute constant $C > 0$:
    \begin{equation*}
        \forall i \in [k]: p_i \geq \exp \lbrb{-Ck \log ((2B) / \Delta)}.
    \end{equation*}
\end{lemma}
\begin{proof}
    Fix $i \in [k]$ and letting $\mathcal{V}$ denote the subspace spanned by $\{\vec{w}_j\}_{j \in [k]}$, define the sets:
    \begin{align*}
        S_i &\coloneqq \left\{\vec{x} \in \R^d: \norm{\mathcal{P}_{\vec{w}_i}^\perp \mathcal{P}_{\mathcal{V}} \vec{x}} \leq \frac{\Delta}{2B}\right\} \\
        T_i &\coloneqq \left\{\vec{x} \in \R^d: \frac{\vec{w}_i^\top \vec{x}}{\norm{\vec{w}_i}} \geq 2 \right\}.
    \end{align*}
    We have from the fact that $\vec{x} \thicksim \mathcal{N} (0, I)$ and the independence of the events $S_i$ and $T_i$:
    \begin{align*}
        \P \left\{ \vec{x} \in S_i \cap T_i \right\} &= \P \left\{ \vec{x} \in S_i \right\} \cdot \P \left\{ \vec{x} \in T_i \right\} \geq c \cdot \P \left\{ \vec{x} \in S_i \right\} \\
        &\geq c \cdot \left(\frac{\Delta}{2B}\right)^{k - 1} \cdot \frac{1}{2^{(k - 1) / 2} \Gamma ((k - 1) / 2 + 1)} \cdot \exp \left\{ - \frac{\Delta^2}{8 B^2} \right\} \\
        &\geq \exp \lbrb{- C k \log ((2 B) / \Delta)}
    \end{align*}
    for some absolute constant $C > 0$. Furthermore, we have for any $\vec{x} \in S_i \cap T_i$ and $j \in [k] \setminus i$:
    \begin{align*}
        \vec{w}_j^\top \vec{x}  &=  \vec{w}_j^\top (\mathcal{P}^\perp_{\vec{w}_i} \mathcal{P}_{\mathcal{V}} \vec{x} + \mathcal{P}_{\vec{w}_i} \vec{x}) \leq \norm{\vec{w}_j} \cdot \norm{\mathcal{P}^\perp_{\vec{w}_i} \mathcal{P}_{\mathcal{V}} \vec{x}} + \frac{\vec{w}_j^\top \vec{w}_i}{\norm{\vec{w}_i}} \cdot \frac{\vec{w}_i^\top \vec{x}}{\norm{\vec{w}_i}} \\
        &\leq B \cdot \frac{\Delta}{2B} + \frac{\vec{w}_i^\top \vec{x}}{\norm{\vec{w}_i}} \cdot (\norm{\vec{w}_i} - \Delta) \leq \Delta \left(\frac{1}{2} -  \frac{\vec{w}_i^\top \vec{x}}{\norm{\vec{w}_i}}\right) + \vec{w}_i^\top \vec{x} \leq \vec{w}_i^\top \vec{x} - \Delta.
    \end{align*}
    Finally, noting that $\vec{w}_i^\top \vec{x} > 0$ for any $\vec{x} \in S_i$:
    \begin{align*}
        p_i &\coloneqq \mathbb{P}\lr{\left\{i = \argmax_{j \in [k]} \vec{w}_j^\top \vec{x} + \eta_j\right\} \text{ and } \left\{ \vec{w}_i^\top \vec{x} + \eta_i > 0\right\}} \\
        &\geq \P \lprp{\lbrb{i = \argmax_{j \in [k]} \vec{w}_j^\top \vec{x}} \cap \lbrb{\vec{w}_i^\top \vec{x} > 0} \cap \lbrb{i = \argmax_{j\ in [k]} \eta_j} \cap \lbrb{\eta_i \geq 0}} \\
        &= \P \lprp{\lbrb{i = \argmax_{j \in [k]} \vec{w}_j^\top \vec{x}} \text{ and } \lbrb{\vec{w}_i^\top \vec{x} > 0}} \cdot \P \lprp{\lbrb{i = \argmax_{j\ \in [k]} \eta_j} \cap \lbrb{\eta_i \geq 0}} \\
        &\geq \P \lprp{\vec{x} \in S_i \cap T_i} \cdot \P \lprp{\eta_i \geq 0} \cdot \P \lprp{\lbrb{i = \argmax_{j \in [k]} \eta_j} \big| \eta_i > 0} \\
        &\geq \P \lprp{\vec{x} \in S_i \cap T_i} \cdot \P \lprp{\eta_i \geq 0} \cdot \frac{1}{2} \cdot \frac{1}{k} \geq \exp\lbrb{- Ck \log ((2B) / \Delta)}
    \end{align*}
    from our previous displays, thus concluding the proof of the lemma.
\end{proof}

\noindent We will now use Lemma~\ref{lem:pi_bnd} to establish a quantitative spectral gap bound on the principal subspaces approximately containing the $\vec{w}_i$. 
\begin{lemma}
    \label{lem:ui_spec_gap}
    Let $\epsilon \in (0, 1/2)$ and $\lbrb{\vec{u}_1}_{i \in [l]}$ be all the
    singular vectors of $\svdmat$ whose corresponding singular values
    $\lbrb{\sigma_i}_{i \in [l]}$ satisfy:
    \begin{equation*}
        \forall i \in [l]: \sigma_i \geq \E \lsrs{\max (0, y)^2} + \min_{i \in [k]} \frac{p_i \epsilon^2}{2}.
    \end{equation*}
    Then, letting $\vec{U} = \mathrm{span} (\lbrb{\vec{u}_i}_{i \in [l]})$, we have:
    \begin{gather*}
        l \leq k,\qquad \forall i \in [k]: \frac{\norm{\vec{w}_i - \proj{\vec{U}} \vec{w}_i}}{\norm{\vec{w}_i}} \leq \epsilon \qquad \text{and} \qquad \norm{\svdmat} \leq 3kd (B^2 + 1).
    \end{gather*}
\end{lemma}
\begin{proof}
    The first claim (i.e., that $\ell \leq k$) follows from Lemma~\ref{lemma:svd_cov}. For the second, consider the contrary and suppose that for some $i \in [k]$,
    \begin{equation*}
        \frac{\norm{\vec{w}_i - \proj{\vec{U}} \vec{w}_i}}{\norm{\vec{w}_i}} > \epsilon.
    \end{equation*}
    Now, define the vector $\vec{v}$ as:
    \begin{equation*}
        \vec{v} \coloneqq \frac{\vec{w}_i - \proj{\vec{U}} \vec{w}_i}{\norm{\vec{w}_i - \proj{\vec{U}} \vec{w}_i}}.
    \end{equation*}
    We have $\vec{v} \perp \vec{U}$ and furthermore, Lemma~\ref{lemma:svd_cov} yields:
    \begin{equation*}
        \vec{v}^\top \svdmat \vec{v} \geq \E \lsrs{\max (0, y)^2} + p_i (\vec{v}^\top \vec{w}_i)^2 \geq \E \lsrs{\max (0, y)^2} + p_i \epsilon^2
    \end{equation*}
    yielding the contradiction, establishing the second claim. Finally, for the last claim, we have:
    \begin{align*}
        \norm{\svdmat} &\leq \mathbb{E} \lsrs{\max (0, y)^2 \norm{\vec{x}\vec{x}^\top}} = \mathbb{E} \lsrs{\max (0, y)^2 \norm{\vec{x}}^2} \leq \sum_{i = 1}^k \sum_{j = 1}^d \mathbb{E} \lsrs{(\vec{w}_i^\top \vec{x} + \eta_i)^2 \cdot (x_j)^2} \\
        &\leq \sum_{i = 1}^k \sum_{j = 1}^d \lprp{\mathbb{E} \lsrs{(\vec{w}_i^\top \vec{x} + \eta_i)^4} \cdot \mathbb{E}\lsrs{ (x_j)^4}}^{1 / 2} \leq 3 kd (B^2 + 1).
    \end{align*}
\end{proof}

\noindent We now combine the spectral gap shown in the previous Lemma with a
matrix concentration argument to argue that $k$-SVD on $\svdmat -
\mathbb{E}[y^2]\cdot \bm{I}$ (where $\svdmat$ is as defined in
Lemma \ref{lemma:svd_cov}) approximately recovers the span of the $\{\vec{\vec{w}_i}\}$. We use
as a primitive the following result about~$k$-SVD:
\begin{fact}[\citep{rokhlin2010randomized}]
    \label{fac:k_svd}
    Let $M \in \mathbb{R}^{k \times k}$, and let $\sigma_1 \geq \sigma_2,
    \geq \ldots \geq \sigma_d$ denote the non-zero singular values of M. For
    any $j \in [k-1]$, define the spectral gap $g_j = \sigma_j /
    \sigma_{j+1}$. Furthermore, suppose we have access to an oracle which
    computes $M\vec{v}$ for any $\vec{v} \in \mathbb{R}^k$ in time $R$. Then, for any 
    $\eta, \delta > 0$, there is an algorithm $\textsc{ApproxSVD}(M,
    \eta, \delta)$ which runs in time $\tilde{O}(\frac{j \cdot R}{\min(1, g_j - 1)} \cdot log(k/(\eta \delta))$
    and with probability at least $1 - \delta$ outputs $U \in R^{k\times j}$
    with orthonormal columns so that $\|U - U_j\|_2 < \eta$, where $U_k$ is
    the matrix whose columns are the top $j$ right singular vectors of M. 
\end{fact}

The combination of Lemma \ref{lemma:svd_cov} and Fact \ref{fac:k_svd} imply that it
suffices to show concentration of the average of a sequence $M^{\vec{x}_i,y_i}$. The
following Lemma helps us establish this concentration:
\begin{restatable}{lemma}{svdconc}
    \label{lem:svd_conc}
    Suppose we generate $n$ samples $(\vec{x}^{(l)}, y^{(l)}) \sim X, Y$ from the
    self-selected linear regression model, i.e., $\vec{x}^{(l)} \sim
    \mathcal{N}(0, \bm{I}_d)$, then $y^{(l)}_i = \bm{w}_i^\top \vec{x}^{(l)} +
    \mathcal{N}(0, 1)$ for vectors $\vec{\vec{w}_i} \in \mathbb{R}^d$ with $\|\vec{\vec{w}_i}\|_2 \leq B$, and $y^{(l)} = \max_{i \in [k]} y^{(l)}_i$. Define the
    empirical second-moment matrix 
    \begin{align*}
        \widehat{\svdmat} = \frac{1}{n}\sum_{l=1}^n \max(0, y^{(l)})^2 \cdot \vec{x}^{(l)} {\vec{x}^{(l)}}^\top.
    \end{align*}
    Fix any $\delta \in (0, 1)$. Then, if $n \geq \Omega(\max(1/\delta, d))$,
    with probability at least $1 - \delta$,
    \begin{align*}
        \left\|\widehat{\svdmat} - \mathbb{E}[\widehat{\svdmat} ]\right\|_2 \in O\lr{
            \poly(B)\cdot \frac{\log(kn)}{\sqrt{n}}
            \max\left\{
                \sqrt{\log(2/\delta)},\sqrt{d}
            \right\}
        }.
    \end{align*}
\end{restatable}
\begin{proof}
    See Appendix \ref{app:proof:svd_conc}.
\end{proof}
\noindent We can thus 
approximately identify the
relevant subspace in polynomial time.

\subsection{Estimating Parameters using the Low-Dimensional Subspace}
\label{ssec:est_low_dim}

Here we leverage the results of Subsections \ref{ssec:identifying_subspace} and
\ref{ssec:identifiability} to build an algorithm for estimating the weight
vectors $\{\vec{w}_j\}_{j = 1}^k$ under the separability Assumption
\ref{as:no_index} on the $\vec{w}_i$. Subsection~\ref{ssec:identifying_subspace}
allows us to effectively reduce the dimensionality of the problem down to $k$
dimensions. We will then adapt the identifiability argument from
Subsection~\ref{ssec:identifiability} in the $k$-dimensional subspace to
estimate the weight vectors and complete the proof of
Theorem~\ref{thm:no_index}, restated below:
\unknownind* 
We prove Theorem \ref{thm:no_index} in the remainder of this section.
From Lemmas~\ref{lem:ui_spec_gap} and \ref{lem:svd_conc}, we may assume access to a $k$-dimensional subspace $\hat{U}$ satisfying:
\begin{equation*}
    \forall i \in [k]: \frac{\norm{\vec{w}_i - \mathcal{P}_{\hat{U}} \vec{w}_i}}{\norm{\vec{w}_i}} \leq \lr{\frac{\Delta}{1024 B}}^{32}
\end{equation*}
where $\mathcal{P}_{\hat{U}}$ denotes the projection operator onto $\hat{U}$. Consider the set $\mathcal{G} = \{\vec{x}: \norm{\vec{x}} = 1,\vec{x} \in \hat{U}\}$ and let $\mathcal{H}$ be a $\gamma$-net over $\mathcal{G}$ with $\gamma = (\Delta / 1024 B)^{32}$. Note that we may assume $| \mathcal{H} | \leq (C / \gamma)^{k}$ as $\hat{U}$ is a $k$-dimensional subspace \cite[Proposition 4.2.12]{vershynin}. For $\vec{v} \in \mathcal{H}$ and $\rho = (\Delta / 1024 B)^{32}$, consider the following event:
\begin{equation*}
    A_{\vec{v}} = \{\norm{\pproj{v} \proj{\hat{U}} \vec{x}} \leq \rho\}.
\end{equation*}

\begin{claim}
    \label{clm:av_prob_lb}
    We have for any $0 < \rho \leq 1$:
    \begin{equation*}
        P_\rho \coloneqq \P \blr{A_{\vec{v}}} \geq \rho^{k - 1} \exp \blr{C k \log k}.
    \end{equation*}
\end{claim}
\begin{proof}
    We have from the fact that the cdf of a standard gaussian is decreasing in length that:
    \begin{equation}
        \label{eq:rho_prob}
        \P \blr{A_{\vec{v}}} \geq \frac{\rho^{k - 1}}{2^{(k - 1) / 2} \Gamma ((k - 1) / 2 + 1)} \cdot \exp \blr{- \frac{\rho^2}{2}} \geq \rho^{k - 1} \exp \blr{C k \log k}
    \end{equation}
    for some absolute constant $C > 0$. %
\end{proof}

As in the proof of Theorem~\ref{thm:identifiability_no_index}, we analyze the moments of $y$ conditioned on the event $A_{\vec{v}}$. We start by proving as before that the moments are upper and lower bounded by the moments of the individual components. 
\begin{claim}
    \label{clm:yl_const_bnd_finite_sample}
    We have for any even $l$:
    \begin{equation*}
        \frac{1}{8} \Psi (\vec{v}) \leq \E [y^l \mid A_{\vec{v}}] \leq k \Psi (\vec{v}) \text{ where } \Psi (\vec{v}) \coloneqq \max_{j \in [k]} \E [y_j^l \mid A_{\vec{v}}]
    \end{equation*}
\end{claim}
\begin{proof}
    For the upper bound, we have:
    \begin{equation*}
        \E [y^l \mid A_{\vec{v}}] \leq \sum_{j = 1}^k \E [y_j^l \mid A_{\vec{v}}].
    \end{equation*}
    For the lower bound, let $V$ be the orthogonal complement of $\vec{v}$ in the subspace $\hat{U}$; the projection operator onto $V$ is thus $\proj{\hat{U}}\pproj{\vec{v}}$. We have for any $j \in [k]$ by the independence of $\proj{V}\vec{x}$, $\pproj{V}\vec{x}$ and $\eta_j$:
    \begin{align*}
        &\E [y^l \mid A_{\vec{v}}] \\
        &\geq \E [y_j^l \bm{1} \blr{\vec{w}_j^\top \proj{V}\vec{x} \geq 0, \vec{w}_j^\top \pproj{V}\vec{x} \geq 0, \eta_j \geq 0} \mid A_{\vec{v}}] \\
        &= \E [(\vec{w}_j^\top \proj{V}\vec{x} + \vec{w}_j^\top \pproj{V}\vec{x} + \eta_j)^l \bm{1} \blr{\vec{w}_j^\top \proj{V}\vec{x} \geq 0, \vec{w}_j^\top \pproj{V}\vec{x} \geq 0, \eta_j \geq 0} \mid A_{\vec{v}}] \\
        &= \sum_{\substack{0 \leq p, q \leq l \\ p + q \leq l}} \binom{l}{p, q} \E [(\vec{w}_j^\top \proj{V}\vec{x})^p (\vec{w}_j^\top \pproj{V}\vec{x})^q \eta_j^{l - p - q} \bm{1} \blr{\vec{w}_j^\top \proj{V}\vec{x} \geq 0, \vec{w}_j^\top \pproj{V}\vec{x} \geq 0, \eta_j \geq 0} \mid A_{\vec{v}}] \\
        &= %
                \sum_{\substack{0 \leq p, q \leq l \\ p + q \leq l}} \binom{l}{p, q} \E [(\vec{w}_j^\top \proj{V}\vec{x})^p \bm{1} \blr{\vec{w}_j^\top \proj{V}\vec{x} \geq 0} \mid A_{\vec{v}}] \cdot \\
                &\qquad\qquad\qquad\qquad\qquad\qquad \E [(\vec{w}_j^\top \pproj{V}\vec{x})^q \bm{1} \blr{\vec{w}_j^\top \pproj{V}\vec{x} \geq 0} \mid A_{\vec{v}}] \cdot \E[\eta_j^{l - p - q} \bm{1} \blr{\eta_j \geq 0} \mid A_{\vec{v}}] \\
        &\geq \frac{1}{8} \sum_{\substack{0 \leq p, q \leq l \\ p + q \leq l}} \binom{l}{p, q} \cdot \E [(\vec{w}_j^\top \proj{V}\vec{x})^p \mid A_{\vec{v}}] \cdot \E [(\vec{w}_j^\top \pproj{V}\vec{x})^q \mid A_{\vec{v}}] \cdot \E[\eta_j^{l - p - q} \mid A_{\vec{v}}]  = \frac{1}{8} \cdot \E [y_j^l \mid A_{\vec{v}}]
    \end{align*}
    Defining $\Psi(\vec{v}) = \max_{j \in [k]} \E [y_j^l \mid A_{\vec{v}}]$, we have from the previous two displays:
    \begin{equation}
        \label{eq:orth_part_bnd}
        \frac{1}{8} \Psi (\vec{v}) \leq \E [y^l | A_{\vec{v}}] \leq k \Psi (\vec{v})
    \end{equation}
    concluding the proof of the claim.
\end{proof}

We will now prove upper and lower bounds on the quantity $\E [y_j^l \mid A_{\vec{v}}]$ for any $j \in l$. The claim will establish bounds that allow us to estimate $\max_{j} \abs{\vec{w}_j^\top \vec{v}}$ in a statistically efficient way. 

\begin{claim}
    \label{clm:yjl_discret_bnds}
    We have for all $j \in [k]$ and even $l$:
    \begin{align*}
        \E \sqlr{(\vec{w}_j^\top \proj{\vec{v}}\vec{x} + \eta_j)^l} &\leq \E [y_j^l \mid A_{\vec{v}}] \\
        &\leq \E_{g \thicksim \mc{N} (0, 1)} \sqlr{\lr{\vec{w}_j^\top \proj{\vec{v}}\vec{x} + \eta_j + \sqrt{\norm{\proj{V} \vec{w}_j}^2 \cdot \rho^2 + \norm{\pproj{\hat{U}} \vec{w}_j}^2} \cdot g}^l}.
    \end{align*}
\end{claim}
\allowdisplaybreaks
\begin{proof}
    We start with the lower bound:
    \begin{align*}
        \E [y_j^l \mid A_{\vec{v}}] &= \E [(\vec{w}_j^\top \proj{V}\vec{x} + \vec{w}_j^\top \proj{\vec{v}}\vec{x} + \vec{w}_j^\top \pproj{\hat{U}}\vec{x} + \eta_j)^l \mid A_{\vec{v}}] \\
        &= \sum_{0 \leq p \leq l} \binom{l}{p} \E [(\vec{w}_j^\top \proj{V}\vec{x} + \vec{w}_j^\top \pproj{\hat{U}}\vec{x})^p (\vec{w}_j^\top \proj{\vec{v}}\vec{x} + \eta_j)^{l - p} \mid A_{\vec{v}}] \\
        &= \sum_{\substack{0 \leq p \leq l \\ p += 2}} \binom{l}{p} \E [(\vec{w}_j^\top \proj{V}\vec{x} + \vec{w}_j^\top \pproj{\hat{U}}\vec{x})^p (\vec{w}_j^\top \proj{\vec{v}}\vec{x} + \eta_j)^{l - p} \mid A_{\vec{v}}] \\
        &\geq \E [(\vec{w}_j^\top \proj{\vec{v}}\vec{x} + \eta_j)^{l} \mid A_{\vec{v}}] = \E [(\vec{w}_j^\top \proj{\vec{v}}\vec{x} + \eta_j)^{l}].
    \end{align*}
    For the upper bound, we have:
    \begin{align*}
        &\E [y_j^l \mid A_{\vec{v}}] \\
        &= \E [(\vec{w}_j^\top \proj{V}\vec{x} + \vec{w}_j^\top \proj{\vec{v}}\vec{x} + \vec{w}_j^\top \pproj{\hat{U}}\vec{x} + \eta_j)^l \mid A_{\vec{v}}] \\
        &= \sum_{0 \leq p \leq l} \binom{l}{p} \E [(\vec{w}_j^\top \proj{V}\vec{x} + \vec{w}_j^\top \pproj{\hat{U}}\vec{x})^p (\vec{w}_j^\top \proj{\vec{v}}\vec{x} + \eta_j)^{l - p} \mid A_{\vec{v}}] \\
        &= \sum_{\substack{0 \leq p \leq l \\ p += 2}} \binom{l}{p} \E [(\vec{w}_j^\top \proj{V}\vec{x} + \vec{w}_j^\top \pproj{\hat{U}}\vec{x})^p (\vec{w}_j^\top \proj{\vec{v}}\vec{x} + \eta_j)^{l - p} \mid A_{\vec{v}}] \\
        &= \sum_{\substack{0 \leq p \leq l \\ p += 2}} \binom{l}{p} \E [(\vec{w}_j^\top \proj{V}\vec{x} + \vec{w}_j^\top \pproj{\hat{U}}\vec{x})^p \mid A_{\vec{v}}] \E [(\vec{w}_j^\top \proj{\vec{v}}\vec{x} + \eta_j)^{l - p}] \\
        &= \sum_{\substack{0 \leq p \leq l \\ p += 2}} \binom{l}{p} \E [(\vec{w}_j^\top \proj{\vec{v}}\vec{x} + \eta_j)^{l - p}] \cdot \sum_{q = 0}^p \binom{p}{q} \E [(\vec{w}_j^\top \proj{V}\vec{x})^q (\vec{w}_j^\top \pproj{\hat{U}}\vec{x})^{p - q} \mid A_{\vec{v}}] \\
        &= \sum_{\substack{0 \leq p \leq l \\ p += 2}} \binom{l}{p} \E [(\vec{w}_j^\top \proj{\vec{v}}\vec{x} + \eta_j)^{l - p}] \cdot \sum_{\substack{0 \leq q \leq p \\ q+=2}} \binom{p}{q} \E [(\vec{w}_j^\top \proj{V}\vec{x})^q (\vec{w}_j^\top \pproj{\hat{U}}\vec{x})^{p - q} \mid A_{\vec{v}}] \\
        &\leq \sum_{\substack{0 \leq p \leq l \\ p += 2}} \binom{l}{p} \E [(\vec{w}_j^\top \proj{\vec{v}}\vec{x} + \eta_j)^{l - p}] \cdot \sum_{\substack{0 \leq q \leq p \\ q+=2}} \binom{p}{q} \E [(\norm{\proj{V} \vec{w}_j} \cdot \rho)^q \cdot (p - q)!! \cdot \norm{\pproj{\hat{U}} \vec{w}_j}^{p - q} \mid A_{\vec{v}}] \\
        &\leq \sum_{\substack{0 \leq p \leq l \\ p += 2}} \binom{l}{p} \E [(\vec{w}_j^\top \proj{\vec{v}}\vec{x} + \eta_j)^{l - p}] \cdot \sum_{\substack{0 \leq q \leq p \\ q+=2}} \binom{p}{q} \E_{g_1, g_2 \overset{iid}{\thicksim} \mc{N} (0, 1)} [(\norm{\proj{V} \vec{w}_j} \cdot \rho \cdot g_1)^q \cdot (\norm{\pproj{\hat{U}} \vec{w}_j} g_2)^{p - q}] \\
        &= \sum_{\substack{0 \leq p \leq l \\ p += 2}} \binom{l}{p} \E [(\vec{w}_j^\top \proj{\vec{v}}\vec{x} + \eta_j)^{l - p}] \cdot \E_{g_1, g_2 \overset{iid}{\thicksim} \mc{N} (0, 1)} [(\norm{\proj{V} \vec{w}_j} \cdot \rho \cdot g_1 + \norm{\pproj{\hat{U}} \vec{w}_j} g_2)^{p}] \\
        &= \E_{g \thicksim \mc{N}(0, 1)} \sqlr{(\vec{w}_j^\top \proj{\vec{v}}\vec{x} + \eta_j + \sqrt{\norm{\proj{V} \vec{w}_j}^2 \cdot \rho^2 + \norm{\pproj{\hat{U}} \vec{w}_j}^2} \cdot g)^l}
    \end{align*}
\end{proof}

In the next claim, we show that it is possible to obtain estimates of $\E [y^l \mid A_{\vec{v}}]$ for all $\vec{v} \in \mc{H}$ with high probability with a small number of samples.
\begin{claim}
    \label{clm:est_mom}
    Given $m$ iid samples from a self-selection model with unknown indices, $\{(\vec{x}_j, y_j)\}_{j = 1}^m$, satisfying $m \geq C \cdot \frac{(2l)!! }{P_\rho} \cdot \lr{\frac{8B}{\Delta}}^{8l} \log (\abs{\mc{H}} / \delta)$, there is a procedure which produces estimates, $\{M_{\vec{v}}\}_{\vec{v} \in \abs{\mc{H}}}$, satisfying:
    \begin{equation*}
        \frac{1}{2} \E [y^l \mid A_{\vec{v}}] \leq M_{\vec{v}} \leq 2 \E [y^l \mid A_{\vec{v}}]
    \end{equation*}
    with probability at least $1 - \delta$. Furthermore, the procedure runs in time $O(\abs{\mc{H}} \cdot m)$.
\end{claim}
\begin{proof}
    We will prove the claim for fixed $\vec{v} \in \mc{H}$ which will establish the conclusion for all $\vec{v} \in \mc{H}$ via a union bound. Let $\vec{v} \in \mc{H}$ and for $q = 100 \log (\abs{\mc{H}} / \delta)$, define:
    \begin{equation*}
        \forall i \in [q]: M_{\vec{v}}^i \coloneqq \frac{\sum_{j = \frac{(i - 1) m}{q} + 1}^{\frac{im}{q}} y^l \bm{1} \blr{(\vec{x}_j, y_j) \in A_{\vec{v}}}}{\sum_{j = \frac{(i - 1) m}{q} + 1}^{\frac{im}{q}}\bm{1} \blr{(\vec{x}_j, y_j) \in A_{\vec{v}}}} , M_{\vec{v}} \coloneqq \mathrm{Median} (\{M_{\vec{v}}^i\}_{i \in [q]}).
    \end{equation*}
    We have from Claim~\ref{clm:av_prob_lb} and an application of \cite[Theorem 4.5]{mitzenmacher}:
    \begin{equation*}
       \P \blr{\sum_{j = \frac{(i - 1)m}{q} + 1}^{\frac{im}{q}} \bm{1} \blr{(\vec{x}_j, y_j) \in A_{\vec{v}}} \geq \frac{1}{2} \cdot \frac{m}{q} \cdot \P \blr{A_{\vec{v}}}} \geq 1 - \exp \blr{- \frac{m \P \blr{A_{\vec{v}}}}{8q}} \geq 0.99.
    \end{equation*}
    Furthermore, we have by Chebyshev's inequality and the fact that $\mathrm{Var} (y^l \mid A_{\vec{v}}) \leq \E [y^{2l} \mid A_{\vec{v}}]$ that for any $i \in [q]$ satisfying the above event:
    \begin{equation*}
        \P \blr{\frac{1}{2} \E [y^l \mid A_{\vec{v}}] \leq M^i_{\vec{v}} \leq 2 \E [y^l \mid A_{\vec{v}}] \Biggr\vert \sum_{j = \frac{(i - 1) m}{q} + 1}^{\frac{im}{q}} \bm{1} \blr{(\vec{x}_j, y_j) \in A_{\vec{v}}} \geq \frac{1}{2} \cdot \frac{m}{q} \cdot \P \{A_{\vec{v}}\}} \geq 0.99.
    \end{equation*}
    A union bound now establishes:
    \begin{equation*}
        \forall i \in [q]: \P \blr{\frac{1}{2} \cdot \E [y^l \mid A_{\vec{v}}] \leq  M^i_{\vec{v}} \leq 2 \cdot \E [y^l \mid A_{\vec{v}}]} \geq 0.98.
    \end{equation*}
    An application of Hoeffding's inequality now yields:
    \begin{equation*}
        \P \blr{\sum_{i = 1}^q \bm{1} \blr{\frac{1}{2} \cdot \E [y^l \mid A_{\vec{v}}] \leq M^i_{\vec{v}} \leq 2 \E[y^l \mid A_{\vec{v}}]} \geq 0.75 q} \geq 1 - \frac{\delta}{\abs{\mc{H}}}
    \end{equation*}
    which implies:
    \begin{equation*}
        \P \blr{\frac{1}{2} \cdot \E [y^l \mid A_{\vec{v}}] \leq M_{\vec{v}} \leq 2 \E[y^l \mid A_{\vec{v}}]} \geq 1 - \frac{\delta}{\abs{\mc{H}}}.
    \end{equation*}
    A union bound over $\vec{v} \in \mc{H}$ now completes the proof of the claim.
\end{proof}

We now set $l = 2 \lceil k((1024 B \log k) / \Delta)^{32} \rceil$ and assume access to estimates, $M_{\vec{v}}$, satisfying the conclusion of Claim~\ref{clm:est_mom}. By dividing through by $(l - 1)!!$ and taking $l^{th}$ roots on both sides, we have from Claims~\ref{clm:yl_const_bnd_finite_sample} and \ref{clm:est_mom}:
\begin{equation*}
    \lr{\frac{\Psi (\vec{v})}{16 (l - 1)!!}}^{1 / l} \leq \underbrace{\lr{\frac{M_{\vec{v}}}{(l - 1)!!}}^{1 / l}}_{\tilde{\sigma}_{\vec{{v}}}} \leq \lr{\frac{2k \Psi (\vec{v})}{(l - 1)!!}}^{1 / l}.
\end{equation*}
We now define the set of directions which will be aligned with one of the $\vec{w}_i$ with high probability. Our candidate set of directions are defined below:
\begin{equation*}
    \mc{S} \coloneqq \blr{\vec{v} \in \mc{H}: \forall \vec{u} \in \mc{H} \text{ s.t } \norm{\vec{v} - \vec{u}} \leq \lr{\frac{\Delta}{128 M}}^2,\ \tilde{\sigma}_u \leq \tilde{\sigma}_{\vec{{v}}}}.
\end{equation*}
We now show that that the elements of $\mc{S}$ are clustered along the $\vec{w}_i$ with high probability and that each $\vec{w}_i$ has at least one element in $\mc{S}$ well aligned with it. Here, we introduce additional notation:
\begin{equation*}
    \forall i \in [k]: \wt{\vec{w}}_i \coloneqq \frac{\vec{w}_i}{\norm{\vec{w}_i}}
\end{equation*}
and state our claim below.

\begin{claim}
    \label{clm:s_set}
    Conditioned on the conclusion of Claim~\ref{clm:est_mom}, we have:
    \begin{align*}
        \forall \vec{v} \in \mc{S}, \exists j \in [k]: \abs*{\vec{v}^\top \wt{\vec{w}}_j} &\geq 1 - \lr{\frac{\Delta}{32 B}}^2 \\
        \forall j \in [k], \exists \vec{v} \in \mc{S}: \abs*{\vec{v}^\top \wt{\vec{w}}_j} &\geq 1 - \lr{\frac{\Delta}{32 B}}^2.
    \end{align*}
\end{claim}
\begin{proof}
    We start with the first part of the claim. Assume the contrary and suppose $\vec{v} \in \mc{S}$ satisfy for all $j \in [k]$, $\abs{\vec{v}^\top \wt{\vec{w}}_j} < 1 - \lr{\frac{\Delta}{32 B}}^2$. Now, let $j^* \in [k]$ be such that $j^* = \argmax_{j \in [k]} \abs{\vec{v}^\top \vec{w}_j}$. We now break into two cases and consider the point $u \in \mc{H}$ such that $u = \argmin_{\vec{z} \in \mc{H}} \norm{\wt{\vec{v}} - \vec{z}}$ where 
    \begin{equation*}
        \wt{\vec{v}} = 
        \begin{cases}
            \frac{\vec{v} + \lr{\frac{\Delta}{256 B}}^2 \cdot \wt{\vec{w}}_{j^*}}{\norm{\vec{v} + \lr{\frac{\Delta}{256 B}}^2 \cdot \wt{\vec{w}}_{j^*}}} &\text{ if } \vec{w}_{j^*}^\top \vec{v} \geq 0 \\
            \frac{\vec{v} - \lr{\frac{\Delta}{256 B}}^2 \cdot \wt{\vec{w}}_{j^*}}{\norm{\vec{v} - \lr{\frac{\Delta}{256 B}}^2 \cdot \wt{\vec{w}}_{j^*}}} &\text{otherwise}
        \end{cases}.
    \end{equation*}
    We now have for the setting where $\vec{w}_{j^*}^\top \vec{v} \geq 0$:
    \begin{align*}
        \abs{\wt{\vec{w}}_{j^*}^\top \vec{u}} &\geq \abs{\wt{\vec{w}}^\top \wt{\vec{v}}} - \norm{\vec{u} - \wt{\vec{v}}} = \frac{\lr{\frac{\Delta}{256 B}}^2 + \abs{\wt{\vec{w}}_{j^*}^\top \vec{v}}}{\norm{\vec{v} + \lr{\frac{\Delta}{256 B}}^2 \cdot \wt{\vec{w}}_{j^*}}} - \norm{\vec{u} - \wt{\vec{v}}} \\
        &\geq \abs{\wt{\vec{w}}_{j^*}^\top \vec{v}} + \lr{\frac{\Delta}{256 B}}^2 \cdot \lr{\frac{1 - \abs{\wt{\vec{w}}_{j^*}^\top \vec{v}}}{2}} - \norm{\vec{u} - \wt{\vec{v}}} \geq \abs{\wt{\vec{w}}_{j^*}^\top \vec{v}} + \lr{\frac{\Delta}{256 B}}^4.
    \end{align*}
     and the alternative case is similar. Furthermore, we have for the case where $\wt{\vec{w}}_{j^*}^\top \vec{v} \geq 0$:
    \begin{equation*}
        \norm{\vec{u} - \vec{v}} \leq \norm{\vec{v} - \wt{\vec{v}}} + \norm{\vec{u} - \wt{\vec{v}}} \leq \abs*{1 - \frac{1}{\norm{\vec{v} + \lr{\frac{\Delta}{256 B}}^2 \cdot\wt{\vec{w}}_{j^*}}}} + \lr{\frac{\Delta}{256 B}}^2 + \norm{\vec{u} - \wt{\vec{v}}} \leq \lr{\frac{\Delta}{128 B}}^2
    \end{equation*}
    and similarly for the alternative case. For $u$, we have by applications of Claims~\ref{clm:yl_const_bnd_finite_sample}, \ref{clm:yjl_discret_bnds} and \ref{clm:est_mom} and our choice of $l$ that $\tilde{\sigma}_u > \tilde{\sigma}_{\vec{{v}}}$ yielding a contradiction. Hence, the first part of the claim follows.
    
    For the second part of the claim, let $j \in [k]$ and $\vec{v}^* = \argmin_{\vec{v} \in \mc{H}} \norm{\vec{v} - \wt{\vec{w}}_j}$. We have that:
    \begin{equation}
        \label{eq:vst_lb_wtj}
        (\vec{v}^*)^\top \wt{\vec{w}}_j \geq 1 - \lr{\frac{\Delta}{1024 B}}^8.
    \end{equation}
    Define $\mc{T} = \blr{\vec{v} \in \mc{H}: \lr{\frac{\Delta}{256 B}}^2 \leq \norm{\vec{v} - \vec{v}^*} \leq \lr{\frac{\Delta}{64B}}^2}$. We will establish that $\tilde{\sigma}_{\vec{v}^*} > \tilde{\sigma}_{\vec{{v}}}$ for all $\vec{v} \in \mc{T}$ establishing the second part of the claim as there exists $u \in \mc{S}$ with $\norm{\vec{u} - \vec{v}^*} \leq (\Delta / (256B))^2$. Suppose $\vec{v} \in \mc{T}$. We have as $\vec{v}^\top \vec{v}^* > 0$:
    \begin{equation}
        \label{eq:v_ub_wtj}
        \abs{\vec{v}^\top \wt{\vec{w}}_j} \leq \abs{\vec{v}^\top \vec{v}^*} + \abs{\vec{v}^\top (\wt{\vec{w}}_j - \vec{v}^*)} = \frac{2 - \norm{\vec{v} - \vec{v}^*}^2}{2} + \norm{\wt{\vec{w}}_j - \vec{v}^*} \leq 1 - \lr{\frac{\Delta}{512 B}}^4.
    \end{equation}
    Note that:
    \begin{equation*}
        \norm{\vec{v} - \wt{\vec{w}}_j} \leq \norm{\vec{v} - \vec{v}^*} + \norm{\vec{v}^* - \wt{\vec{w}}_j} \leq \lr{\frac{\Delta}{32B}}^2.
    \end{equation*}
    We have as a consequence:
    \begin{equation*}
        \vec{v}^\top \vec{w}_j \geq \norm{\vec{w}_j} \cdot \vec{v}^\top \wt{\vec{w}}_j \geq \norm{\vec{w}_j} \cdot \lr{\frac{2 - \norm{\vec{v} - \wt{\vec{w}}_j}^2}{2}} \geq \norm{\vec{w}_j} \cdot \lr{1 - \lr{\frac{\Delta}{32 B}}^4}.
    \end{equation*}
    For any alternative $i \in [k]$ with $i \neq j$, we have from Assumption~\ref{as:no_index} and the definition of $\wt{\vec{w}}_i$:
    \begin{align*}
        \abs{\vec{v}^\top \vec{w}_i} &\leq \norm{\vec{w}_i} \cdot \norm{\vec{v} - \wt{\vec{w}}_j} + \abs{\wt{\vec{w}}_j^\top \vec{w}_i} \leq \min \lr{\norm{\vec{w}_j} - \Delta, \frac{\norm{\vec{w}_i}^2}{\norm{\vec{w}_j}} - \frac{\norm{\vec{w}_i}}{\norm{\vec{w}_j}}\Delta} + \norm{\vec{w}_i} \cdot \lr{\frac{\Delta}{32 B}}^2 \\
        &\leq \min \lr{\norm{\vec{w}_j}, \frac{\norm{\vec{w}_i}^2}{\norm{\vec{w}_j}} - \frac{\norm{\vec{w}_i}}{\norm{\vec{w}_j}}\Delta} + \norm{\vec{w}_i} \cdot \lr{\frac{\Delta}{32 B}}^2 \leq \sqrt{\norm{\vec{w}_i}^2 - \norm{\vec{w}_i} \Delta} + \norm{\vec{w}_i} \cdot \lr{\frac{\Delta}{32B}}^2 \\
        &\leq \norm{\vec{w}_i} \cdot \lr{1 - \frac{1}{2} \cdot \frac{\Delta}{\norm{\vec{w}_i}}} + \norm{\vec{w}_i} \cdot \lr{\frac{\Delta}{32B}}^2 \leq \norm{\vec{w}_i} - \frac{\Delta}{4} \leq \norm{\vec{w}_i} \lr{1 - \frac{\Delta}{4B}} 
    \end{align*}
    which implies $j = \argmax_{i \in [k]} \abs{\vec{w}_j^\top \vec{v}}$ and we get again from this fact and Claims~\ref{clm:yl_const_bnd_finite_sample}, \ref{clm:yjl_discret_bnds} and \ref{clm:est_mom} along with Equations~\ref{eq:vst_lb_wtj} and \ref{eq:v_ub_wtj} that $\tilde{\sigma}_{\vec{v}^*} > \tilde{\sigma}_{\vec{{v}}}$ for our setting of $l$ concluding the proof of the claim.
\end{proof}

Finally, we use the set $\mc{S}$ constructed previously to construct estimates of the vectors $\{\vec{w}_j\}_{j \in [k]}$. For each $u \in \mc{H}$, let $\wt{M}_u$ be estimates of $\E [(y - \sqrt{\tilde{\sigma}_u^2 - 1} u^\top\vec{x})^l \mid A_{\vec{v}}]$ obtained via a union bound and Claim~\ref{clm:est_mom} by setting the failure probability to $\delta / \abs{\mc{H}}^4$ and by using an independent set of $m \geq C \cdot \frac{(2l)!! }{P_\rho} \cdot \lr{\frac{8B}{\Delta}}^{8l} \log (\abs{\mc{H}}^4 / \delta)$ samples from the self-selection model with unknown indices. From Claims~\ref{clm:yl_const_bnd_finite_sample}, \ref{clm:yjl_discret_bnds} and \ref{clm:est_mom}, $\wt{M}_u$ satisfy with probability at least $1 - \delta/2$ for all $u \in \mc{H}$:
\begin{multline}
    \label{eq:mtld_bnd}
    \lr{1 - \lr{\frac{\Delta}{256B}}^{4}} \sqrt{\max_{j \in [k]} (u^\top \vec{w}_j - \sqrt{\tilde{\sigma}_u^2 - 1})^2 + 1} \leq \lr{\frac{\wt{M}_u}{(l - 1)!!}}^{1 / l} \\
    \leq \lr{1 + \lr{\frac{\Delta}{256B}}^{4}} \sqrt{\max_{j \in [k]} (u^\top \vec{w}_j - \sqrt{\tilde{\sigma}_u^2 - 1})^2 + 1}.
\end{multline}
Our (potentially large) set of estimates $\mc{T}$ is now constructed as follows:
\begin{equation*}
    w \in \mc{T} \iff \exists \vec{v} \in \mc{S}: 
    \blr{
        \begin{gathered}
            w = \sqrt{\tilde{\sigma}_{\vec{{v}}}^2 - 1} \vec{v} \text{ and } \sqrt{\lr{\frac{\wt{M}_{\vec{{v}}}}{(l - 1)!!}}^{2 / l} - 1} \leq 2 \sqrt{\tilde{\sigma}_{\vec{{v}}}^2 - 1} - \frac{\Delta}{8} \\
            \text{ or } \\
            w = -\sqrt{\tilde{\sigma}_{\vec{{v}}}^2 - 1} \vec{v} \text{ and } \sqrt{\lr{\frac{\wt{M}_{\vec{{v}}}}{(l - 1)!!}}^{2 / l} - 1} \geq 2 \sqrt{\tilde{\sigma}_{\vec{{v}}}^2 - 1} - \frac{\Delta}{16}
        \end{gathered}
    }
\end{equation*}

\begin{claim}
    \label{clm:t_set}
    We have:
    \begin{align*}
        \forall w \in \mc{T}, \exists j \in [k]: \norm{\vec{w}_j - \vec{w}} &\leq \frac{\Delta}{8} \\
        \forall j \in [k], \exists w \in \mc{T}: \norm{\vec{w}_j - \vec{w}} &\leq \frac{\Delta}{8}
    \end{align*}
\end{claim}
\begin{proof}
    For the first claim, let $\vec{v} \in \mc{S}$ and $j^* = \argmax_{j \in [k]} \abs*{\vec{w}_j^\top \vec{v}}$. From Claim~\ref{clm:s_set}, we have:
    \begin{equation*}
        \abs{\wt{\vec{w}}_{j^*}^\top \vec{v}} \geq 1 - \lr{\frac{\Delta}{32 B}}^2.
    \end{equation*}
    Furthermore, we have from Claims~\ref{clm:yl_const_bnd_finite_sample}, \ref{clm:yjl_discret_bnds} and \ref{clm:est_mom} that:
    \begin{equation*}
        \lr{1 - \lr{\frac{\Delta}{256B}}^4} \sqrt{(\vec{w}_{j^*}^\top \vec{v})^2 + 1} \leq \tilde{\sigma}_{\vec{v}} \leq \lr{1 + \lr{\frac{\Delta}{256B}}^4} \sqrt{(\vec{w}_{j^*}^\top \vec{v})^2 + 1}.
    \end{equation*}
    By squaring, subtracting $1$ from both sides and noting that $\abs{\wt{\vec{w}}_{j^*}^\top \vec{v}} \geq \Delta / 2$, we obtain:
    \begin{equation}
        \label{eq:sig_v_bnd}
        \lr{1 - \lr{\frac{\Delta}{32 B}}^2} \abs{\vec{w}_{j^*}^\top \vec{v}} \leq \sqrt{\tilde{\sigma}_{\vec{{v}}}^2 - 1} \leq \lr{1 + \lr{\frac{\Delta}{32 B}}^2} \abs{\vec{w}_{j^*}^\top \vec{v}}.
    \end{equation}
    
    As in the proof of Theorem~\ref{thm:identifiability_no_index}, all that remains is to distinguish the two cases: the first when $\vec{w}_{j^*}^\top \vec{v} > 0$ and the second when $\vec{w}_{j^*}^\top \vec{v} < 0$. For the first, we have from Equation~\ref{eq:sig_v_bnd}:
    \begin{equation*}
        -\lr{\frac{\Delta}{32B}}^2 \norm{\vec{w}_{j^*}} \leq \vec{v}^\top (\vec{w}_{j^*} - \sqrt{\tilde{\sigma}_{\vec{{v}}}^2 - 1} \vec{v}) \leq \lr{\frac{\Delta}{32B}}^2 \norm{\vec{w}_{j^*}}
    \end{equation*}
    Furthermore, we have for all $j \in [k]$ with $j \neq j^*$:
    \begin{align*}
        \abs*{\vec{v}^\top (\vec{w}_j - \sqrt{\tilde{\sigma}_{\vec{{v}}}^2 - 1} \vec{v})} &\leq \abs*{\vec{v}^\top \vec{w}_j} + \sqrt{\tilde{\sigma}_{\vec{{v}}}^2 - 1} = \abs*{((\vec{v} - (\vec{v}^\top \wt{\vec{w}}_{j^*}) \wt{\vec{w}}_{j^*}) + (\vec{v}^\top \wt{\vec{w}}_{j^*}) \wt{\vec{w}}_{j^*})^\top \vec{w}_j} + \sqrt{\tilde{\sigma}_{\vec{{v}}}^2 - 1} \\
        &\leq \abs*{(\vec{v} - (\vec{v}^\top \wt{\vec{w}}_{j^*}) \wt{\vec{w}}_{j^*})^\top \vec{w}_j} + \abs*{(\vec{v}^\top \wt{\vec{w}}_{j^*}) \wt{\vec{w}}_{j^*}^\top \vec{w}_j} + \sqrt{\tilde{\sigma}_{\vec{{v}}}^2 - 1} \\
        &\leq \norm{\vec{w}_j} \norm{\vec{v} - (\vec{v}^\top \wt{\vec{w}}_{j^*})\wt{\vec{w}}_{j^*}} + \abs*{\wt{\vec{w}}_{j^*}^\top \vec{w}_j} + \sqrt{\tilde{\sigma}_{\vec{{v}}}^2 - 1} \\
        &\leq B \cdot \frac{\Delta}{16B} + \norm{\vec{w}_{j^*}} - \Delta + \sqrt{\tilde{\sigma}_{\vec{{v}}}^2 - 1} \leq 2 \norm{\vec{w}_{j^*}} - \frac{\Delta}{2} \leq 2 \sqrt{\tilde{\sigma}_{\vec{{v}}}^2 - 1} - \frac{\Delta}{4}.
    \end{align*}
    Hence, in the first case we have from Equation~\ref{eq:mtld_bnd}:
    \begin{equation*}
        \sqrt{\lr{\frac{\wt{M}_{\vec{{v}}}}{(l - 1)!!}}^{2 / l} - 1} \leq 2 \sqrt{\tilde{\sigma}_{\vec{{v}}}^2 - 1} - \frac{\Delta}{8}.
    \end{equation*}
    In the alternative case where $\vec{v}^\top \vec{w}_{j^*} < 0$, we have:
    \begin{align*}
        \abs*{\vec{v}^\top (\vec{w}_{j^*} - \sqrt{\tilde{\sigma}_{\vec{{v}}}^2 - 1} \vec{v})} &\geq -\vec{v}^\top (\vec{w}_{j^*} - \sqrt{\tilde{\sigma}_{\vec{{v}}}^2 - 1} \vec{v}) \\
        &\geq \lr{1 - \lr{\frac{\Delta}{32B}}^2} \norm{\vec{w}_{j^*}} + \sqrt{\tilde{\sigma}_{\vec{{v}}}^2 - 1} \geq 2 \sqrt{\tilde{\sigma}_{\vec{{v}}}^2 - 1} - \frac{\Delta}{32}.
    \end{align*}
    Consequently, we have in this case again from Equation~\ref{eq:mtld_bnd}:
    \begin{equation*}
        \sqrt{\lr{\frac{\wt{M}_{\vec{{v}}}}{(l - 1)!!}}^{2 / l} - 1} \geq 2 \sqrt{\tilde{\sigma}_{\vec{{v}}}^2 - 1} - \frac{\Delta}{16}.
    \end{equation*}
    
    We now show that any $\vec{w} \in \mc{T}$ can be uniquely associated with a $\vec{w}_j$. Let $\vec{w} \in \mc{T}$. Suppose $\vec{w} = \sqrt{\tilde{\sigma}_{\vec{{v}}}^2 - 1} \vec{v}$ for some $\vec{v} \in \mc{S}$; that is, the first case occurs and $\vec{v}^\top \vec{w}_{j^*} > 0$. We now have:
    \begin{align*}
        \norm{\vec{w} - \vec{w}_{j^*}} &\leq \norm{\vec{w}_{j^*} - (\vec{w}_{j^*}^\top \vec{v}) \vec{v}} + \norm{(\vec{w}_{j^*}^\top \vec{v}) \vec{v} - \sqrt{\tilde{\sigma}_{\vec{{v}}}^2 - 1} \cdot \vec{v}} \\
        &\leq \norm{\vec{w}_{j^*}} \cdot \norm{\pproj{\vec{v}} \wt{\vec{w}}_{j^*}}+ \lr{\frac{\Delta}{32B}}^2 \leq B \cdot \frac{\Delta}{16B} + \lr{\frac{\Delta}{32B}}^2 \leq \frac{\Delta}{8}.
    \end{align*}
    The alternative case is similar. This establishes the first claim of the lemma.
    
    For the second part of the claim, let $j \in [k]$ and $\vec{v} \in \mc{S}$ satisfy (Claim~\ref{clm:s_set}):
    \begin{equation*}
        \abs{\vec{v}^\top \wt{\vec{w}}_j} \geq 1 - \lr{\frac{\Delta}{32B}}^2.
    \end{equation*}
    For all $i \neq j$, we have as in the proof of Claim~\ref{clm:s_set}, from Assumption~\ref{as:no_index} and the definition of $\wt{\vec{w}}_i$:
    \begin{align*}
        \abs{\vec{v}^\top \vec{w}_i} &\leq \norm{\vec{w}_i} \cdot \norm{\vec{v} - \wt{\vec{w}}_j} + \abs{\wt{\vec{w}}_j^\top \vec{w}_i} \leq \min \lr{\norm{\vec{w}_j} - \Delta, \frac{\norm{\vec{w}_i}^2}{\norm{\vec{w}_j}} - \frac{\norm{\vec{w}_i}}{\norm{\vec{w}_j}}\Delta} + \norm{\vec{w}_i} \cdot \lr{\frac{\Delta}{32 B}}^2 \\
        &\leq \min \lr{\norm{\vec{w}_j}, \frac{\norm{\vec{w}_i}^2}{\norm{\vec{w}_j}} - \frac{\norm{\vec{w}_i}}{\norm{\vec{w}_j}}\Delta} + \norm{\vec{w}_i} \cdot \lr{\frac{\Delta}{32 B}}^2 \leq \sqrt{\norm{\vec{w}_i}^2 - \norm{\vec{w}_i} \Delta} + \norm{\vec{w}_i} \cdot \lr{\frac{\Delta}{32B}}^2 \\
        &\leq \norm{\vec{w}_i} \cdot \lr{1 - \frac{1}{2} \cdot \frac{\Delta}{\norm{\vec{w}_i}}} + \norm{\vec{w}_i} \cdot \lr{\frac{\Delta}{32B}}^2 \leq \norm{\vec{w}_i} - \frac{\Delta}{4} \leq \norm{\vec{w}_i} \lr{1 - \frac{\Delta}{4B}} 
    \end{align*}
    and hence $j = \argmax_{i \in [k]} \abs{\vec{v}^\top \wt{\vec{w}}_i}$. The second part now follows from the first part as the corresponding element, $w \in \mc{T}$ satisfies $\norm{\vec{w} - \vec{w}_j} \leq \Delta / 8$ finishing the proof of the claim.
\end{proof}

Finally, we prune the set $\mc{T}$ to remove duplicate vectors for the same $\vec{w}_j$. Note that all $\vec{w}, \vec{u} \in \mc{T}$ with $\norm{\vec{w} - \vec{w}_j}, \norm{\vec{u} - \vec{w}_j} \leq \Delta / 8$ satisfy 
\begin{equation*}
    \norm{\vec{w} - \vec{u}} \leq \norm{\vec{w} - \vec{w}_j} + \norm{\vec{u} - \vec{w}_j} \leq \frac{\Delta}{4}.
\end{equation*}
    by the triangle inequality. Furthermore, all $\vec{w}, \vec{u} \in \mc{T}$ with $\norm{\vec{w} - \vec{w}_j} \leq \Delta / 8$ and $\norm{\vec{u} - \vec{w}_i} \leq \Delta / 8$ for $i \neq j$ satisfy:
\begin{equation*}
    \norm{\vec{w} - \vec{u}} \geq \norm{\vec{w}_i - \vec{w}_j} - \norm{\vec{w} - \vec{w}_j} - \norm{\vec{u} - \vec{w}_j} \geq 3\Delta / 4
\end{equation*}
by the triangle inequality and Assumption~\ref{as:no_index}. Therefore, from Claim~\ref{clm:s_set}, a simple de-duplication step is to cluster all points within a radius of $\Delta / 2$ of each other and picking one representative from each of them. This concludes the proof of Theorem~\ref{thm:no_index}.

\qed

\subsection{Estimation in the $k = 2$ case}
\label{ssec:two_models}
We will now demonstrate how, when $k = 2$, we can use a moment-based
algorithm to estimate $\{\vec{w_1}, \vec{w_2}\}$ in $\text{poly}(1/\epsilon$)
time. In particular, we provide the algorithm corresponding to Theorem \ref{thm:no_index:2}, 
restated below:
\unknownindtwo*
\noindent The algorithm will operate as follows: 
\begin{enumerate}
    \item Using the procedure outlined in Section \ref{ssec:identifying_subspace}, we find an
    approximation $U$ to the linear subspace $U^*$ containing
    $\text{span}(\vec{w_1}, \vec{w_2})$. 
    \item Set up an $(\eps/6)$-covering over $U \cap \mathcal{B}(B)$, where $\mathcal{B}(B)$
    is the $\ell_2$ ball with radius $B$.
    Since $\|\vec{w_i}\| \leq B$ both vectors are contained in the covering, 
    and the covering is of size $O(B^2/\varepsilon^2)$.
    \item For each element $\widehat{\vec{w}}$ of the
    covering, we collect samples $(\vec{x}, y - \vec{x}^\top \widehat{\vec{w}})$ where $(x, y)$ are from the
    no-index self selection model.
    \item Using the moments of $y - \vec{x}^\top \widehat{\vec{w}}$, we estimate
    $\min_{i \in \{0, 1\}} \|\vec{w_i} - \widehat{\vec{w}}\|^2$.
    \item We will show that
    $O(\delta^{-1}\eps^{-2}\text{poly}(1/B))$ samples suffice to get an
    $\eps$-close approximation to this quantity with probability $1-\delta$.
    Setting 
    $\delta = O(\rho \eps^2)$ for some $\rho < 1$ ensures that we get
    accurate estimates of this quantity for each element $\vec{w}$ of our covering.
    \item As long as $\vec{w_1}$ and $\vec{w_2}$ are sufficiently separated,
    we can estimate $\vec{w_1}$ to be the minimum of our estimate over the
    $\eps$-covering, i.e., $\widehat{\vec{w_1}} = \arg\min_{w} \min_{i \in \{0, 1\}}
    \|\vec{w_i} - \vec{w}\|$. We can then estimate $\vec{w_2}$ to be the
    minimizer of the estimate over points that are far enough from
    $\widehat{\vec{w_1}}$. 
\end{enumerate}

Turning this outline into an efficient algorithm entails tackling a few distinct
technical challenges. First, we will show how to estimate $\min_{i \in \{0, 1\}}
\|\vec{w_i} - \vec{v}\|$ using samples $(\vec{x}, y)$ from our data-generating
process. Then, we will show that our sequential approach to estimating
$\vec{w_1}$ and $\vec{w_2}$ indeed suffices to recover both with good enough
accuracy. Finally, we will show that the error incurred by the subspace-finding
step does not adversely affect our estimation.

\paragraph{Subspace recovery.}
From Section \ref{ssec:identifying_subspace}, we may assume the existence of a 
$2$-dimensional subspace ${U}$ satisfying the following:
\begin{equation*}
    \forall i \in [2]: \frac{\norm{w_i - \mathcal{P}_{{U}} w_i}}{\norm{w_i}} \leq \eps/6,
\end{equation*}
where to find ${U}$ we need $\poly(B, d, 1/\eps, 1/\alpha, 1/\Delta)$ sample and time 
complexity. Thus, in the remainder of this section, 
we will operate over the (at most two-dimensional)
subspace ${U}$.

\paragraph{Finding the nearest weight vector.}
When $k = 2$, direct integration allows us to compute the moment generating
function of $y = \max \{\vec{w_1}^\top \vec{x} + \eta_1, \vec{w_2}^\top
\vec{x} + \eta_2\}$ in terms of the covariance matrix between $y_1$ and
$y_2$. This in turn allows us to accurately estimate the lesser of
$\text{Var}[y_1]$ and $\text{Var}[y_2]$, as captured by the following Lemma:
\begin{lemma} \label{lem:minVariance}
    Given a two-dimensional Gaussian random variable $\vec{z} \sim \mathcal{N}(0,
    \Sigma)$ with $0 \prec \Sigma \in \mathbb{R}^2$, there exists an
    algorithm $\textsc{MinVariance}(\delta, \epsilon)$ which given
    $O(\delta^{-1}\eps^{-4})$ samples of $\max(\vec{z})$, outputs an estimate
    $\widehat{\sigma}$ of $\min_{0, 1} \Sigma_{ii}$ satisfying
    \[
        \min(\Sigma_{1,1},\Sigma_{2,2}) - \eps 
        \leq \widehat{\sigma} 
        \leq \min(\Sigma_{1,1},\Sigma_{2,2}) + \eps
    \]
    with probability $1 - \delta$.
\end{lemma}
\begin{proof}
    We make use of the following closed form for the moment generating function of
the maximum of two Gaussians $X_1$, $X_2$ as given by \citet{nadarajah2008exact}:
\begin{align*}
    m(t) &= \exp\lr{t\mu_1 + \frac{t^2 \sigma_1^2}{2}}\cdot \Phi\lr{\frac{\mu_1 - \mu_2 + t(\sigma_1^2 - \rho \sigma_1 \sigma_2)}{\sqrt{\sigma_1^2 + \sigma_2^2 - 2\rho \sigma_1 \sigma_2}}} \\
    &\qquad + \exp\lr{t\mu_2 + \frac{t^2 \sigma_2^2}{2}}\cdot \Phi\lr{\frac{\mu_2 - \mu_1 + t(\sigma_2^2 - \rho \sigma_1 \sigma_2)}{\sqrt{\sigma_1^2 + \sigma_2^2 - 2\rho \sigma_1 \sigma_2}}},
\end{align*}
where $\mu_i = \mathbb{E}[X_i]$, $\sigma_i^2 = \mathbb{E}[(X_i - \mu_i)^2]$, and
$\rho = \mathbb{E}[X_1 X_2] - \mu_1 \mu_2$. In our case, $\mu_1 = \mu_2 = 0$, and so
\begin{align*}
    m(t) &= \exp\lr{\frac{t^2 \sigma_1^2}{2}}\cdot \Phi\lr{\frac{t(\sigma_1^2 - \rho \sigma_1 \sigma_2)}{\sqrt{\sigma_1^2 + \sigma_2^2 - 2\rho \sigma_1 \sigma_2}}} 
    + \exp\lr{\frac{t^2 \sigma_2^2}{2}}\cdot \Phi\lr{\frac{t(\sigma_2^2 - \rho \sigma_1 \sigma_2)}{\sqrt{\sigma_1^2 + \sigma_2^2 - 2\rho \sigma_1 \sigma_2}}}.
\end{align*}
Let $X = \max\{X_1, X_2\}$ be the observed variable. Manual differentiation
of the MGF yields:
$\mathbb{E}[X^2] = \frac{1}{2}(\sigma_1^2 + \sigma_2^2)$ and 
$\mathbb{E}[X^4] = \frac{3}{2}(\sigma_1^4 + \sigma_2^4)$; in particular, note
that the moments are independent of $\rho$ even if $\rho > 0$.
Now, re-parameterizing in terms of $a = \sigma_1^2 + \sigma_2^2$ and $b =
\sigma_1 \sigma_2$:
\begin{align*}
    2\mathbb{E}[X^2] = a, \qquad
    \frac{2}{3}\mathbb{E}[X^4] = a^2 - 2b^2 \implies
    a = 2\cdot \mathbb{E}[X^2], \qquad
    b = \sqrt{2\cdot \mathbb{E}[X^2]^2 - \frac{1}{3}\mathbb{E}[X^4]} 
\end{align*}
Note that the solution to the system is unique because
$a, b > 0$ (since $\sigma_1 > 0$ and $\sigma_2 > 0$).
We can now solve for both $\sigma_1$ and $\sigma_2$,
since we have $\sigma_1^2 + \sigma_2^2 = a$ and $\sigma_1^2\sigma_2^2 =
b^2$, and two real numbers are uniquely determined by their sum and product:
\begin{align*}
    \sigma_1^2 & = \mathbb{E}[X^2] - \sqrt{\frac{1}{3}\mathbb{E}[X^4] - \mathbb{E}[X^2]^2} 
    \qquad \qquad \\
    \sigma_2^2 & =  \mathbb{E}[X^2] + \sqrt{\frac{1}{3}\mathbb{E}[X^4] - \mathbb{E}[X^2]^2} 
\end{align*}

Thus, estimating the minimum (and maximum) variance in the mixture amounts to
estimating the second and fourth moments of $X = \max \{X_1, X_2\}$.
Note that $\text{Var}[X^4 - \mathbb{E}[X^4]] \leq \mathbb{E}[X^8] =
\frac{105}{2}(\sigma_1^8 + \sigma_2^8),$ where the latter is attained via
direct calculation from the MGF. Applying Chebyshev's inequality and letting
$X^{(1)},\ldots, X^{(n)}$ be i.i.d. samples of $X$,
\[
    \mathbb{P}\lr{\left|\frac{1}{n}\sum_{i=1}^n \lr{X^{(i)}}^4 - \mathbb{E}[X^4]\right| \geq \eps} < \frac{105 (\sigma_1^8 + \sigma_2^8)}{2 n\eps^2}.
\]
The same argument applies for the second moment. Thus, taking $n =
O\lr{\delta^{-1} \sigma_{max}^{8} \eps^{-4}}$ and propagating errors through the equation
for $\sigma_1^2$ concludes the proof.

\end{proof}
\noindent As a corollary, we can estimate $\min_{i \in \{0, 1\}}
\|\vec{w_i} - \vec{w}\|^2$ to $\eps$-precision with probability at least $1
- \delta$:
\begin{corollary}
    \label{cor:minVariance}
    Suppose we have samples $\{(\vec{x}, y)\}$ generated from the self-selection model
    with unobserved index. Suppose further that $\|\vec{w_i}\| \leq B$ for $i
    \in  \{0, 1\}$. Then, we can use the \textsc{MinVariance} algorithm of Lemma 
    \ref{lem:minVariance} to recover $\sigma$ such that with probability at least
    $1 - \delta$, 
    \[
        \left|\sigma - \min_{i \in \{0, 1\}} \|\vec{w_i} - \vec{w}\| \right| \leq \eps/6,
    \]
    using $n \in O(\delta^{-1}B^8\eps^{-8})$ samples.
\end{corollary} 
\begin{proof}
    Define the random variable $X = y - \vec{w}^\top \vec{x}$. Then, $X$ is the
    maximum of the two Gaussians
    $X_1 = (\vec{w_1} - \vec{w})^\top \vec{x} + \eta_i$ and 
    $X_2 = (\vec{w_2} - \vec{w})^\top \vec{x} + \eta_i$.
    In particular, $\text{Var}[X_i] = \|\vec{w_i} - \vec{w}\|^2 + 1$. Thus, 
    applying the \textsc{MinVariance} algorithm to $X$ (with precision $\eps^2$) 
    recovers the desired quantity.
\end{proof}

\paragraph{Estimating the weight vectors.} 

Note that our separability assumption (Assumption \ref{as:no_index}) implies that 
$\|\vec{w}_1 - \vec{w}_2\| \geq \Delta$.
Now, using this and Corollary \ref{cor:minVariance}, we have that 
we can find a point on the
$2$-dimensional cover that is at least $\eps/2$-close to $\vec{w}_1$
(with $\eps/6$ error coming from the subspace identification, the grid granularity,
and the estimation step).

Once we have identified $\vec{w}_1$, we can substract 
$\vec{w}_1^{\top} \vec{x}$ from all the next samples of the form $(\vec{x}, y)$
find the minimum variance among the grid points that are at least $\Delta$
far away from the estimated $\vec{w}_1$, and this way we are guaranteed to recover
the remaining $\vec{w}_2$ as well.

\section{Acknowledgements}
We thank Sitan Chen for pointing out an error in the original proof of Lemma \ref{lemma:svd_cov}.
This work is supported by NSF Awards CCF-1901292, DMS-2022448 and
DMS2134108, the DOE PhILMs project
(DE-AC05-76RL01830), 
a Simons Investigator Award, the Simons Collaboration on the Theory
of Algorithmic Fairness, a DSTA grant, an Open Philanthropy AI Fellowship and a Microsoft
Research-BAIR Open Research Commons grant. 

\printbibliography

\clearpage
\appendix

\section{Computations of Gradient and Hessian for the Known-Index Case} 
\label{app:computationsLogLikelihood}
Suppose we have a given parameter estimate for $\bm{W}^* = [\vec{w}_j^*]_{j=1}^k$ 
given by $\bm{W} = [\vec{w}_j]_{j=1}^k$. For a single sample $(\vec{x}, y, j_*)$
from the known-index self-selection model (Definition \ref{defn:index_observed}), 
the likelihood under the current parameter estimate (conditioned on a fixed
$\vec{x}$) can be written as the likelihood of observing $y$ from the $j_*$-th
model, multiplied by the probability of the $j_*$-th model being observed
conditioned on its output being $y$. In particular,
\begin{align*}
    p(\bm{W}; \vec{x}, y, j_*) 
        &= f_{\sigma}(y - \vec{w}_{j_*}^\top \vec{x})
                \cdot {\int_{C_{j_*}(y)} \prod_{j \neq j_*} f_{\sigma}(z_j - \vec{w}_j^\top \vec{x})\, d\bm{z}},
\end{align*}
where $C_{j_*}(y) \in \mathbb{R}^{k-1}$ is the (convex) set such that for any $\bm{z} \in
\mathbb{R}^k$, 
\[
   \bm{z}_{-j_*} \in C_{j_*}(y) \text{ and } z_{j_*} = y \iff S(\bm{z}) = j_*,
\]
and where $f_{\sigma}$ is the canonical probability density function of the
normal distribution with mean zero and variance $\sigma^2$. Thus, the conditional
log-likelihood for a single sample is given by
\begin{align}
    \label{eq:ll}
    \ell(\bm{W}; \vec{x}, y, j_*) = 
        \log\lr{f_{\sigma}(y - \bm{w}_{j_*}^\top \vec{x})} + 
        \log\lr{\int_{C_{j_*}(y)} \prod_{j \neq j_*} f_{\sigma}(z_j - \vec{w}_j^\top \vec{x})\, d\bm{z} }.
\end{align}
Finally, the objective function $\overline{\ell}$ that we use it the population 
version of the above with respect to $(y, i)$ and the average over the observed 
samples over $\vec{x}$.
\begin{align*}
    \overline{\ell}(\bm{W}) = 
    \frac{1}{n} \sum_{i = 1}^n 
    \mathbb{E}_{(y, j_*) \sim \mathcal{D}(\vec{x}^{(i)}; \vec{W}^*)} 
    \left[
        \log\lr{f_{\sigma}(y - \bm{w}_{j_*}^\top \vec{x}^{(i)})} + 
        \log\lr{\int_{C_i(y)} \prod_{j \neq j_*} 
        f_{\sigma}(\bm{z}_j - \vec{w}_j^\top \vec{x}^{(i)})\, d\bm{z} }
    \right],
\end{align*}
which matches \eqref{eq:objectiveDefinition}.
The most important properties of $\ell$ are:
\begin{itemize}
    \item[(i)] there is an appropriate projection set that contains the true
    parameters such that inside the set the function $\ell$ is strongly concave, 
    \item[(ii)] its maxima correspond to the true set of parameters
    $[\vec{w}_j^*]_{j=1}^k$. 
\end{itemize}

We next derive the gradient of the log-likelihood for a single sample $(\vec{x},
i, y)$. The gradient 
will be a vector of the form $[\nabla_{\vec{w_1}} \ell(\bm{W}; \vec{x}, y,
j_*); \ldots; \nabla_{\vec{w_k}} \ell(\bm{W}; \vec{x}, y, j_*)]$:  
we will handle the gradients with respect to $\vec{w}_{j_*}$ and  
$\{\vec{w}_{j}: j \neq j_*\}$
separately. Throughout this section and the next, we use the following
about the standard Gaussian density:
\begin{fact}[Obtained via direct calculation]
    For any $\vec{w}, \vec{x} \in \mathbb{R}^d$ and $z, \sigma \in \mathbb{R}$, we
    have that:
    \begin{align*}
        \nabla_{\vec{w}} f_{\sigma}(z - \bm{w}^\top \vec{x}) 
        &= \frac{z - \bm{w}^\top \vec{x}}{\sigma^2}
            f_{\sigma}(z - \bm{w}^\top \vec{x}) \cdot \vec{x},
        \text{ and} \\
        \nabla^2_{\vec{w}} f_{\sigma}(z - \bm{w}^\top \vec{x})
        &= \frac{(z - \bm{w}^\top \vec{x})^2-\sigma^2}{\sigma^4}
            f_{\sigma}(z - \bm{w}^\top \vec{x}) \cdot \vec{x} \vec{x}^\top,
    \end{align*}
    where $f_\sigma$ is the canonical PDF of a mean-zero, variance-$\sigma^2$
    Gaussian random variable.
\end{fact}
\noindent We now return to deriving the gradient of the log-likelihood. First,
for the gradient with respect to $\bm{w}_{j_*}$ note that the second term of the
objective function above is independent of $\bm{w}_{j_*}$ and thus
\begin{align}
    \label{eq:ll_grad_i}
    \nabla_{\vec{w}_{j_*}} \ell(\bm{W}; \vec{x}, y, j_*) 
    &= \frac{1}{\sigma^2}(\vec{w}_{j_*}^\top \vec{x} - y)\cdot \vec{x}.
\end{align}
For the $\vec{w}_j$ terms where $j \neq j_*$, the first term in the objective
disappears and we are left with:
\begin{align}
    \label{eq:ll_grad_j}
    \nabla_{\vec{w}_j} \ell(\bm{W}; \vec{x}, y, j_*) &= 
        \frac{\int_{C_{j_*}(y)} \nabla_{\vec{w}_j} f_{\sigma}(z_j - \vec{w}_j^\top \vec{x}) \cdot 
                            \prod_{l \in [k] \setminus \{j,j_*\}} f_{\sigma}(z_l - \vec{w}_l^\top \vec{x})\, d\bm{z}_{-j_*}}
                {\int_{C_{j_*}(y)} \prod_{l \neq j_*} f_{\sigma}(z_l - \vec{w}_l^\top \vec{x})\, d\bm{z}_{-j_*}} \\
    \nonumber
    &= \mathbb{E}_{z_{-j_*} \sim \mathcal{N}((\bm{W}^\top \vec{x})_{-j_*}, \sigma \bm{I}_{k-1} )}\left[
        \frac{1}{\sigma^2}(z_j - \vec{w}_j^\top \vec{x}) \big\vert \vec{z}_{-j_*} \in C_{j_*}(y)
    \right]\cdot \vec{x}.
\end{align}

We continue with the computation of the Hessian. 
In particular, the function $\ell$ admits a Hessian $\bm{H}$ 
made of blocks $\bm{H}_{j,l}$, where 
\[ (\bm{H}_{j,l})_{ab} = \ddt{(\bm{w}_j)_a}{(\bm{w}_l)_b} \overline{\ell}(\bm{W};
\vec{x}, y, j_*). \]

From the above computations, it follows that for a single sample $(\vec{x}, y, j_*)$,
the matrix block $\bm{H}_{j,j_*} = 0$ for all $j \neq j_*$. Thus, it remains to
consider only the blocks $\bm{H}_{j_*,j_*}$ and blocks $\bm{H}_{j,l}$ for which 
$j \neq j_*$ and $l \neq j_*$.
Now, to get $\bm{H}_{j_*,j_*}$ we differentiate \eqref{eq:ll_grad_i} with
respect to $\vec{w}_{j_*}$ again which yields:
\begin{align*}
    \bm{H}_{j_*,j_*}
    = \frac{\nabla^2_{\vec{w}_{j_*}} f_{\sigma}(y - \bm{w}_{j_*}^\top \vec{x})}{
        f_{\sigma}(y - \bm{w}_{j_*}^\top \vec{x})} -
    \frac{\nabla_{\vec{w}_{j_*}} f_{\sigma}(y - \bm{w}_{j_*}^\top \vec{x}) 
    \cdot \nabla_{\vec{w}_{j_*}} f_{\sigma}(y - \bm{w}_{j_*}^\top \vec{x})^\top
    }{{f_{\sigma}(y - \bm{w}_{j_*}^\top \vec{x})}^2} = -\frac{1}{\sigma^2}\vec{x}\vec{x}^\top
\end{align*}

We now turn to the entries $\bm{H}_{j,j}$ for $j \neq j_*$:
\begin{align*}
    \bm{H}_{j,j} &= \left[\frac{\int_{C_{j_*}(y)} \frac{1}{\sigma^4} ((\vec{w}_j^\top \vec{x} - z_j)^2 - \sigma^2) f_{\sigma}(z_j - \vec{w}_j^\top \vec{x}) \cdot 
                            \prod_{l \in [k] \setminus \{i,j\}} f_{\sigma}(z_l - \vec{w}_l^\top \vec{x})\, d\bm{z}_{-j_*}}
                {\int_{C_{j_*}(y)} \prod_{l \neq j_*} f_{\sigma}(z_l - \vec{w}_l^\top \vec{x})\, d\bm{z}_{-j_*} }\right. \\
                &\qquad \left. - \lr{\frac{\int_{C_{j_*}(y)} \frac{1}{\sigma^2} (\vec{w}_j^\top \vec{x} - z_j) f_{\sigma}(z_j - \vec{w}_j^\top \vec{x}) \cdot 
            \prod_{l \in [k] \setminus \{j,j_*\}} f_{\sigma}(z_l - \vec{w}_l^\top \vec{x})\, d\bm{z}_{-j_*}}
{\int_{C_{j_*}(y)} \prod_{l \neq j_*} f_{\sigma}(z_l - \vec{w}_l^\top \vec{x})\, d\bm{z}_{-j_*} }}^2\right]\cdot \vec{x}\vec{x}^\top \\
                &= \frac{1}{\sigma^4}\lr{
                    \text{Var}_{\bm{z}_{-j_*} \sim \mathcal{N}((\bm{W}^\top \vec{x})_{-j_*},\, 
                    \sigma^2 \bm{I}_{k-1}
                    )}
                    \left[z_j|\bm{z}_{-j_*} \in C_{j_*}(y)\right] 
                    - \sigma^2}\cdot \vec{x}\vec{x}^\top
\end{align*}

\noindent Using the same procedure to find the off-diagonal terms ($\bm{H}_{jl}$
for $j \neq l$) yields 
\begin{align*}
    \bm{H}_{j,l}
        &= \frac{1}{\sigma^4} 
        \text{Cov}_{\bm{z}_{-j_*} \sim \mathcal{N}((\bm{W}^\top \vec{x})_{-j_*},\, \sigma^2 \bm{I}_{k-1})}
            \left[z_j, z_l\,|\, \bm{z}_{-j_*} \in C_{j_*}(y)\right] \cdot \vec{x}\vec{x}^\top.
\end{align*}
Thus, putting together the blocks $\bm{H}_{j,l}$ for which $j, l \neq j_*$, 
\begin{equation}
    \label{eq:hessian}
    \bm{H}_{\vec{W}_{-j_*}} = \frac{1}{\sigma^4} \lr{
        \text{Cov}_{\bm{z}_{-j_*} \sim \mathcal{N}((\bm{W}^\top \vec{x})_{-j_*},\, \sigma^2 \bm{I}_{k-1})}
            \left[z_j, z_l\,|\, \bm{z}_{-j_*} \in C_{j_*}(y)\right]
        - \sigma^2 \bm{I}
    } \otimes \vec{x}\vec{x}^\top,
\end{equation}
where $\otimes$ represents the Kronecker product.
 Thus, the complete Hessian for a single sample $(\vec{x}, i, y)$ can be
expressed as a block matrix of the form:
\[
    \bm{H} 
    = \left[\begin{matrix}
        -\frac{1}{\sigma^2} \vec{x}\vec{x}^\top & \bm{0} \in \mathbb{R}^{d \times d(k-1)} \\
        \bm{0} \in \mathbb{R}^{d(k-1) \times d} & \bm{H}_{\vec{W}_{-j_*}}
    \end{matrix}\right]
    \preceq \left[\begin{matrix}
        -\frac{1}{\sigma^2} \vec{x}\vec{x}^\top & \bm{0} \\
        \bm{0} & \bm{0}
    \end{matrix}\right].
\]
Our minimum-probability assumption (Assumption \ref{asp:known:1})
together with thickness thus implies that $\bm{H}_{pop} \preceq -\frac{(\alpha /
k)}{\sigma} \bm{I}$, where $\bm{H}_{pop}$ is the Hessian of the population
log-likelihood.

\section{Missing Proofs from Section \ref{sec:known}} \label{app:known}

  In this section we present the missing proof of the lemmas for the known-setting
estimation that we presented in Section \ref{sec:known}.

\subsection{Proof of Lemma \ref{lem:known:stationary}} \label{app:known:stationary}

\newcommand{\latent}{\tilde{Y}}
Recall that we have $j_* \in [k]$ the observed index, $y_j$ for all $j \in [k]$ the 
(unobserved) samples from the self-selection model, and $y = y_{j_*}$ the observed response 
variable from the model. 
Fixing a single example $\vec{x}$, we consider the population log-likelihood:
\begin{align*}
    \nabla_{\vec{w}_j} \bar{\ell}(\vec{W}^*;\vec{x}) 
        &= \Exp_{(y,j_*)} \left[ 
            \frac{1}{\sigma^2}\left(
            \bm{1}_{j = j_*} \cdot y 
            + \bm{1}_{j \neq j_*} \cdot 
            \mathbb{E}_{\vec{z}_{-j_*} \sim \mathcal{N}((\vec{W^*}^\top x^{(i)})_{-j_*}, \sigma^2 \bm{I}_{k-1})}\left[z_j|\bm{z}_{-j_*} \in C_{j_*}(y)\right]
            - \vec{w}_j^{*\top} \vec{x}
            \right) 
        \right] \cdot \vec{x} \\
        &= \Exp_{(y_1,\ldots, y_k)} \left[ 
            \frac{1}{\sigma^2}\left(
            \bm{1}_{j = j_*} \cdot y 
            + \bm{1}_{j \neq j_*} \cdot 
            \mathbb{E}_{\vec{z}_{-j_*} \sim \mathcal{N}((\vec{W^*}^\top x^{(i)})_{-j_*}, \sigma^2 \bm{I}_{k-1})}\left[z_j|\bm{z}_{-j_*} \in C_{j_*}(y)\right]
            - \vec{w}_j^{*\top} \vec{x}
            \right) 
        \right] \cdot \vec{x} \\
        &= 
        -\Exp_{(y_1, \ldots, y_k)} \left[ 
            \frac{1}{\sigma^2}\lr{\vec{w}_j^{*\top} \vec{x} - y_j} 
            + \bm{1}_{j \neq j_*} \frac{1}{\sigma^2}\lr{y_j 
            - \Exp_{\vec{z}_{-j_*}}\left[z_j|z_{-j_*} \in C_{j_*}(y)\right]
            }
        \right] \cdot \vec{x},
\end{align*}
where in the last two expectations $y$ and $j_*$ are deterministic functions of
the sampled latent variables $y_1, \ldots, y_k$. 
By definition, $\Exp[\vec{\bm{w}_j}^{*\top} \vec{x} - y_j] = 0$, so
\begin{align*}
    -\nabla_{\vec{w_j}} \bar{\ell}(\vec{W}^*;\vec{x}) 
        &= 
        \frac{1}{\sigma^2}
        \Exp \left[ 
            \bm{1}_{j \neq j_*}  \lr{y_j 
            - \Exp_{\vec{z}_{-j_*}}\left[z_j|\bm{z}_{-j_*} \in C_{j_*}(y)\right]
            }
        \right] \cdot \vec{x} \\ 
        &= \frac{1}{\sigma^2}
        \sum_{l \neq j}
        \Exp \left[ 
            \bm{1}_{l = j_*}\lr{y_j 
            - \Exp_{\vec{z}_{-j_*}}\left[z_j|\bm{z}_{-j_*} \in C_{j_*}(y)\right]
            }
        \right] \cdot \vec{x} \\
        &= \frac{1}{\sigma^2}
        \sum_{l \neq j}
        \Exp_{(y_l, j_*)} \left[ 
        \Exp_{\bm{y}_{-l}} \left[ 
            \bm{1}_{l = j_*}\lr{y_j 
            - \Exp_{\vec{z}_{-j_*}}\left[z_j|\bm{z}_{-j_*} \in C_{j_*}(y)\right]
            }
        \bigg\vert y_l, j_* \right]\right] \cdot \vec{x} \\ 
        &= \frac{1}{\sigma^2}
        \sum_{l \neq j}
        \Exp_{(i, y_l)} \left[ 
            \bm{1}_{l = j_*} \lr{
        \Exp_{y_{-l}} \left[ 
            y_j | y_{l}, j_*\right]
            - \Exp_{\vec{z}_{-j_*}}\left[z_j|\bm{z}_{-j_*} \in C_{j_*}(y)\right]
            }
        \right] \cdot \vec{x}.
\end{align*}
Observing that for the true parameters $\bm{W}^*$, the two inner expectations
above are equal conditioned on $l = j_*$ concludes the proof.

\subsection{Proof of Lemma \ref{lem:samp_trunc}} \label{app:lem:samp_trunc}

Our first lemma establishes that suitably ``wide'' convex sets contain 
a ball of non-trivial radius. For the rest of this proof we use 
$d = k - 1$ for simplicity and we also drop the bold letters for vectors and 
matrices. 

\begin{lemma}
    \label{lem:wide_convex}
    Let $K \subset \R^d$ be a bounded convex set satisfying for some $\alpha > 0$:
    \begin{equation*}
        \forall \norm{v} = 1: \sup_{x,y \in K} v^\top (x - y) \geq \alpha.
    \end{equation*}
    Then, there exists $x^* \in K$ such that:
    \begin{equation*}
        \blr{x: \norm{x - x^*} \leq \frac{\alpha}{256 d^2} } \subset K.
    \end{equation*}
\end{lemma}
\begin{proof}
    We first probe the lemma in the setting where $K$ is closed. Now, let $\mu$ be the uniform distribution over $K$ and $x^*$ be the mean of $\mu$. Consider any $x \in \R^d \setminus K$. Since, $K$ is a compact convex set, there exists $v \in \mb{R}^d, \gamma \in \R$ such that $\norm{v} = 1$ and $v^\top x > \gamma$ and $v^\top y < \gamma$ for all $y \in K$ by the separating hyperplane theorem. Note that:
    \begin{equation}
        \label{eq:dir_bound}
        \norm{x^* - x} \geq v^\top (x - x^*) \geq \max_{y \in K} v^\top (y - x^*).
    \end{equation}
    Let $m = \min_{y \in K} v^\top y$ and $M = \max_{y \in K} v^\top y$ and for any $\beta \in [m, M]$:
    \begin{equation*}
        K_{\beta} \coloneqq \blr{x: v^\top x = \beta} \cap K,\ V (\beta) = \mathrm{Vol}_{d - 1} \lr{K_{\beta}} \text{ and } \beta^* = \argmax_{\beta} V(\beta).
    \end{equation*}
    That is, $V(\beta)$ denotes the $(d - 1)$-dimensional volume of the $(d - 1)$-dimensional slice of $K$ with the hyperplane $\{x: v^\top x = \beta\}$. Note that $V(\beta^*) \neq 0$ as then $\mathrm{Vol} (K) = 0$. Defining, $\eta \coloneqq M - m$, we now prove the following claim.
    \begin{claim}
        \label{clm:anti_conc_proj}
        We have:
        \begin{equation*}
        \exists \beta_1, \beta_2 \in [m, M]: \beta_1 \leq \beta_2 - \frac{\eta}{4d} \text{ and } \forall \beta \in [\beta_1, \beta_2], V(\beta) \geq \frac{V(\beta^*)}{2}.
    \end{equation*}
    \end{claim}
    \begin{proof}
        We start by breaking into two cases:
        \begin{itemize}
            \item[] \textbf{Case 1: } $\abs{\beta^* - m} \leq \abs{\beta^* - M}$ and
            \item[] \textbf{Case 2: } $\abs{\beta^* - M} \leq \abs{\beta^* - m}$.
        \end{itemize}
        For the first case, let $y \in K$ be such that $v^\top y = M$ and define for any $\rho \in [0, 1]$, the set 
        \begin{equation*}
            \forall \rho \in [0, 1]: L_\rho \coloneqq \{y + \rho (z - y): z \in K_{\beta^*}\}
        \end{equation*}
        Note that convexity, $L_\rho \subset K$ for all $\rho \in [0, 1]$. Observing $\mathrm{Vol}_\rho (L_\rho) = \rho^d \mathrm{Vol} (K_{\beta^*})$, we get:
        \begin{equation*}
            \forall \rho \in \sqlr{1 - \frac{1}{2d}, 1}: \mathrm{Vol}_{d - 1} (L_\rho) \geq \frac{\mathrm{Vol}_{d - 1} (L_1)}{2}.
        \end{equation*}
        Furthermore, note that $\mathrm{Vol}_{d - 1}(V_1) = V(\beta^*)$ and that $L_\rho \subset K_{M - (\beta^* - M) \rho}$ for all $\rho \in [0, 1]$. Noting that $M - \beta^* \geq \eta / 2$ concludes the proof of the claim in this case. The alternative case is similar.
    \end{proof}
    Now, we have:
    \begin{align*}
        v^\top x^* &= \frac{\int_{m}^M  \beta V(\beta) d\beta}{\int_m^M V(\beta) d\beta} \leq M - \frac{\int_{m}^M  (M - \beta) V(\beta) d\beta}{\int_m^M V(\beta) d\beta} \leq M - \frac{\int_{m}^M  (M - \beta) V(\beta) d\beta}{\eta V(\beta^*)} \\
        &\leq M - \frac{\int_{\beta_1}^{\beta_2}  (M - \beta) V(\beta) d\beta}{\eta V(\beta^*)} \leq M - \frac{\int_{\beta_1}^{\beta_1 + \eta / 8d}  \eta V(\beta) d\beta}{8d \eta V(\beta^*)} \\
        &\leq M - \frac{\eta}{8d} \cdot \frac{\eta V(\beta^*)}{2} \cdot \frac{1}{8d\eta V(\beta^*)} \leq M - \frac{\eta}{128 d^2}.
    \end{align*}
    Similarly, we have:
    \begin{equation*}
        v^\top x^* \geq m - \frac{\eta}{128 d^2}
    \end{equation*}
    which establishes the lemma by noting that $\eta \geq \alpha$ and using Equation~\ref{eq:dir_bound}. The general result follows by considering the closure of $K$ and choosing the radius to be $\eta / (256 d^2)$.
\end{proof}

We now prove that one can sample from an arbitrarily centered truncated gaussian
distribution as long as the truncation set has large mass under the standard
gaussian measure where we crucially utilize an analysis of the projected
Langevin sampling algorithm by Bubeck, Eldan and Lehec \cite{bubecksampling}. In
what follows, $\mc{N} ((c, I) \mid K)$ will denote the truncated gaussian
density restricted to $K$; i.e, the density function may be written as:
\begin{equation*}
    f_{w, K} (x) = 
    \begin{cases}
        \frac{\exp \blr{- \frac{\norm{x - c}^2}{2}}}{\int_K \exp \blr{- \frac{\norm{y - c}^2}{2}} dy} & \text{if } x \in K\\
        0 & \text{if } x \notin K
    \end{cases}
\end{equation*}

\begin{proof}[Proof of Lemma \ref{lem:samp_trunc}.]
    First, let $g \thicksim \mc{N} (0, I)$. We have by the concentration of Lipschitz functions of gaussians:
    \begin{equation*}
        \P \blr{\norm{g} \geq \sqrt{d} + \sqrt{2 \log 1 / \delta}} \leq \delta
    \end{equation*}
    Setting, $\delta = \alpha / 2$, we get:
    \begin{equation*}
        \P \blr{g \in \underbrace{\mathbb{B} (0, \overbrace{\sqrt{d} + \sqrt{2 \log 2 / \alpha}}^{R_1}) \cap K}_{\wt{K}}} \geq \frac{\alpha}{2}.
    \end{equation*}
    For $\wt{K}$, we must have as the pdf of a standard gaussian is lower bounded by $1 / \sqrt{2\pi}$:
    \begin{equation*}
        \forall \norm{v} = 1: \max_{x, y \in \wt{K}} v^\top (x - y) \geq \frac{\alpha}{2}.
    \end{equation*}
    Therefore, we get from Lemma~\ref{lem:wide_convex};
    \begin{equation}
        \label{eq:int_wtk}
        \exists c \in \wt{K}: \mathbb{B} \lr{c, \underbrace{\frac{\alpha}{512 d^2}}_{\beta}} \subset \wt{K} \text{ and } \norm{c} \leq \sqrt{d} + \sqrt{2 \log 2 / \alpha}.
    \end{equation}
    
    For a gaussian random variable centered at $w$, $h \thicksim \mc{N} (w, I)$, we have:
    \begin{align*}
        \P \blr{h \in \wt{K}} &= \int_{\wt{K}} \frac{1}{(\sqrt{2 \pi})^d} \exp\blr{- \frac{\norm{x - w}^2}{2}} dx \\
        &= \int_{\wt{K}} \frac{1}{(\sqrt{2 \pi})^d} \exp\blr{- \frac{\norm{x}^2 - 2 x^\top w + \norm{w}^2}{2}} dx \\
        &\geq \exp \blr{- \lr{\frac{\norm{w}^2 + 2 \norm{w} R_1}{2}}} \cdot \int_{\wt{K}} \frac{1}{(\sqrt{2 \pi})^d} \exp \blr{- \frac{\norm{x}^2}{2}} dx \\
        &\geq \underbrace{\frac{\alpha}{2} \exp \blr{- \lr{\frac{\norm{w}^2 + 2 \norm{w} R_1}{2}}}}_{\gamma} \addtocounter{equation}{1} \tag{\theequation} \label{eq:g_wtk}.
    \end{align*}
    
    Similarly to $g$, we have by the triangle inequality:
    \begin{equation*}
        \P \blr{\norm{h} \geq \sqrt{d} + \norm{w} + \sqrt{2 \log 1 / \delta}} \leq \delta
    \end{equation*}
    and consequently, by setting $\delta = \gamma (\epsilon / 16)$, we get:
    \begin{equation*}
        \P \blr{h \notin \mb{B} (0, \underbrace{\sqrt{d} + \norm{w} + \sqrt{2 \log (16 / (\gamma\epsilon))}}_{R_2})} \leq \frac{\gamma \epsilon}{10}.
    \end{equation*}
    Along with Equation~\ref{eq:g_wtk}, we get:
    \begin{equation*}
        \frac{\P \blr{h \notin \overbrace{\mb{B} (0, R_2) \cap K}^{\hat{K}}}}{\P \blr{h \in K}} \leq \frac{\epsilon}{16}.
    \end{equation*}
    Note that since $R_2 > R_1$, Equation~\ref{eq:int_wtk} applies to $\hat{K}$ as well. 
    
    Let $X$ and $Y$ be random variables distributed according to $N ((w, I) \mid \hat{K})$ and $N ((w, I) \mid K)$ respectively. Then, we have:
    \begin{equation*}
        TV (X, Y) = \P (Y \notin \hat{K}) + (\P(X \in \hat{K}) - \P (Y \in \hat{K})) = 2 \P (Y \notin \hat{K}) \leq \frac{\epsilon}{8}.
    \end{equation*}
    
    Therefore, it suffices to generate a sample $\hat{X}$ with distribution close to $X$ in TV distance. Consider the rescaled random variable $\bar{X} = \frac{\hat{X}}{\gamma}$ and the rescaled set $\bar{K} = \frac{\hat{K}}{\gamma}$. Letting $\mu$ denote the distribution function of $\bar{X}$, we have after rescaling:
    \begin{gather*}
        \mb{B} \lr{\frac{c}{\gamma}, 1} \subset \bar{K} \\
        \frac{d\mu}{dx} = \frac{1}{Z} \exp \blr{- f(x)} \mb{1} \blr{x \in \bar{K}} \text{ where } Z = \int_{\bar{K}} \exp \blr{- f(x)} dx \text{ and } f(x) = \frac{\norm{\gamma x - w}^2}{2}.
    \end{gather*}
    We have:
    \begin{gather*}
        \forall x \in \bar{K} \norm{\nabla f(x)} = \max \norm{\gamma (\gamma x - w)} \leq \gamma \cdot (R_2 + \norm{w}) \\
        \bar{K} \subset \mb{B} \lr{0, \frac{R_2}{\gamma}} \\
        \forall x,y \in \bar{K}: \norm{\nabla f(x) - \nabla f(y)} \leq \gamma^2 \norm{x - y}.
    \end{gather*}
    Applying \cite[Theorem 1]{bubecksampling} and rescaling, we can generate in time $\mathrm{poly} (d, \norm{w}, 1 / \epsilon, 1 / \alpha)$, a random vector $\hat{X}$ satisfying by the triangle inequality:
    \begin{equation*}
        \mathrm{TV} \lr{\hat{X}, X} \leq \frac{\epsilon}{8} \implies \mathrm{TV} \lr{\hat{X}, Y} \leq \frac{\epsilon}{4}.
    \end{equation*}
    This concludes the proof of Lemma \ref{lem:samp_trunc}.
\end{proof}

\subsection{Proof of Lemma \ref{lem:samplingGradient}} \label{app:lem:samplingGradient}

  From the description of the algorithm in Section \ref{sec:Langevin} and from 
the fact that $n \ge 2 T^2/\zeta$ we have that after $T$ calls of the estimation 
Algorithm \ref{alg:langevin} there are no collisions with respect to the sampled 
index $a$ with probability at least $1 - \zeta$. So our method has failure 
probability $\zeta$ and for the rest of the proof we focus on the event that 
there are no collisions. 
\smallskip

  From the discussion in Section \ref{sec:Langevin} and from the form of 
$\vec{g}^{p}$ we have that
\begin{align*}
    \norm{\Exp \left[\vec{g}^{p} \mid \vec{W}^{(p - 1)}, \vec{g}^{(p - 1)} \right]  -  \nabla \bar{\ell}(\vec{W}^{(p)})}_2 \le C \cdot \norm{\vec{z}^{(m)}  - \Exp_{\vec{z} \sim \mathcal{F}_{i, a}(\vec{W^{(p)}}, y)}\left[\vec{z}\right]}_2
\end{align*}
where $C$ is the upper bound on the norm of the vector of covariates 
$\vec{x}^{a}$ for all $a \in [n]$. Using Lemma \ref{lem:samp_trunc} we have that
there exists a distribution $\hat{\mathcal{F}}$ such that 
$\mathrm{TV}(\hat{\mathcal{F}}, \mathcal{F}_{i, a}(\vec{W}^{(p)}, y)) \le \omega$ and
\begin{align*}
    \norm{\Exp \left[\vec{g}^{p} \mid \vec{W}^{(p - 1)}, \vec{g}^{(p - 1)} \right]  -  \nabla \bar{\ell}(\vec{W}^{(p)})}_2 & \le C \cdot \norm{\Exp_{\vec{z} \sim \hat{\mathcal{F}}}\left[\vec{z}\right] - \Exp_{\vec{z} \sim \mathcal{F}_{i, a}(\vec{W^{(p)}}, y)}\left[\vec{z}\right]}_2 \\
    & \le C \cdot R \cdot \omega.
\end{align*}
where $R$ is $\mathrm{poly} (k, B, 1 / \epsilon, 1 / \alpha, \sigma^2, 1/\sigma^2)$.
The last inequality comes from the proof of Lemma \ref{lem:samp_trunc} where
we have shown that $\hat{\mathcal{F}}$ actually has support inside the ball 
$\mathcal{B}(R)$ and the fact that the expected value of 
$\mathcal{F}_{i, a}(\vec{W}^{(p)}, y)$ is at most some 
$\mathrm{poly}(k, B, 1/a)$. The latter follows from Lemma 6 of 
\cite{daskalakis2018efficient} combined with Lemma 2 of 
\cite{daskalakis2019computationally}. Therefore, we can make $\omega$ small 
enough so that $1/\omega$ is at least 
$\mathrm{poly} (k, B, 1 / \epsilon, 1 / \alpha, \sigma^2, 1/\sigma^2)$ by setting
$m$ larger than 
$\mathrm{poly} (k, B, 1 / \epsilon, 1 / \alpha, \sigma^2, 1/\sigma^2)$ and the 
lemma follows.

\subsection{Proof of Lemma \ref{lemma:sgd}}
\label{app:bias_sgd_proof}
\newcommand{\xt}{\vec{x}_t} 
\newcommand{\xtp}{\vec{x}_{t+1}} 
\newcommand{\ztp}{\vec{z}_{t+1}} 
\newcommand{\xopt}{\vec{x}_{*}} 
\newcommand{\inner}[2]{\langle {#1}, {#2} \rangle}

  To prove this lemma we adapt the proof of Theorem 14.11 from 
\citep{shalev2014understanding}. Consider the PSGD algorithm defined by the
following update:
\begin{align*}
    z_{t+1} = x_t - \eta_t \cdot \vec{g}_t \qquad \text{ and } \qquad x_{t+1} = \Pi(z_{t+1}),
\end{align*}
where $\vec{g}_t$ is a {\em biased} estimate of the gradient such that
$\Exp[\vec{g}_t|\xt] - \nabla_{\vec{x}} f(\xt) \coloneqq \vec{b}_t$.
For convenience, let $\nabla_t = \nabla_{\vec{x}} f(\xt)$. By strong convexity, we have that 
\begin{align}
    \label{eq:bias_sgd_1}
    (\xt - \xopt)^\top \nabla_t &\geq f(\xt) - f(\xopt) + \frac{\lambda}{2}\|\xt - \xopt\|^2
\end{align}
Furthermore, since $\xtp$ is the projection of $\ztp$ onto a convex set
containing $\xopt$, we must have $\|\xtp - \xopt\| \leq \|\ztp - \xopt\|$.
Thus, 
\begin{align*}
   \|\xt - \xopt\|^2 - \|\xtp - \xopt\|^2 \geq 
   \|\xt - \xopt\|^2 - \|\ztp - \xopt\|^2 = 2\eta_t (\xt - \xopt)^\top \vec{g}_t - \eta_t^2 \|\vec{g}_t\|^2.
\end{align*}
Taking the expectation of both sides and rearranging,
\begin{align}
    \nonumber
   \Exp\lr{\|\xt - \xopt\|^2 - \|\xtp - \xopt\|^2} &\geq 2\eta_t (\xt - \xopt)^\top \lr{\nabla_t + \vec{b}_t} - \eta_t^2 \rho^2 \\
    \label{eq:bias_sgd_2}
   \frac{\Exp\lr{\|\xt - \xopt\|^2 - \|\xtp - \xopt\|^2}}{2\eta_t} + \frac{\eta_t \rho^2}{2} &\geq (\xt - \xopt)^\top \lr{\nabla_t + \vec{b}_t}
\end{align}
Combining \eqref{eq:bias_sgd_1} and \eqref{eq:bias_sgd_2} and summing over $t
\in [T]$ yields:
\begin{align*}
    \sum_{t=1}^T \Exp[f(\xt)] - f(\xopt) \leq 
    \sum_{t=1}^T \Exp\lr{\lr{
   \frac{\|\xt - \xopt\|^2 - \|\xtp - \xopt\|^2}{2\eta_t} + \frac{\eta_t \rho^2}{2} 
   - (\xt - \xopt)^\top \vec{b}_t - \frac{\lambda}{2}\|\xt - \xopt\|^2
    }}.
\end{align*}
Setting $\eta_t = \frac{1}{t \lambda}$ causes the first and last terms above to
collapse to $-T\lambda \|\vec{x}_T - \xopt\|^2 \leq 0$, so
\begin{align*}
    \sum_{t=1}^T \Exp[f(\xt)] - f(\xopt) \leq 
    \sum_{t=1}^T \lr{ \frac{\rho^2}{2\lambda t} - (\xt - \xopt)^\top \vec{b}_t
    }
    \leq \frac{\rho^2}{2\lambda}(1 + \log(T)) + R\sum_{t=1}^T\|b_t\|.
\end{align*}
Thus, ensuring $\|b_t\|^2 \leq \frac{\rho^2}{2\lambda Rt}$, dividing by $T$ and
applying Jensen's inequality yields
\begin{align*}
    \Exp[f(\bar{\vec{x}})] - f(\xopt) \leq \frac{\rho^2}{\lambda T}(1+\log(T)).
\end{align*}

\section{Missing Proofs from Section \ref{sec:unknown}}
\subsection{Proof of Lemma \ref{lem:svd_conc}}
\label{app:proof:svd_conc}
In this section, we provide the proof for Lemma \ref{lem:svd_conc}, showing that the
empirical second moment matrix constructed in Section
\ref{ssec:identifying_subspace} concentrates around its expectation.
Our proof will make use of the following technical Lemma from \citep{yi2016solving}:
\begin{lemma}[Lemma 13 of \citep{yi2016solving}]
    \label{lemma:their_moment_bounds}
    Let $X \sim \mathcal{N}(0, \bm{I}_d)$ and for each $i \in [k]$, let $Z_i
    \sim \mathcal{N}(0, 1)$ be independent Gaussian random variables;
    $V_i \in \mathbb{R}^d$ be a fixed unit vector; and $Y_i = a_i X^\top V_i +
    b_i Z$. 
    For $\tau_1, \tau_2 \geq 1$, define the events $\mathcal{E}_i = \{|X^\top V_i| \leq
    \tau_1, |Z_i| \leq \tau_2\}$.
    Then, we have that for all $i \in [k]$,
   \begin{align*}
        \left\|
            \mathbb{E}[Y_i^2 \cdot XX^\top \big\vert \mathcal{E}_i]
        \right\|_2 &\lesssim a_i^2 + b_i^2 \\
        \left\| 
            \mathbb{E}[Y_i^2 \cdot XX^\top \big\vert \mathcal{E}_i^c] \cdot \mathbb{P}(\mathcal{E}_i^c)
        \right\|_2 &\lesssim
            (a_i^2 + b_i^2)\lr{
                \tau_1^3 \cdot e^{-\tau_1^2/2} 
                + \tau_1\tau_2 \cdot e^{-(\tau_1^2 + \tau_2^2)/2}
                + \tau_2 \cdot e^{-\tau_2^2/2}
            }.
   \end{align*}
\end{lemma}
\begin{proof}
    Both statements are shown within the proof of Lemma 13 in
    \citep{yi2016solving}.
    Note that the second statement is Equation (63) in \citep{yi2016solving} but
    with a small correction---in particular, in the last display on page 35 of
    that work, the $\exp(-\tau_2^2/2)$ term is dropped between the last and
    second-to-last lines.
\end{proof}
We will also use the following standard sub-Gaussian concentration inequality:
\begin{lemma}[Theorem 5.39 of \citep{vershynin2010introduction}]
    \label{lem:subgaussian_conc}
    Suppose $\vec{x}_1, \vec{x}_2, \ldots \vec{x}_n \in \mathbb{R}^d$ are
    i.i.d. subgaussian random vectors with with Orlicz norm
    $\|\vec{x}_i\|_{\psi_2} \leq K$. Then, there exist constants $C_1, C_2$ such
    that, for every $t \in (0, K^2)$ and $n \geq C_1(K^2/t)^2d$
    \[
        \mathbb{P}\lr{\left\|
            \frac{1}{n} \sum_{i \in [n]} \vec{x}_i\vec{x}_i^\top
            - \mathbb{E}[\vec{x}_1\vec{x}_1^\top]
        \right\|_2 \geq t} \leq 
        \exp\lr{-C_2 nt^2 / K^4}.
    \]
\end{lemma}

With this result in hand, we can prove a similar result for the random variable
$Y = \max_{i \in [k]} Y_i$: 
\begin{lemma}
    \label{lemma:cond_mean_dev}
    Assume the setting of Lemma \ref{lemma:their_moment_bounds} with $a
    \coloneqq \max_i a_i$ and $b \coloneqq \max_i b_i$.
    Define $Y = \max_{i \in [k]} Y_i$ and $\tilde{Y} = \max(0, \tilde{Y})$, 
    and let
    $\mathcal{E} = \cap_{i=1}^k \mathcal{E}_i$. Then, 
    \begin{align*}
        \left\|\mathbb{E}[\tilde{Y}^2 \cdot XX^\top \big\vert \mathcal{E}] - 
        \mathbb{E}[\tilde{Y}^2 \cdot XX^\top] \right\|_2
        &\lesssim k \cdot (a^2 + b^2)(\tau_1^3 + \tau_2^2)\lr{
            e^{-\tau_1^2/2} + e^{-\tau_2^2/2}
        }.
    \end{align*}
\end{lemma}
\begin{proof}
    First, observe that 
    \[
        \mathbb{E}[\tilde{Y}^2 \cdot XX^T] = 
    \mathbb{E}[\tilde{Y}^2 \cdot XX^T|\mathcal{E}]\cdot \mathbb{P}(\mathcal{E}) 
    + \mathbb{E}[\tilde{Y}^2 \cdot XX^T|\mathcal{E}^c] \cdot \mathbb{P}(\mathcal{E}^c).
    \]
    Subtracting $\mathbb{E}[Y^2 \cdot XX^T|\mathcal{E}]$ from each side,
    applying the triangle inequality, and noting that $\tilde{Y}^2 \leq Y^2$,
    \begin{align*}
        \left\|\mathbb{E}[\tilde{Y}^2 \cdot XX^\top \big\vert \mathcal{E}] - 
        \mathbb{E}[\tilde{Y}^2 \cdot XX^\top] \right\|_2
        &\leq \left\|\mathbb{E}[Y^2 \cdot XX^\top \big\vert \mathcal{E}]\right\|_2 
        \cdot \mathbb{P}(\mathcal{E}^c) +
        \left\|\mathbb{E}[Y^2 \cdot XX^\top \big\vert \mathcal{E}^c]
        \cdot \mathbb{P}(\mathcal{E}^c)\right\|_2
    \end{align*}
    We bound each term individually; for the first, we can use the definition of
    the event $\mathcal{E}$ directly, together with the first part of Lemma
    \ref{lemma:their_moment_bounds}, since $\mathcal{E}$ implies each
    $\mathcal{E}_i$:
    \begin{align*}
        \left\|\mathbb{E}[Y^2 \cdot XX^\top \,\vert\, \mathcal{E}]\right\|_2 \leq 
        \sum_{i=1}^k \left\|\mathbb{E}[Y_i^2 \cdot XX^\top \,\vert\, \mathcal{E}]\right\|_2 \lesssim k \cdot (a^2 + b^2).
    \end{align*}
    For the second, we define $\widetilde{\mathcal{E}} = \mathcal{E}^c \cap
    \{|Y| \geq a\tau_1 + b\tau_2\}$, so that
    \begin{align*}
        \left\|\mathbb{E}[Y^2 \cdot XX^\top \,\vert\, \mathcal{E}^c]
        \cdot \mathbb{P}(\mathcal{E}^c)\right\|_2
        &= \left\|\mathbb{E}\left[
            Y^2 \cdot XX^\top \cdot \bm{1}_{\mathcal{E}^c}
        \right] \right\|_2 \\
        &\leq \underbrace{\left\|\mathbb{E}\left[
            Y^2 \cdot XX^\top \cdot \bm{1}_{\widetilde{\mathcal{E}}}
        \right] \right\|_2}_{\delta_1} +
        \underbrace{\left\|\mathbb{E}\left[
            Y^2 \cdot XX^\top \cdot \bm{1}_{\mathcal{E}^c \setminus \widetilde{\mathcal{E}}}
        \right] \right\|_2}_{\delta_2}.
    \end{align*}
    Now, $\widetilde{\mathcal{E}}$ implies that $Y_i \geq a\tau_1 + b\tau_2$ for
   at least one $i \in [k]$, so $\bm{1}_{\widetilde{\mathcal{E}}} \leq \max_{i \in [k]}
   \bm{1}_{|Y_i| \geq a\tau_1 + b\tau_2}$. Thus, we can simplify the first term
   above to 
    \begin{align*}
        \delta_1 \coloneqq \left\|\mathbb{E}\left[
            Y^2 \cdot XX^\top \cdot \bm{1}_{\widetilde{\mathcal{E}}}
        \right] \right\|_2
        &\leq \left\|\mathbb{E}\left[
            \lr{\max_{i \in [k]} Y_i^2} \cdot XX^\top \cdot \lr{\max_{i \in [k]} \bm{1}_{|Y_i| \geq a\tau_1 + b\tau_2}}
        \right] \right\|_2.
    \end{align*}
    By construction, any maximizer of $\max_{i \in [k]} Y_i^2$ will also be a
    maximizer of $\max_{i \in [k]} \bm{1}_{|Y_i| \geq a\tau_1 + b\tau_2}$, so
    \begin{align*}
        \delta_1 &\leq \left\|\mathbb{E}\left[
            \lr{\max_{i \in [k]} Y_i^2 \cdot \bm{1}_{|Y_i| \geq a\tau_1 + b\tau_2}} \cdot XX^\top
        \right] \right\|_2 
        \leq \sum_{i=1}^k \left\|\mathbb{E}\left[
            Y_i^2 \cdot \bm{1}_{|Y_i| \geq a\tau_1 + b\tau_2} \cdot XX^\top
        \right] \right\|_2.
    \end{align*}
    Since $|Y_i| \leq a|X^\top V_i| + b|Z_i|$, the event $|Y_i| \geq a\tau_1 +
    b\tau_2$ implies either $|X^\top V_i| \geq \tau_1$ or $|Z_i| \geq \tau_2$,
    thus implying the event $\mathcal{E}_i^c$. Using the second part of Lemma
    \ref{lemma:their_moment_bounds} yields: 
    \begin{align*}
        \delta_1 
        &\leq \sum_{i=1}^k \left\|\mathbb{E}\left[
            Y_i^2 \cdot \bm{1}_{\mathcal{E}_i^c} \cdot XX^\top
        \right] \right\|_2 \\
        &\leq \sum_{i=1}^k (a_i^2 + b_i^2)\lr{
            \tau_1^3 \cdot e^{-\tau_1^2/2} 
            + \tau_1\tau_2 \cdot e^{-(\tau_1^2 + \tau_2^2)/2}
            + \tau_2 \cdot e^{-\tau_2^2/2}
        } \\
        &\lesssim k \cdot (a^2 + b^2)(\tau_1^3 + \tau_1\tau_2)(e^{-\tau_1^2/2} + e^{-\tau_2^2/2})
    \end{align*}
    Bounding $\delta_2$ is more straightforward, since we have an upper bound on
    $|Y|$, and $\mathcal{E}^c \setminus \widetilde{\mathcal{E}}$ implies $\{X\}$
    \begin{align*}
        \left\|\mathbb{E}\left[
            Y^2 \cdot XX^\top \cdot \bm{1}_{\mathcal{E}^c \setminus \widetilde{\mathcal{E}}_i}
        \right] \right\|_2 \leq (a\tau_1 + b\tau_2)^2
        \left\|\mathbb{E}\left[
           XX^\top \cdot \bm{1}_{\mathcal{E}^c}
        \right] \right\|_2 
        &\leq (a\tau_1 + b\tau_2)^2 \cdot \mathbb{P}(\mathcal{E}^c) \\
        &\lesssim (a^2 + b^2)(\tau_1^2 + \tau_2^2) \cdot \mathbb{P}(\mathcal{E}^c).
    \end{align*}
    Putting together the previous bounds and noting that
    $\mathbb{P}(\mathcal{E}^c) \lesssim k \cdot (e^{-\tau_1^2/2} + e^{-\tau_2^2/2})$ yields
    that: 
    \begin{align*}
        \left\|\mathbb{E}[Y^2 \cdot XX^\top \big\vert \mathcal{E}] - 
        \mathbb{E}[Y^2 \cdot XX^\top] \right\|_2
        &\lesssim k \cdot (a^2 + b^2)(\tau_1^3 + \tau_2^2)\lr{
            e^{-\tau_1^2/2} + e^{-\tau_2^2/2}
        }.
    \end{align*}
\end{proof}

\noindent We can now use Lemma \ref{lemma:cond_mean_dev} to establish the desired result:
\svdconc*
\begin{proof}
    For the sake of simplicity, we assume that our weight vectors $\vec{w_i}$
    have norm bounded by one (below, setting $a = B$ and $\tau_1$ by $\tau_1 /
    B$ where $B$ is a norm bound on the weights recovers at most a polynomial
    dependence on $B$).
    We first translate our setting to that of Lemma
    \ref{lemma:cond_mean_dev} by setting $V_i =
    \vec{w_i}$, $a_i = 1$, and $b_i = 1$. We
    then define $X, Y, \tilde{Y}$, and $\mathcal{E}$ as in Lemma
    \ref{lemma:cond_mean_dev}. 
    First, note that 
    \begin{align*}
        \left\|\widehat{\bm{M}} - \mathbb{E}[\widehat{\bm{M}}]\right\|_2 
        &\leq 
        \overbrace{\left\|
        \frac{1}{n}\sum_{l=1}^n \max(0, y^{(l)})^2 \cdot \vec{x}^{(l)} {\vec{x}^{(l)}}^\top
        - \mathbb{E}\left[
            \tilde{Y}^2 \cdot XX^T
            \big\vert 
            \mathcal{E}
        \right]
        \right\|_2}^{\delta_1}  \\
        &\qquad + \overbrace{\left\|
        \mathbb{E}\left[ \tilde{Y}^2 \cdot XX^T \big\vert \mathcal{E} \right]
        - \mathbb{E}\left[ \tilde{Y}^2 \cdot XX^T\right]
        \right\|_2}^{\delta_2}.
    \end{align*}
    We begin by bounding the second term using Lemma \ref{lemma:cond_mean_dev}
    with $(a, b) = (1, 1)$
    (to adapt the proof to $B > 1$, we set $a = B$ and $\tau_1 = \tau_1/B$):
    \begin{align*}
        \delta_2 \lesssim kB^2 \lr{\tau_1^3 + \tau_2^2} \lr{e^{-\tau_1^2/2} + e^{-\tau_2^2/2}}.
    \end{align*}
    To bound the first term, we introduce the random variables $X_+ = X |
    \mathcal{E}$ and $\tilde{Y}_+ = \tilde{Y} | \mathcal{E}$ and let
    $\vec{x}^{(l)}_+$ and $y^{(l)}_+$ be samples from the corresponding
    distributions. Let $\mathcal{E}_n$ represent the event 
    \begin{align*}
    \mathcal{E}_n = \bigcap_{i \in [k]} \bigcap_{l \in [n]} \left\{
        \abs*{\vec{w_i}^\top \vec{x}^{(l)}} \leq \tau_1,\ 
        |Z_i| \leq \tau_2
    \right\}
    \end{align*}
    We can then decompose the first term as
    \begin{align*}
        \mathbb{P}(\delta_1 \geq t) 
        &\leq \mathbb{P}\lr{\left\|
        \frac{1}{n}\sum_{l=1}^n \max(0, y^{(l)}_+)^2 \cdot \vec{x}^{(l)}_+ {\vec{x}^{(l)}_+}^\top
        - \mathbb{E}\left[
            \tilde{Y}^2 \cdot XX^T
            \big\vert 
            \mathcal{E}
        \right]
        \right\|_2 \geq t} + \mathbb{P}\lr{\mathcal{E}^c_n}.
    \end{align*}
    By definition of $X_+$ and $Y_+$, we have that $\max(0, y^{(l)}_+) \cdot
    \vec{x}^{(l)}_+$ is a sub-Gaussian random vector with Orlicz norm $O(\tau_1
    + \tau_2)$. By Lemma \ref{lem:subgaussian_conc}, for any $t \in (0, (\tau_1 +
    \tau_2)^2)$ and for $n \geq C_1 ((\tau_1+\tau_2)^2/t)^2d$,
    \begin{align*}
        \mathbb{P}\lr{\left\|
        \frac{1}{n}\sum_{l=1}^n \max(0, y^{(l)}_+)^2 \cdot \vec{x}^{(l)}_+ {\vec{x}^{(l)}_+}^\top
        - \mathbb{E}\left[
            \tilde{Y}^2 \cdot XX^T
            \big\vert 
            \mathcal{E}
        \right]
        \right\|_2 \geq t} \leq \exp\lr{-C_2 n t^2 / (\tau_1 + \tau_2)^4}
    \end{align*}
    A union bound over $i \in [k]$ and $l \in [n]$ along with a Gaussian tail
    bound yields that 
    \begin{align*}
        \mathbb{P}(\delta_1 \geq t) \leq \exp\lr{-C_2 n t^2 / (\tau_1 + \tau_2)^4} + kn\lr{\exp(-\tau_1^2/2) + \exp(-\tau_2^2/2)}.
    \end{align*}
    Now, we proceed identically to the proof of (28) in
    \citep{yi2016solving}: set $\tau_1 = \tau_2 = C\sqrt{\log(kn)}$ for
    large enough $C$
    and $t = \Theta(\frac{\log(kn)}{\sqrt{n}}
    \sqrt{\max(\log(\frac{2}{\delta}), d)}$, so that $\delta_2 \lesssim
    \frac{1}{n}$ and $\mathbb{P}(\delta_1 \geq t) \leq \delta$, completing the
    proof.
\end{proof}

\end{document}